\documentclass{article}
\usepackage[utf8]{inputenc}
\usepackage{authblk}
\usepackage{setspace}
\usepackage[margin=1in]{geometry}
\usepackage{graphicx}
\graphicspath{ {./figures/} }
\usepackage{subcaption}
\usepackage{amsmath}
\usepackage{mathrsfs}
\usepackage{lineno}
\usepackage{comment}
\usepackage{amsthm}
\newtheorem{theorem}{Theorem}
\usepackage{amsfonts} 
\usepackage{mathtools} 
\usepackage{multirow} 
\usepackage{url}
\usepackage{hyperref}
\usepackage{xcolor}
\usepackage{float}

\hypersetup{
    colorlinks,
    linkcolor={red!50!black},
    citecolor={blue!50!black},
    urlcolor={blue!80!black}
}
\usepackage[natbib=true,backend=biber,sorting=none,style=ieee, citestyle=numeric-comp]{biblatex}

\title{{\bf{A Study of Quantitative Correlations Between
Crucial Bio-markers and the Optimal Drug Regimen of Type-I
Lepra Reaction}}}

\author[1]{CH Ramanjaneyulu}
\author[2]{Dinesh Nayak}
\author[3,4,*]{D K K Vamsi}

\affil[1, 2, 3]{ \ Department of Mathematics and Computer Science, Sri Sathya Sai Institute of Higher Learning, India.}
\affil[4]{ \ Centre for Excellence in Mathematical Biology, Sri Sathya Sai Institute of Higher Learning, India.}
\affil[1]{First Author. Email: chramanjaneyulu@sssihl.edu.in}
\affil[2]{Second author. Email: dineshnayak@sssihl.edu.in}
\affil[*]{Corresponding author. Email: dkkvamsi@sssihl.edu.in}

\date{}

\begin{document}

\maketitle

\begin{abstract} {{
\noindent Leprosy (Hansen's) is a disease caused by Mycobacterium leprae. This disease slowly leads to occurrence of leprae reactions which mainly damage peripheral nervous system which cause loss of organs.
We can prevent occurring leprae reactions by monitoring the bio-markers involved in it. Motivated by these observations in this research work we do a exhaustive study dealing with the  quantitative correlations between crucial bio-markers and the  Multi Drug Thearphy (MDT) used in treating the type I lepra reaction. We frame and study a complex 11 compartment model dealing with the  the concentrations of plasma $c_1(t)$ and effective drug action $c_2(t)$, susceptible schwann cells $S(t)$, infected schwann cells $I(t)$, bacterial load $B(t)$, and five cytokines pivotal in Type-1 Lepra reaction: IFN-$\gamma$, TNF-$\alpha$, IL-$10$, IL-$12$, IL-$15$, and IL-$17$. We explore exhaustively and establish the quantitative correlations with respect to the  optimal drug dosage of the MDT drugs such as rifampin, clofazimine \& dapsone and the crucial bio-markers involved in type I lepra reaction. We conclude this work by reitrating the fact that the optimal drug dosage of the MDT drugs found through these optimal control studies and the dosage prescribed as per  WHO guidelines are almost the same. }}
\end{abstract}

{ \bf {keywords:} } Hansen’s disease, Type I lepra reaction, bio-markers, multidrug therapy, Newton’s gradient
method, optimal drug regimen \\



\section{Introduction} \label{Intro}
Leprosy, one of the oldest diseases has remained a neglected tropical disease long since. It is caused primarily by slow-growing bacterium Mycobacterium leprae (M. leprae). This bacterium primarily affects Schwann cells, leading to skin damage and impacting the peripheral nervous system, as well as affecting the eyes and mucosa of the upper respiratory tract. According to the World Health Organization (WHO), over 200,000 new cases of leprosy are reported annually in approximately 120 countries \cite{web3}. In 2022, India alone recorded about 103,819 new cases \cite{web2}. Leprosy is transmitted via droplets from the nose and mouth during close and frequent contact with untreated cases. Leprosy can progress to a chronic phase known as Lepra reaction, resulting in permanent disabilities and organ loss. Early detection of the disease through monitoring key changes in biomarker levels is crucial for preventing these consequences. \\

The modeling of leprosy began in the 1970s with simple compartmental models, such as the susceptible-infectious-recovered framework \cite{lechat1974epidemetric}. Subsequently, more complex models were developed to assess the effectiveness of long-term control and elimination strategies, including mass drug administration, contact tracing, and vaccination \cite{lechat1985simulation}. Several investigations focusing on the population-level dynamics of the disease are discussed in studies \cite{blok2015mathematical, giraldo2018multibacillary}. Additionally, the work \cite{ghosh2021mathematical} deal with the cellular dynamics within the host. \\

The alterations in the chemical and metabolic properties of the cytosolic environment within host cells due to the presence of M. leprae were first elucidated by Rudolf Virchow (1821–1902) in the late nineteenth century \cite{virchow1865krankhaften}. Subsequent clinical studies have delineated the pathways of cytokine responses, leading to the identification of two main types of Lepra reactions. Type-1 Lepra reactions are associated with cellular immune responses, while Type-2 reactions are linked to humoral immune responses \cite{luo2021host, bilik2019leprosy}. Both pathways involve crucial biomarkers/cytokines such as $IFN-\gamma$, $TNF-\alpha$, $IL-10$, $IL-12$, $IL-15$, and $IL-17$ \cite{oliveira2005cytokines}. Numerous biochemical studies have investigated the pathogenesis of lepra reaction \cite{ojo2022mycobacterium}, as well as the growth of the bacteria and its chemical consequences \cite{oliveira2005cytokines}.  \\

Motivated by these observations, the authors have done a comprehensive studies dealing with the qualitative  correlations
between crucial bio-markers and the Multi Drug Thearphy (MDT) used in treating the type I lepra reaction. More details of the same can be found in the references \cite{nayak2023study,nayak2023comprehensive}. \\

In the present work we explore and do a exhaustive study of quantitative  correlations between crucial bio-markers and the Multi Drug Thearphy (MDT) used in treating the type I lepra reaction. Since these quantitative studies involves concentration levels of the biomarkers  and dosages of the drugs, a novel model has been developed and the corresponding dynamics and findings have been dealt in this work. \\

The organization of this paper is as follows. In the next section we formulate and describe the  single dosage model and  in section 3 we frame the corresponding optimal control problem and discuss the existence of optimal control followed by the numerical studies in section 4. In section 5 we do the detailed optimal control studies incorporating second drug dosage. We end this work with the discussions and conclusions in section 6.

\section{Mathematical model formulation} \label{sec2}

\subsection{ Single Dosage Model Formulation} \label{sec2a}

Based on the discussion earlier and the  clinical literature and medical guidelines \cite{maymone2020leprosy}  for drug regimen for Lepra reaction 1, as mentioned in tables \ref{tab:treatment1} and \ref{tab:treatment2} below,  we consider a model incorporating  11 compartments, which deal with the concentrations of plasma $c_1(t)$ and effective drug action $c_2(t)$, susceptible schwann cells $S(t)$, infected schwann cells $I(t)$, bacterial load $B(t)$, and five cytokines pivotal in Type-1 Lepra reaction: IFN-$\gamma$, TNF-$\alpha$, IL-$10$, IL-$12$, IL-$15$, and IL-$17$. We analyze the concentration dynamics of these cytokines in Type-1 Lepra reaction by capturing their dynamics in our model compartments. We incorporate compartment $c_1,c_2$ in similar lines to \cite{Iliadis2000Optimizing}.\\

\begin{table}[ht!]
    \centering
    \begin{tabular}{|c|c|c|c|c|}
    \hline
        \textbf{Drugs} & \textbf{Frequency} & \textbf{Dosage 15 years \& above} & \textbf{Dosage 10-14 years} & \textbf{Dosage below 10 years } \\  
        \hline
         Rifampicin & monthly & 600 mg & 450mg & 300mg  \\
         \hline
         Clofazimine & monthly &  300 mg &  150 mg & 100mg \\ 
         \hline
         Dapsone & Daily & 100 mg & 50 mg & 25mg \\
         \hline 
    \end{tabular}
    \caption{ Leprosy Treatment for Multibacillary (MB) Type Leprosy with RFT (Release from Treatment) Criteria: Completion of 12 Monthly Pulses in 18 Consecutive Months }
    \label{tab:treatment1}
\end{table}

\begin{table}[ht!]
    \centering
    \begin{tabular}{|c|c|c|c|c|}
    \hline
        \textbf{Drugs} & \textbf{Frequency} & \textbf{Dosage 15 years \& above} & \textbf{Dosage 10-14 years} & \textbf{Dosage below 10 years } \\  
        \hline
         Rifampicin & monthly & 600 mg & 450mg & 300mg  \\
         \hline
         Dapsone & Daily & 100 mg & 50 mg & 25mg \\
         \hline 
    \end{tabular}
    \caption{ Leprosy Treatment for Paucibacillary (PB) Type Leprosy with RFT (Release from Treatment) Criteria: Completion of 6 Monthly Pulses in 9 Consecutive Months }
    \label{tab:treatment2}
\end{table}

\newpage

The 2018 WHO guidelines advocate for a Multi Drug Therapy (MDT) regimen for leprosy comprising three drugs: Rifampin, Dapsone, and Clofazimine \cite{maymone2020leprosy,tripathi2013essentials}. The influence of each of these drugs and their mathematical representation as control variables are incorporated as follows:

 $$U=\Big\{D_{i}(t) \ \big| \ D_{i}(t)\in[0,D_{i}max],  1\leq i \leq 3,  t\in[0,T]\Big\}$$ \\

\begin{align}
     \frac{dc_{1}}{dt}  \ &=  \ \frac{D_{1}(t) + D_{2}(t) + D_{3}(t)}{V_{1}} \ - \big( k_{12}  + k_{1}\big)c_{1}  \label{sec2equ1} \\
      \frac{dc_{2}}{dt} \ & = \ k_{12}\frac{V_{1}}{V_{2}} c_{1} \ - k_{2}c_{2}  \label{sec2equ2} \\
    \frac{dS}{dt} \ & = \ \omega \ - \beta S B  - \gamma S  - \mu_{1} S - (\mu_{d_1} + \mu_{d_2} +\mu_{d_3}) c_{1}(t-\tau_{d})S\label{sec2equ3} \\
	  \frac{dI}{dt} \ &= \ \beta SB \ - \delta I  - \mu_{1} I - \eta( k_{d_1} + k_{d_2} + k_{d_3} ) \cdot (c_{2} - C_{min})  \cdot  H\left[(c_{2} - C_{min})\right] \cdot I \label{sec2equ4}\\ 
   	\frac{dB}{dt} \ &= \ \alpha I   \ - y B - \mu_{2} B  -  (k_{d_1} + k_{d_2} + k_{d_3})\cdot (c_{2} - C_{min} )  \cdot  H\left[(c_{2} - C_{min})\right] \cdot B\label{sec2equ5}\\
     \frac{dI_{\gamma}}{dt} \ &= \ \alpha_{I_{\gamma}} I   \ - \left[\delta_{T_{\alpha}}^{I_{\gamma}}T_{\alpha} + \delta_{I_{12}}^{I_{\gamma}}I_{12} +\delta_{I_{15}}^{I_{\gamma}}I_{15} +\delta_{I_{17}}^{I_{\gamma}}I_{17}\right]I - \mu_{I_{\gamma}}\big(I_{\gamma} - Q_{I_{\gamma}}\big) \label{sec2equ6}\\
     \frac{dT_{\alpha}}{dt} \ &= \ \beta_{T_{\alpha}}I_{\gamma} I   \  - \mu_{T_{\alpha}}\big(T_{\alpha} - Q_{T_{\alpha}}\big) \label{sec2equ7}\\
     \frac{dI_{10}}{dt} \ &= \ \alpha_{I_{10}} I   \ - \delta_{I_{\gamma}}^{I_{10}}I_{\gamma} - \mu_{I_{10}}\big(I_{10} - Q_{I_{10}}\big) \label{sec2equ8}\\
     \frac{dI_{12}}{dt} \ &= \ \beta_{I_{12}} I_{\gamma} I \ -  \mu_{I_{12}}\big(I_{12} - Q_{I_{12}}\big) \label{sec2equ9}\\
     \frac{dI_{15}}{dt} \ &= \  \beta_{I_{15}} I_{\gamma}I \ -  \mu_{I_{15}}\big(I_{15} - Q_{I_{15}}\big) \label{sec2equ10}\\
     \frac{dI_{17}}{dt} \ &= \  \beta_{I_{17}} I_{\gamma}I \ -  \mu_{I_{17}}\big(I_{17} - Q_{I_{17}}\big) \label{sec2equ11}
\end{align}

The biological meaning of all symbols involved in the above system of differential equations  (\ref{sec2equ1}) - (\ref{sec2equ11}) is described in tables \ref{param_tab1} and \ref{param_tab2}. \newpage

\begin{table}[bht!]
    \centering
    \begin{tabular}{|c|c|}
        \hline
        \textbf{Symbols} & \textbf{Biological Meaning} \\  
        \hline
      		$c_1$ & Concentration of plasma compartment\\
      		\hline
      		$c_2$ & Concentration of site of action compartment\\
      		\hline
      		$S$ & Susceptible schwann cells\\
      		\hline
      		$I$ & Infected schwann cells \\
      		\hline
      		$B$ & Bacterial load \\
      		\hline
      		$I_{\gamma}$ & Concentration of IFN-$\gamma$ \\
      		\hline
             $T_{\alpha}$ & Concentration of TNF-$\alpha$ \\
      		\hline
             $I_{10}$ & Concentration of IL-10 \\
      		\hline
             $I_{12}$ & Concentration of IL-12 \\
      		\hline
             $I_{15}$ & Concentration of IL-15 \\
             \hline
             $I_{17}$ & Concentration of IL-17 \\
             \hline        
      		$D_1$ & Amount of rifampin drug introduced  \\
      		\hline   
                $D_2$ & Amount of dapsone drug introduced  \\
      		\hline   
                $D_3$ & Amount of clofazimine drug introduced  \\
      		\hline   
      		$V_1$ & Volume of plasma compartment  \\	
            \hline
      		$k_{12}$ &Exchange rate of drugs from plasma to site of action\\
      		\hline
      		$k_1$ & Rate of elimination of drugs from plasma compartment \\
      		\hline
      		$V_{2} $ & Volume of plasma action compartment \\
      		\hline
      		$k_{2}$ & Rate of elimination of drugs from action compartment \\
      		\hline
      		$\omega$ & Natural birth rate of the susceptible cells \\
      	    \hline
      		$\tau$& Delay time \\
      		\hline
             $\beta$ & Rate at which schwann cells are infected\\
             \hline
      		$\gamma$& Death rate of the susceptible cells due to cytokines \\
      		\hline
      		$\mu_1$& Natural death rate of schwann cells and infected schwann cells  \\ 
      		\hline
      		$\delta$ & Death rate of infected schwann cells due to cytokines   \\ 
      		\hline
            $\tau_{d}$ & Delay due to toxicity of the drug \\ 
      		\hline
             $\mu_{d}$ & Delayed toxicity of drug concentrations   \\ 
      		\hline
        $\mu_{d_1}$ & Delayed toxicity of rifampin drug concentration   \\ 
      		\hline
        $\mu_{d_2}$ & Delayed toxicity of dapsone drug concentration    \\ 
      		\hline
        $\mu_{d_3}$ &   Delayed toxicity of clofazimine drug concentration  \\ 
      		\hline
                    $\eta$ & Coefficient ratio of bacteria and infected cell  \\ 
      		\hline
      		$H$ & Heaviside step function  \\ 
        \hline
    \end{tabular}
    \caption{Description of variables and parameters present in the system  of ODE's (\ref{sec2equ1}) - (\ref{sec2equ11})}
    \label{param_tab1}
\end{table}

\begin{table}[bht!]
    \centering
    \begin{tabular}{|c|c|}
        \hline
        \textbf{Symbols} & \textbf{Biological Meaning} \\  
        \hline
            $\alpha$ & Burst rate of infected schwann cells realising the bacteria\\
      		\hline
      		$y$ & Rates at which M. Leprae is removed by cytokines\\
      		\hline
      		$\mu_{2}$ & Natural death rate of M. Leprae\\
      		\hline
      		$\alpha_{I_{\gamma}}$ & Production rate of IFN-$\gamma$\\
      		\hline
      		$\delta_{T_{\alpha}}^{I_{\gamma}}$ & Inhibition of IFN-$\gamma$ due to TNF-$\alpha$\\
      		\hline 
             $\delta_{I_{12}}^{I_{\gamma}}$ &Inhibition of IFN-$\gamma$ due to IL-12 \\
      		\hline
             $\delta_{I_{15}}^{I_{\gamma}}$ & Inhibition of  IFN-$\gamma$ due to IL-15 \\
      		\hline
             $\delta_{I_{17}}^{I_{\gamma}}$ & Inhibition of IFN-$\gamma$ due to IL-17 \\
      		\hline
             $\mu_{I_{\gamma}}$ & Decay rate of IFN-$\gamma$ \\
             \hline
             $\beta_{T_{\alpha}}$ & Production rate of TNF-$\alpha$ \\
             \hline
      		$\mu_{T_{\alpha}}$ & Decay rate of TNF-$\alpha$ \\
      		\hline
      		$\alpha_{I_{10}}$ & Production rate of IL-10\\
      		\hline
      		$\delta_{I_{\gamma}}^{I_{10}}$ & Inhibition IL-10 of due to IFN-$\gamma$\\
      		\hline
      		$\mu_{I_{10}}$ & Decay rate of IL-10\\
      		\hline
      		$\beta_{I_{12}}$ & Production rate of IL-12\\
      		\hline
      		$\mu_{I_{12}}$ & Decay rate of IL-12\\
      		\hline
             $\beta_{I_{15}}$ & Production rate of IL-15\\
      		\hline
      		$\mu_{I_{15}}$ & Decay rate of IL-15\\     	\hline
            $\beta_{I_{17}}$ & Production rate of IL-17\\
      		\hline       
      		$\mu_{I_{17}}$ & Decay rate of IL-17\\
      		\hline   
               $Q_{I_{\gamma}}$ & Quantity of IFN-$\gamma$before infection \\
      		\hline
             $Q_{T_{\alpha}}$ &Quantity of TNF-$\alpha$ before infection \\
      		\hline
             $Q_{I_{10}}$ &Quantity of IL-10 before infection \\
      		\hline
             $Q_{I_{12}}$ &Quantity of IL-12 before infection\\
      		\hline            
             $Q_{I_{15}}$ & Quantity of IL-15 before infection \\
      		\hline
             $Q_{I_{17}}$ & Quantity of IL-17 before infection\\
        \hline
    \end{tabular}
    \caption{Description of variables and parameters present in the system of ODE's (\ref{sec2equ1}) - (\ref{sec2equ11})}
    \label{param_tab2}
\end{table}

\subsection{ Single Dosage Model Description}

We now give a brief overview of each compartment in the model.\\

\textbf{$\mathbf{c_1(t)}$ compartment:} \ In equation (\ref{sec2equ1}), the first term  $\frac{D_{1}(t) + D_{2}(t) + D_{3}(t)}{V_{1}}$, represents the administration of the drugs, which is then divided by the volume of the plasma to yield the drug concentration in the plasma. The term $-k_{12}c_1$ accounts for the transfer of drug concentration from the plasma to the site of drug action compartment, while $-k_{1}c_{1}$ represents the elimination of drug concentration directly from the plasma.\\
\\
\textbf{$\mathbf{c_2(t)}$ compartment:} \ In equation (\ref{sec2equ2}), the first term accounts for the transferred drug concentration from the plasma compartment into it. The second term, $-k_2c_2$, represents the elimination of the drug from this compartment.\\
\\
\textbf{$\mathbf{S(t)}$ compartment:} \ In equation (\ref{sec2equ3}), the first term corresponds to the natural birth rate of susceptible Schwann cells. The subsequent term accounts for the reduction in the number of susceptible cells S(t) at a rate $\beta$ due to infection by the bacteria, following the law of mass action. The parameter $\gamma$ represents the death of susceptible cells due to the cytokines response, while $\mu_1$ represents the natural death rate of susceptible cells. Lastly, the final term illustrates the death of schwann cells due to the drugs present in the host body's plasma, with a delay $\tau_d$.\\
\\
\textbf{$\mathbf{I(t)}$ compartment:} \ The growth of infected cells is represented by the term $\beta S B$ in equation (\ref{sec2equ4}). These cells decrease due to the cytokines response at a rate $\delta$, and also experience natural death at a rate $\mu_1$. The final term represents the decay of infected cells due to $c_2$. Within this specific term, $H$ denotes the Heaviside step function \cite{web1}, and $\eta$ is the coefficient ratio of bacteria to infected cells. This ratio signifies that the death of one infected cell will eliminate all bacteria present within it.\\

\textbf{$\mathbf{B(t)}$ compartment:} \ The bacterial load increases indirectly due to an increase in $I(t)$, as the burst of more cells with bacteria leads to increased replication. This rate, denoted by $\alpha$, is accounted for in the first term of equation (\ref{sec2equ5}). $y$ represents the rate of clearance of $B(t)$ due to cytokines responses, while $\mu_2$ is the natural death rate of bacteria. The last term of this equation represents the reduction of the bacterial load due to the concentration $c_2$, owing to its bactericidal and bacteriostatic properties.\\

The compartments  $\mathbf{I_{\gamma}(t)},\mathbf{T_{\alpha}(t)},\mathbf{I_{10}(t)},\mathbf{I_{12}(t)},\mathbf{I_{15}(t)},\mathbf{I_{17}(t)}$ are influenced similarly as in \cite{nayak2023study}.

\section{Optimal Control Studies for Single Dosage Model} \label{sec3}

Based on above model we define the \\

{\textit{Cost functional}}:
\begin{equation}
  \begin{aligned}
   \mathcal{J}_{min}\big(I,B,D_{1},D_{2},D_{3}\big) \ &= \int_{0}^{T} \Big(I(t) + B(t)+P\cdot D^{2}_{1}(t) + Q\cdot D^{2}_{2}(t) +  R\cdot D_{3}^{2}(t)  \Big) dt 
  \label{costf}
  \end{aligned}
\end{equation} \\
{\textit{Lagrangian}} of the {\textit{cost functional}} is given by 
\begin{equation}
  \begin{aligned}
   L\big(I,B,D_{1},D_{2},D_{3}\big) \ &= \ I(t) + B(t)+P\cdot D^{2}_{1}(t) + Q\cdot D^{2}_{2}(t) +  R\cdot D_{3}^{2}(t)  
  \label{opti}
  \end{aligned}
\end{equation} 

Admissible solution set given as follows
\[\Omega = \Big\{(I,B,D_{1},D_{2},D_{3}) \big| I,B \text{ are satisfying system of O.D.E's (\ref{sec2equ1}) - (\ref{sec2equ11})}, D_{i}(t) \in [0,D_{i\max}], 1\leq i \leq 3, t\in[0,T] \Big\}\] \\

{\large{\bf{Existence of Optimal Control} }}\\

In this section, we prove the existence of solution to the optimal control to the system (\ref{sec2equ1}) - (\ref{costf}) by using theorem 2.2 in \cite{Boyarsky1976existence}.
\begin{theorem} \label{th1}
For the control system (\ref{sec2equ1}) - (\ref{sec2equ11}) with admissible control set $U$ and the cost functional  (\ref{costf}) there exist an 3-tuple of optimal control $\big(D_{1}^{*},D_{2}^{*},D_{3}^{*}\big) \in  U.$ Further more optimal state variables of system (\ref{sec2equ1}) - (\ref{sec2equ11}), which minimize the cost functional are given as
\[\mathcal{J}_{min}\big(I^{*},B^{*},D_{1}^{*},D_{2}^{*},D_{3}^{*}\big) \ =  \underset{(D_{1},D_{2},D_{3})\in U}{\min} \mathcal{J}_{min}\big(I,B,D_{1},D_{2},D_{3}\big).\]
\end{theorem} 
\begin{proof}
    Let us consider  
    $\frac{dc_{1}}{dt}  = f_{1}(t,x,D)$ , 
    $\frac{dc_{2}}{dt}  = f_{2}(t,x,D)$ ,
    $\frac{dS}{dt}  = f_{3}(t,x,D)$ ,
    $\frac{dI}{dt}  = f_{4}(t,x,D)$ ,\\
    $\frac{dB}{dt}  = f_{5}(t,x,D)$ ,
    $\frac{dI_{\gamma}}{dt}  = f_{6}(t,x,D)$ ,
    $\frac{dT_{\alpha}}{dt}  = f_{7}(t,x,D)$ ,
    $ \frac{dI_{10}}{dt}  = f_{8}(t,x,D)$ , 
    $ \frac{dI_{12}}{dt}  = f_{9}(t,x,D)$ ,\\
    $ \frac{dI_{15}}{dt}  = f_{10}(t,x,D)$ ,
    $ \frac{dI_{17}}{dt}  = f_{11}(t,x,D)$. of the control system (\ref{sec2equ1}) - (\ref{sec2equ11})
    where $x \in X $ denotes state variables$\big(c_{1},c_{2},S,I,B,I_{\gamma},T_{\alpha},I_{10},I_{12},I_{15},I_{17}\big)$, and $D \in U$ denotes control variables $\big(D_{1},D_{2},D_{3}\big)$.\\
     Take $f = \big(f_{1},f_{2},f_{3},f_{4},f_{5},f_{6},f_{7},f_{8},f_{9},f_{10},f_{11}\big)$ ,we have $X \in \mathbb{R}^{11}$ and 
    \[ f : [0, T] \times X \times U \rightarrow \mathbb{R}^{11} \]
     since $f_{j}'s$ are polynomials so $f$ is a continuous function with respect to $t$ and $x$ for each $D_{i}'s$.\\ where $1\leq i \leq 3$,$1\leq j \leq 11.$\\
     Now we try to show that $(\mathbf{F1})$ to $(\mathbf{F3})$ conditions in theorem 2.2 of \cite{Boyarsky1976existence} holds true for all $f_j$'s. \\
     
      \textbf{F1:} Here each of the $f_j$'s has a continuous and bounded partial derivative implying that  $f$ is  Lipschitz's continuous.\\
      
       \textbf{F2:} let define $g_{1}(D_1,D_2,D_3) = \frac{D_{1}(t) + D_{2}(t) + D_{3}(t)}{V_{1}} $
       which is bounded on $U$. \\
       so
       
\begin{equation}
  \begin{aligned}
  \frac{ f_1(t,x,D^{(1)})- f_1(t,x,D^{(2)})}{\big[g_1(D^{(1)})-g_1(D^{(2)})\big]}
 &= \frac{\big[D^{(1)}_{1}+D^{(1)}_{2}+D^{(1)}_{3}-D^{(2)}_{1}-D^{(2)}_{2}-D^{(2)}_{3}\big]}{\big[D^{(1)}_{1}+D^{(1)}_{2}+D^{(1)}_{3}-D^{(2)}_{1}-D^{(2)}_{2}-D^{(2)}_{3}\big]}\\
   & \leq n = F_1(t,x)\\
 f_1(t,x,D^{(1)})- f_1(t,x,D^{(2)}) &\leq F_1(t,x)\cdot \big[g_1(D^{(1)})-g_1(D^{(2)})\big]
  \end{aligned}
\end{equation}
where n is a real number and $n \geq 1$ .
Since $U$ is compact and $g_1$ is continuous by result that if a function is continuous and domain is compact then the range of function is compact so $g_1(U)$ will be compact.\\
Also since the function $g_1$ is linear so it range $g_1(U)$ will be convex. Since $U$ is non-negative set so $g_{1}^{-1}$ will be non-negative.\\
For satisfy this condition for remaining $f_{j}$'s we use corollary 2.1 of \cite{Boyarsky1976existence} which show that we can use condition \textbf{F4} instead of \textbf{F2}.Hence considering $g_{2}(D_1,D_2,D_3) = 0$, which is bounded measurable function and $F_2(t,x) = 1$ we have relation 
\[ f_2(t,x,D^{(1)})- f_2(t,x,D^{(2)})= 0 = 1\cdot0 = F_2(t,x)\cdot \big[g_2(D^{(1)}-D^{(2)})\big]\]
Similarly taking $F_j(t,x) = 1$ and $g_{j}(D_1,D_2,D_3) = 0$ for j = 3,4,5,6,7,8,9,10,11 we have relations
\[ f_j(t,x,D^{(1)})- f_j(t,x,D^{(2)}) = F_j(t,x)\cdot \big[g_j(D^{(1)}-D^{(2)})\big]\]
Therefore $f$ satisfied condition \textbf{F2}.

 \textbf{F3:} Since $c_{1},c_{2},S,I,B,I_{\gamma},T_{\alpha},I_{10},I_{12},I_{15},I_{17}$ are bounded on $[0,T]$ \\ 
 hence $F_j(\bullet,x^{u}(\bullet)) \in  \mathscr{L}_{1}$ for $1\leq j \leq 11$.
Now we have to show that the running cost function\\
$ C : [0, T] \times X \times U \rightarrow \mathbb{R} $ as
\[C(t,x,D) = I(t) + B(t)+P\cdot D^{2}_{1}(t) + Q\cdot D^{2}_{2}(t) +  R\cdot D_{3}^{2}(t)  \]
satisfy the conditions \textbf{C1-C5} of theorem 2.2 of \cite{Boyarsky1976existence}.\\
\textbf{C1:} Since $C(t,\cdot,\cdot)$ is sum of all continuous functions of t so it is a continuous function for all $t \in [0,T]$.\\
\textbf{C2:}  $I,B$  and all $D_{i}$'s are bounded implying that $C(\cdot,x,D)$ is bounded and hence measurable for each $x\in X$ and $D_{i} \in U$.\\
\textbf{C3:} Consider $\Psi(t) = \kappa$ such that $ \kappa = \min \{I(0),B(0)\} $ then $\Psi$ will bounded such that for all $t\in [0,T]$, $x \in X$ and $D_{i}\in U,$ we have 
\[C(t,x,D)\geq \Psi(t)\]
\textbf{C4:} Since $C(t,x,D)$ is sum of the function  which are convex in $U$ for each fixed $(t,x)\in [0,T]\times X $\\  therefore $C(t,x,D)$ follows the same.\\
\textbf{C5:} Using similar type of argument, we can easily show that for each fixed $(t,x)\in [0,T]\times X $, $C(t,x,D)$ is a monotonically increasing function. \\
Hence by using  theorem 2.2 of \cite{Boyarsky1976existence} for the system (\ref{sec2equ1}) - (\ref{sec2equ11}) we have showed that it satisfies hypothesis. this implies that an Optimal Control and Optimal State variables for the system (\ref{sec2equ1}) - (\ref{sec2equ11}) exists and minimizes the cost functional. 
\end{proof}



\section{Numerical Studies with Reference to Single Dosage Model} \label{sec4}
\subsection{Theory} \label{sec4a} 
In this section, we elaborate on the methodology employed to assess the optimal control problem  (\ref{sec2equ1}) - (\ref{costf}) described earlier. The evaluation of optimal control variables and state variables is conducted using the forward-backward sweep method \cite{mcasey2012convergence} in conjunction with the Pontryagin maximum principle \cite{liberzon2011calculus}.

The Hamiltonian of the control system (\ref{sec2equ1}) - (\ref{sec2equ11}) is given by

\begin{align}
\begin{split}
    \mathcal{H}(I,B,D_{1},D_{2},D_{3},\lambda) & = I(t) + B(t)+P D^{2}_{1}(t) + Q D^{2}_{2}(t) +  R D_{3}^{2}(t) + \lambda_{1}\frac{dc_1}{dt} + \lambda_{2}\frac{dc_2}{dt} + \lambda_{3}\frac{dS}{dt} + \lambda_{4}\frac{dI}{dt} + \lambda_{5}\frac{dB}{dt} \\
    \ & + \lambda_{6}\frac{dI_{\gamma}}{dt} + \lambda_{7}\frac{dT_{\alpha}}{dt} + \lambda_{8}\frac{dI_{10}}{dt} + \lambda_{9}\frac{dI_{12}}{dt} +\lambda_{10}\frac{dI_{15}}{dt} +\lambda_{11}\frac{dI_{17}}{dt}
\end{split}
\end{align}
where $\lambda = (\lambda_{1},\lambda_{2},\lambda_{3},\lambda_{4},\lambda_{5},\lambda_{6},\lambda_{7},\lambda_{8},\lambda_{9},\lambda_{10},\lambda_{11})$ is co-state variable or adjoint vector.
Since we have $D^{*} = (D_{1}^{*},D_{2}^{*},D_{3}^{*})$ and $X^{*} = (x_{1},x_{2},x_{3},x_{4},x_{5},x_{6},x_{7},x_{8},x_{9},x_{10},x_{11})$  as optimal control and state variable respectively,  using Pontryagin maximum principle there exists an optimal co-state variable say $\lambda^{*}$\\
which satisfy the canonical equation
\begin{equation}\label{co}
       \frac{d\lambda_{j}}{dt} = -\frac{\partial \mathcal{H}(X^{*},D^{*},\lambda^{*})}{\partial x_{j}}
\end{equation}
where $j = 1,2,3,...,11$ \\

Using above equation we get system of ODE's for co-state variables as follows\\
\begin{align}
    \frac{d\lambda_{1}}{dt} &= (k_{12} + k_1)\lambda_{1} - k_{12}\left(\frac{V_1}{V_2}\right)\lambda_{2} + (\mu_{d_1} + \mu_{d_2} + \mu_{d_3})S\lambda_{3} \\
   \frac{d\lambda_{2}}{dt} &= k_2\lambda_{2} + \eta(k_{d_1}+ k_{d_2} + k_{d_3}) H[(c_2-c_{min})] \cdot I\lambda_{4} + (k_{d_1} + k_{d_2} + k_{d_3})  H[(c_2-c_{min})] \cdot B\lambda_{5} \\
    \frac{d\lambda_{3}}{dt} &= \big(\beta B + \mu_1 + \gamma + (\mu_{d_1} + \mu_{d_2} + \mu_{d_3})c_1(t - \tau_d)\big)\lambda_{3} - \beta B \lambda_{4} \\
    \begin{split}
    \frac{d\lambda_{4}}{dt} &= \big(\delta + \mu_{1} + (\eta k_{d_1} + \eta k_{d_2} + \eta k_{d_3})(c_2 - C_{min}) H[(c_2 - C_{min})])\lambda_{4} - \alpha \lambda_5 - \alpha_{I_{\gamma}}\lambda_{6} \\
    \ & + \left(\delta_{T_{\alpha}}^{I_{\gamma}}T_{\alpha} + \delta_{I_{12}}^{I_{\gamma}}I_{12} +\delta_{I_{15}}^{I_{\gamma}}I_{15} +\delta_{I_{17}}^{I_{\gamma}}I_{17}\right) \lambda_{6} \beta_{T_{\alpha}}I_{\gamma} \lambda_{7} - \alpha_{I_{10}} \lambda_{8} - \beta_{I_{12}} I_{\gamma} \lambda_{9} - \beta_{I_{15}} I_{\gamma} \lambda_{10} - \beta_{I_{17}} I_{\gamma} \lambda_{11} - 1 
    \end{split} \\
   \frac{d\lambda_{5}}{dt} &= \beta S \lambda_{3} - \beta S \lambda_{4} + \left(y + \mu_{2} + (k_{d_1} + k_{d_2} + k_{d_3})(c_2 - C_{min}) H[(c_2 - C_{min})]\right)\lambda_{5} - 1 \\
    \frac{d\lambda_{6}}{dt} &= \mu_{I_{\gamma}}\lambda_{6} - \beta_{T_\alpha}I\lambda_{7} + \delta_{I_{10}}^{I_\gamma} I\lambda_8 - \beta_{I_{12}}I\lambda_{9} - \beta_{I_{15}}I\lambda_{10} - \beta_{I_{17}}I\lambda_{11} \\
    \frac{d\lambda_{7}}{dt} &= \delta_{T_{\alpha}}^{I_\gamma} I\lambda_6 + \mu_{T_{\alpha}}\lambda_{7} \\
    \frac{d\lambda_{8}}{dt} &= \mu_{I_{10}} \lambda_{8} \\
    \frac{d\lambda_{9}}{dt} &= \delta_{I_{12}}^{I_\gamma} I\lambda_6  + \mu_{I_{12}}\lambda_{9}  \\
    \frac{d\lambda_{10}}{dt} &= \delta_{I_{15}}^{I_\gamma} I\lambda_6  + \mu_{I_{15}}\lambda_{10} \\
    \frac{d\lambda_{11}}{dt} &= \delta_{I_{17}}^{I_\gamma} I\lambda_6  + \mu_{I_{17}}\lambda_{11}
\end{align}
and the transversality condition $\lambda_{i}(T) = \frac{\partial \phi}{\partial x_i}\big|_{t=T} = 0 $ for all $i = 1,2,3,...,11$ where in this case, the terminal cost function, represented by $\phi$, is constantly zero. \\

Now we use $Newton's \ Gradient \ method$ from \cite{edge1976function} to obtain the optimal value of the controls. \\

For this recursive formula is employed to update the control at each step of the numerical simulation as follows

\begin{align}\label{updatecontrol1}
    D_{i}^{k+1}(t) = D_{i}^{k}(t) + \theta_{k}d_k
\end{align}
Here,  $D_{i}^{k}(t)$ represents the control value at the $k^{th}$ iteration at a given time $t$, $d_k$ signifies the direction, and $\theta_k$ denotes the step size. The direction $d_k$ can be evaluated as negative of gradient of the objective function i.e $d_k = - g_i(D_{i}^{k})$ ,where $g_i(D_{i}^{k}) = \frac{\partial \mathcal{H}}{\partial D_i}\big|_{D_{i}^{k}(t)}$ as mentioned in\cite{edge1976function}.The step size $\theta_k$ is determined at each iteration using a linear search technique aimed at minimizing the Hamiltonian,$\mathcal{H}$. Therefore (\ref{updatecontrol1}) can become as 
\begin{align}\label{updatecontrol2}
    D_{i}^{k+1}(t) = D_{i}^{k}(t) - \theta_{k}\frac{\partial \mathcal{H}}{\partial D_i}\Big|_{D_{i}^{k}(t)}
\end{align}
To implement the aforementioned approach, we need to compute the gradient for each control, denoted as $g_i(D_{i}^{k})$ , which are listed as follows
\begin{align*}
        g_1(D_{1}) = 2PD_{1}(t) + \frac{\lambda_{1}}{V_1}\\
        g_2(D_{2}) = 2QD_{2}(t) + \frac{\lambda_{2}}{V_1}\\
        g_3(D_{3}) = 2RD_{3}(t) + \frac{\lambda_{1}}{V_1} 
\end{align*}
\subsection{Numerical simulations} \label{sec4b}

In this section, we conduct numerical simulations to quantitatively examine the correlation between cytokine levels in Type-1 Lepra reaction and the drugs utilized in MDT. \\

The parameters' values utilized are sourced from diverse clinical articles, with corresponding references provided in table \ref{parameter1}.\\

Doubling time information was accessible for certain parameters like $\beta$, $\gamma$, $\alpha$ and $\delta$, allowing estimation through the following formula:
$$rate\ \% = \frac{ln(2)}{doubling \ time} \cdot 100 $$
We subsequently divide these percentage rates by 100 to derive the values of these parameters. \\

In certain instances, we calculated the average of the resulting yields from various mediums, including 7-AAD and TUNEL, as outlined in \cite{oliveira2005cytokines}. We have taken $\tau_d = 30.$ \\

Certain parameters are meticulously adjusted to meet specific hypotheses or assumptions, facilitating the numerical simulation process.\\

For these simulations, we utilize a time duration of 30 days (T = 30), and most of the parameter values are selected from table \ref{parameter1}  and other values are chosen as $$ V_1 = 1200, \ V_2 = 500, \ \omega = 20.90, \ \beta = 0.000030, \ \mu_1 = 0.00018, \ \gamma = 0.01795, \ \alpha = 0.2, \  y = 0.03, \  \alpha_{I_{10}(t)} = 0.5282.$$\\

Initially, we solved the system numerically without any drug intervention. All numerical computations were performed using MATLAB, and we employed the fourth-order Runge-Kutta method to solve the system of ODEs. To determine the value of $\theta$ in each iteration, we utilized MATLAB's fminsearch() function. In this context, we regard the initial values of the state variables as $c_1(0) = 0$, $c_2(0) = 0$, $S(0) = 520$, $I(0) = 250$, $B(0) = 2500$, $I_{\gamma}(0) = 50$, $T_{\alpha}(0) = 50$, $I_{10}(0) = 75$, $I_{12}(0) = 125$, $I_{15}(0) = 125$, and $I_{17}(0) = 100$ as in \cite{ghosh2021mathematical,liao2013role}.\\

Moreover, to simulate the system with controls, we employ the forward-backward sweep method, commencing with the initial control values set to $20, 100, 10$ for $D_1, D_2, D_3$ respectively and estimate the state variables forward in time. Subsequently, since the transversality conditions involve the adjoint vector's value at the end time T, we compute the adjoint vector backward in time.\\

Utilizing the state variables and adjoint vector values, we compute the control variables at each time step, which are subsequently updated in each iteration. The control update strategy involves implementing Newton's gradient method, as described by equation (\ref{updatecontrol2}). We iterate this process until the convergence criterion, as outlined in reference \cite{edge1976function}, is satisfied.\\

The weights $P$, $Q$, and $R$ in the cost function $\mathcal{J}_{min}$  are chosen for numerical simulation, with each weight set to 1.5.\\

We proceed to numerically simulate the populations of $c_1$ ,$c_2$, S, I, and B, along with cytokines levels, employing single, double, and triple control interventions of MDT for 30 days.\\

 \begin{table}[ht!]   
 \centering 
\begin{tabular}{|c|c|c|} 
\hline

\textbf{Symbols} &  \textbf{Values} & \textbf{Units} \\  

\hline 
$V_1$ & 25 \cite{Iliadis2000Optimizing}& $litr $ \\

\hline
$k_{12}$ & 0.4 \cite{Iliadis2000Optimizing} & $
day^{-1} $  \\

\hline
$k_1$& 1.6 \cite{Iliadis2000Optimizing}  & $day^{-1}$ \\

\hline
$V_2$ & 20 \cite{Iliadis2000Optimizing} & $litr$  \\

\hline
$k_2$ & 0.8 \cite{Iliadis2000Optimizing} & $day^{-1}$ \\

\hline 
$\mu_{d_1}$ & 1 table:\ref{hazard} & $day^{-1}conc^{-1}$  \\

\hline
$\mu_{d_2}$ & 3.81, table:\ref{hazard} & $day^{-1}conc^{-1}$  \\

\hline
$\mu_{d_3}$ & 7.1 table:\ref{hazard} & $day^{-1}conc^{-1}$  \\

\hline
$\eta$ & 0.01* & $dimensionless$ \\

\hline 

$k_{d_1}$ & 0.26 table:\ref{hazard} & $day^{-1} conc^{-1}$ \\

\hline 
$k_{d_2}$ & 0.99 table:\ref{hazard} & $day^{-1} conc^{-1}$ \\

\hline 
$k_{d_3}$ & 1.85 table:\ref{hazard} & $day^{-1} conc^{-1}$ \\

\hline 
$\omega$ & 0.0220 \cite{kim2017schwann}& $pg.ml^{-1}.day^{-1} $ \\

\hline
$\beta$ & 3.4400 \cite{jin2017formation} & $pg.ml^{-1}.day^{-1} $  \\

\hline
$\gamma$& 0.1795 \cite{oliveira2005cytokines}  & $day^{-1}$ \\

\hline
$\mu_{1}$ & 0.0018 \cite{oliveira2005cytokines} & $day^{-1}$  \\

\hline
$\delta$ & 0.2681 \cite{oliveira2005cytokines} & $day^{-1}$ \\

\hline
$\alpha$ & 0.0630 \cite{levy2006mouse} & $pg.ml^{-1}.day^{-1}$ \\

\hline
$y$ & $0.0003$ \cite{ghosh2021mathematical} & $day^{-1}$ \\
\hline
$\mu_{2}$ & 0.5700 \cite{international2020international}& $day^{-1}$ \\
\hline
$\alpha_{I_\gamma}$ & 0.0003 \cite{pagalay2014mathematical}& $pg.ml^{-1}.day^{-1} $ \\
\hline
$\delta^{I_{\gamma}}_{T_{\alpha}}$ &0.005540* & $pg.ml^{-1}$  \\

\hline

$\delta^{I_{\gamma}}_{I_{12}}$ & 0.009030* & $pg.ml^{-1}$ \\

\hline

$\delta^{I_{\gamma}}_{I_{15}}$ & 0.006250* & $pg.ml^{-1}$ \\

\hline

$\delta^{I_{\gamma}}_{I_{17}}$ & 0.004990*& $pg.ml^{-1}$ \\

\hline

$\mu_{I_{\gamma}}$ & 2.1600 \cite{pagalay2014mathematical} & $day^{-1} $\\
\hline

$\beta_{T_{\alpha}}$ & 0.0040 \cite{pagalay2014mathematical}&  $pg.ml^{-1}.day^{-1} $ \\
\hline
$\mu_{T_{\alpha}}$ & 1.1120 \cite{pagalay2014mathematical}  & $day^{-1}$  \\
\hline
$\alpha_{I_{10}}$ & 0.0440 \cite{liao2013role} & $pg.ml^{-1}.day^{-1} $  \\
\hline

$\delta^{I_{10}}_{I_{\gamma}}$ & 0.001460*& $pg.ml^{-1}$\\
\hline

$\mu_{I_{10}}$ & 16.000 \cite{liao2013role} & $day^{-1}$ \\
\hline

$\beta_{I_{12}}$ &0.0110 \cite{liao2013role}  & $pg.ml^{-1}.day^{-1}$\\
\hline
$\mu_{I_{12}}$ & 1.8800 \cite{pagalay2014mathematical} & $day^{-1}$ \\
\hline
$\beta_{I_{15}}$ &0.0250 \cite{su2009mathematical} & $pg.ml^{-1}.day^{-1} $   \\
\hline
$\mu_{I_{15}}$ & 2.1600 \cite{su2009mathematical} & $day^{-1}$\\
\hline
$\beta_{I_{17}}$ & 0.0290 \cite{su2009mathematical} & $pg.ml^{-1}.day^{-1} $  \\
\hline
$\mu_{I_{17}}$ & 2.3400 \cite{su2009mathematical}& $day^{-1}$\\
\hline
$Q_{I_{\gamma}}$ & 0.1000 \cite{talaei2021mathematical} & Relative concentration\\
\hline
$Q_{T_{\alpha}}$ & 0.1400 \cite{brady2016personalized}  & Relative concentration \\
\hline 
$Q_{I_{10}}$ & 0.1500 \cite{talaei2021mathematical} &  Relative concentration\\
\hline 
$Q_{I_{12}}$ &1.1100 \cite{brady2016personalized}  & Relative concentration \\
\hline 
$Q_{I_{15}}$ & 0.2000 \cite{brady2016personalized} & Relative concentration  \\
\hline 
$Q_{I_{17}}$ & 0.3170 \cite{brady2016personalized} & Relative concentration \\
\hline 
\end{tabular}
\caption{The parameter values have been compiled from clinical literature, with (*) indicating assumed values for certain parameters.}
\label{parameter1}
\end{table} 

\begin{table}[ht!]
   \centering
    \begin{tabular}{|c|c|c|}
    \hline
        \textbf{Drugs} & \textbf{Hazard ratio} & \textbf{Source} \\  
        \hline
Rifampin & 0.26   & \cite{bakker2005prevention} \\
\hline
Dapsone & 0.99 & \cite{cerqueira2021influence} \\
\hline
Clofazimine & 1.85  & \cite{cerqueira2021influence} \\
\hline
\end{tabular}
\caption{The hazard ratio associated with the drugs}
\label{hazard}
\end{table}

\newpage

\subsection{Findings} \label{sec4c}
In this section, we analyze the results from the simulations described earlier. Figures \ref{fig:Rifampin} - \ref{fig:mdt} depict the dynamics of the $S$, $I$, $B$, $I_{\gamma}$, $T_{\alpha}$, $I_{10}$, $I_{12}$, $I_{15}$, and $I_{17}$ compartments in our model (\ref{sec2equ3})–(\ref{sec2equ11}) under different drug administration scenarios for 30 days. Each panel represents a compartment and compares its dynamics with and without drug intervention. \\

Figure \ref{fig:Rifampin} depicts the dynamics under rifampin administration, while figures \ref{fig:Dapsone} and \ref{fig:Clofazimine} show the dynamics under dapsone and clofazimine administration, respectively. Average and 30th-day values of each compartment without control and with single drug intervention such as rifampin, dapsone, and clofazimine are presented in tables \ref{tab:avg_1d} and \ref{tab:last_1d} respectively. \\

In all cases of single drug intervention, such as with drugs rifampin, dapsone and clofazimine, we observe from figures \ref{fig:Rifampin}, \ref{fig:Dapsone}, and \ref{fig:Clofazimine} that the compartments susceptible cells $S(t)$, infected cells $I(t)$, and bacterial load $B(t)$, as well as for IFN-$\gamma$ ($I_{\gamma}(t)$), TNF-${\alpha}$ ($T_{\alpha}(t)$), IL-10 ($I_{10}(t)$), and IL-12 ($I_{12}(t)$), all show a decreasing trend compared to the scenario without drug intervention. Conversely, compartments IL-15 ($I_{15}(t)$) and IL-17 ($I_{17}(t)$), show an increasing trend. \\

The most significant reduction in susceptible cells is observed with dapsone, while the least reduction is observed with rifampin as a single drug intervention. Figure results for infected cells show consistency across all single drug interventions, but differences are noticeable in the values presented in tables \ref{tab:avg_1d} and \ref{tab:last_1d}. These tables indicate a decrease in infected cells, consistent with the trend observed for susceptible cells. \\

However, there is no discernible difference in the compartments bacterial load $B(t)$, IFN-$\gamma$ ($I_{\gamma}(t)$), TNF-${\alpha}$ ($T_{\alpha}(t)$), IL-10 ($I_{10}(t)$), IL-12 ($I_{12}(t)$), IL-15 ($I_{15}(t)$), and IL-17  ($I_{17}(t)$) when comparing the effects of all single drug interventions, as shown in both the figures \ref{fig:Rifampin} - \ref{fig:Clofazimine} and tables \ref{tab:avg_1d}, \ref{tab:last_1d}.\\

Figures \ref{fig:Rif&Dap}, \ref{fig:Clo&Dap} and \ref{fig:Rif&Clo} illustrate the dynamics under combinations of drugs: rifampin \& dapsone, clofazimin \& dapsone and rifampin \& clofazimine respectively. Average and 30th-day values of each compartment without control and with a two-drug combined intervention such as rifampin \& dapsone, clofazimine \& dapsone and rifampin \& clofazimine are presented in tables \ref{tab:avg_2d} and \ref{tab:last_2d} respectively. \\

In all cases of two-drug combination intervention, such as with combinations of rifampin and dapsone, dapsone and clofazimine, and rifampin and clofazimine, we observe from figures \ref{fig:Rif&Dap}, \ref{fig:Clo&Dap}, and \ref{fig:Rif&Clo} that the compartments exhibit trends similar to single drug intervention. \\

The most significant reduction in susceptible cells is observed with a combination of dapsone and clofazimine, while the least reduction is observed with a combination of rifampin and clofazimine as two-drug combined intervention. Results for infected cells show consistency across all combined two-drug interventions, but differences are noticeable in the values presented in tables \ref{tab:avg_2d} and \ref{tab:last_2d}. These tables indicate a decrease in infected cells, consistent with the trend observed for susceptible cells. \\

However, there is no discernible difference in the compartments bacterial load $B(t)$, $I_{\gamma}(t)$,  $T_{\alpha}(t)$,  $I_{10}(t)$,  $I_{12}(t),$  $I_{15}(t)$, and  $I_{17}(t)$ when comparing the effects of all combined two-drug interventions, as shown in both figures \ref{fig:Rif&Dap} - \ref{fig:Rif&Clo} and table \ref{tab:avg_2d}. However, there is a very slight increment in the compartment  $I_{\gamma}(t)$, and decrement in the compartments $T_{\alpha}(t)$,  $I_{10}(t)$, and  $I_{12}(t)$ with dapsone and clofazimine when comparing the effects of all combined two-drug interventions, as shown in table \ref{tab:last_2d}.\\

Figure \ref{fig:mdt} illustrates the dynamics under the administration of MDT drugs, comprising rifampin, clofazimine, and dapsone. Average and 30th-day values of each compartment without control and with MDT drug intervention such as rifampin, dapsone, and clofazimine are presented in table \ref{tab:Avg&last_3d}. \\

In case of MDT drugs (rifampin, clofazimine, and dapsone) intervention we observe from figure \ref{fig:mdt} that the compartments susceptible cells $S(t)$, infected cells $I(t)$, and bacterial load $B(t)$, as well as for $I_{\gamma}(t)$,  $T_{\alpha}(t)$,  $I_{10}(t)$, and  $I_{12}(t)$, show a decreasing trend compared to the scenario without drug intervention. Conversely, compartments $I_{15}(t)$, and $I_{17}(t)$, show an increasing trend.\\

Optimal drug values for individual drug administration, combination of two drugs, and MDT drug administration are presented in tables  \ref{tab:result1}, \ref{tab:result2}, and \ref{tab:result3} respectively. \\

Our findings indicate that current MDT drugs dosage for leprosy, as prescribed by physicians \cite{maymone2020leprosy}, are optimal according to our model. \\

\begin{table}[ht!]
   \centering
    \begin{tabular}{|c|c|c|c|}
    \hline
        \textbf{Single Drug} & \textbf{Dosage monthly(in mg)} & \textbf{Initial drug dosage(in mg)} & \textbf{Optimal drug dosage(in mg)}\\  
        \hline
Rifampin & 600   & 20  & 19.999 \\
\hline
Dapsone & 3000  & 100 & 100.01 \\
\hline
Clofazimine & 300  & 10  & 10.001 \\
\hline
\end{tabular}
\caption{Dosage levels for individual drug administration for 30-days}
\label{tab:result1}
\end{table}

\begin{table}[ht!]
   \centering
    \begin{tabular}{|c|c|c|c|}
    \hline
       \multirow{2}{*}{\textbf{Two Drugs}} & \multirow{2}{*}{\textbf{Monthly drug Dosage(in mg)}} & \multicolumn{2}{c|}{\textbf{Combined dosage of two drugs(in mg)}} \\
        \cline{3-4}
         & & \textbf{Initial}  & \textbf{Optimal} \\
        \hline
         Dapsone and Clofazimine & 3000 + 300 & 100, 10  & 100.01, 9.999 \\
         \hline
          Rifampin and Clofazimine & 600 + 300 & 20, 10 & 19.999, 9.999 \\
         \hline
          Rifampin and Dapsone & 600 + 3000 & 20, 100  & 20.002, 99.999 \\
         \hline
    \end{tabular}
    \caption{Dosage levels for combination of two drugs administration for 30-days}
    \label{tab:result2}
\end{table}

\begin{table}[ht!]
   \centering
    \begin{tabular}{|c|c|c|c|}
    \hline 
    \multirow{2}{*}{\textbf{Three Drugs}} & \multirow{2}{*}{\textbf{Monthly drug Dosage(in mg)}} & \multicolumn{2}{c|}{\textbf{ MDT drug dosage (in mg)}} \\
        \cline{3-4}
         & & \textbf{Initial}  & \textbf{Optimal} \\
        \hline
Rifampin, Dapsone, Clofazimine & 600+3000+300 & 20, 100, 10  & 20, 100, 10.001 \\
\hline
\end{tabular}
\caption{Dosage levels for the administration of all three drugs in MDT for 30-days}
\label{tab:result3}
\end{table}

\begin{figure}[htbp]
    \centering
    \begin{subfigure}{0.30\textwidth}
        \includegraphics[width=\textwidth]{S_r.png}
        \caption{Graph 1}
        \label{fig:graph1}
    \end{subfigure}
    \hfill
    \begin{subfigure}{0.30\textwidth}
        \includegraphics[width=\textwidth]{I_r.png}
        \caption{Graph 2}
        \label{fig:graph2}
    \end{subfigure}
    \hfill
    \begin{subfigure}{0.30\textwidth}
        \includegraphics[width=\textwidth]{B_r.png}
        \caption{Graph 3}
        \label{fig:graph3}
    \end{subfigure}
    \hfill
    \begin{subfigure}{0.30\textwidth}
        \includegraphics[width=\textwidth]{IFNg_r.png}
        \caption{Graph 4}
        \label{fig:graph4}
    \end{subfigure}
    \hfill
    \begin{subfigure}{0.30\textwidth}
        \includegraphics[width=\textwidth]{TNFa_r.png}
        \caption{Graph 5}
        \label{fig:graph5}
    \end{subfigure}
    \hfill
    \begin{subfigure}{0.30\textwidth}
        \includegraphics[width=\textwidth]{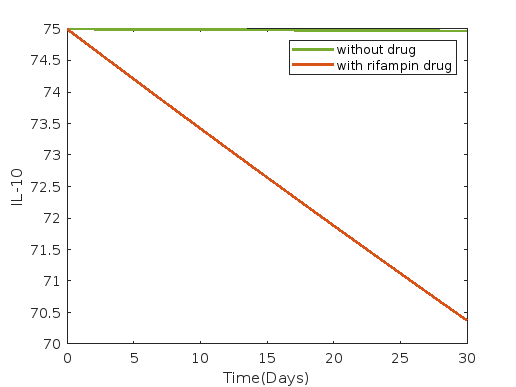}
        \caption{Graph 6}
        \label{fig:graph6}
    \end{subfigure}
    \hfill
    \begin{subfigure}{0.30\textwidth}
        \includegraphics[width=\textwidth]{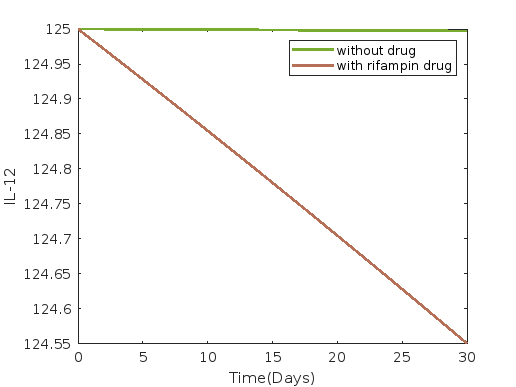}
        \caption{Graph 7}
        \label{fig:graph7}
    \end{subfigure}
    \hfill
    \begin{subfigure}{0.30\textwidth}
        \includegraphics[width=\textwidth]{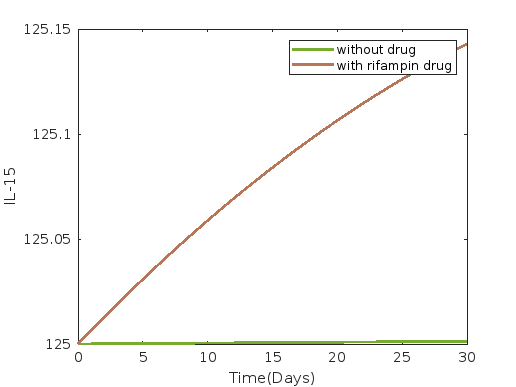}
        \caption{Graph 8}
        \label{fig:graph8}
    \end{subfigure}
    \hfill
    \begin{subfigure}{0.30\textwidth}
        \includegraphics[width=\textwidth]{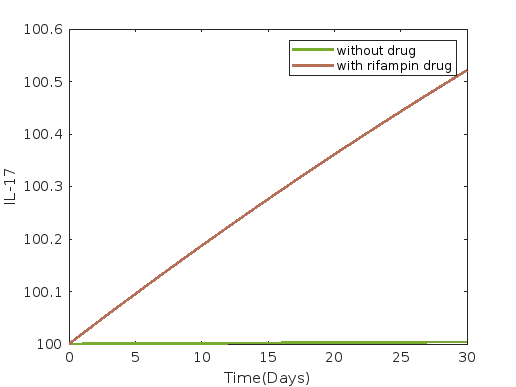}
        \caption{Graph 9}
        \label{fig:graph9}
    \end{subfigure}
    \caption{Plots depicting the influence of rifampin drug for one month  }
    \label{fig:Rifampin}
\end{figure}

\begin{figure}[htbp]
    \centering
    \begin{subfigure}{0.30\textwidth}
        \includegraphics[width=\textwidth]{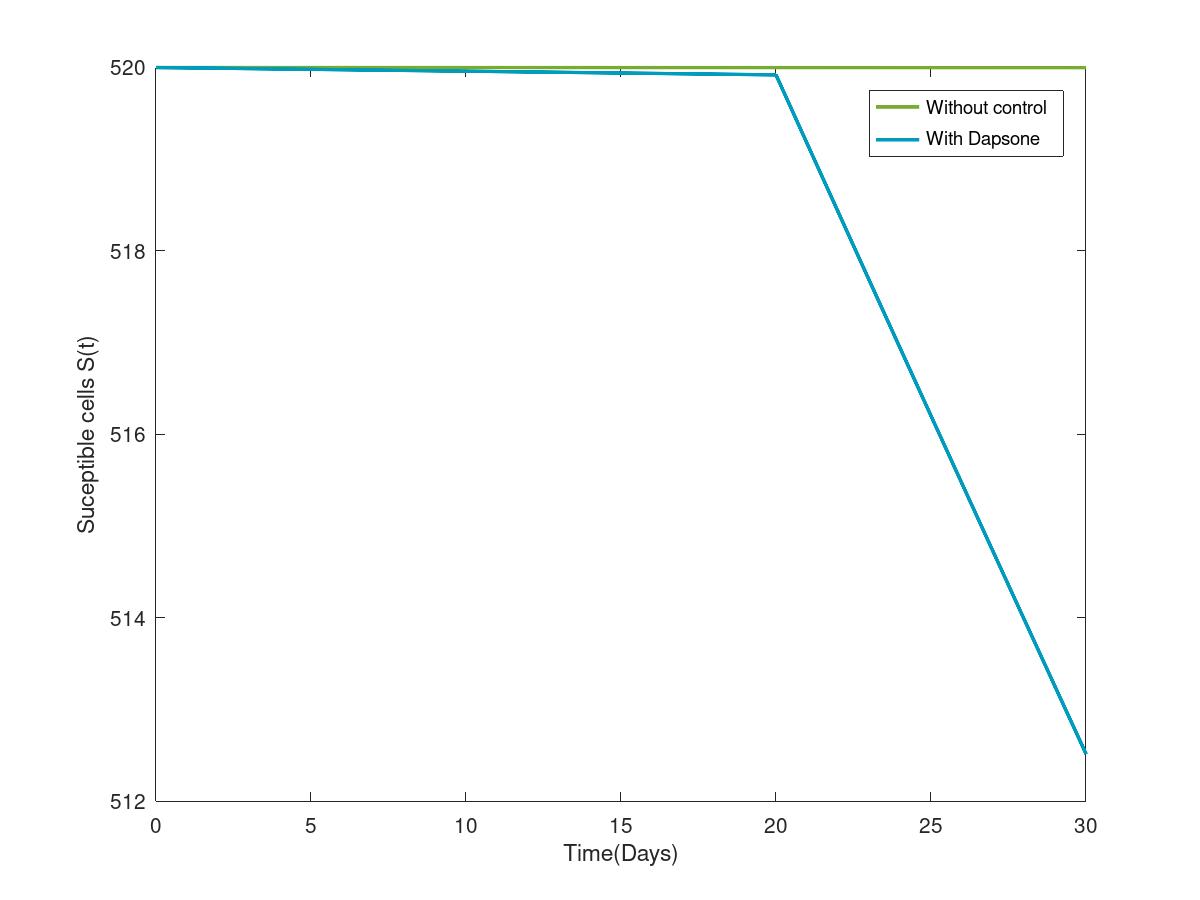}
        \caption{Graph 1}
        \label{fig:graph10}
    \end{subfigure}
    \hfill
    \begin{subfigure}{0.30\textwidth}
        \includegraphics[width=\textwidth]{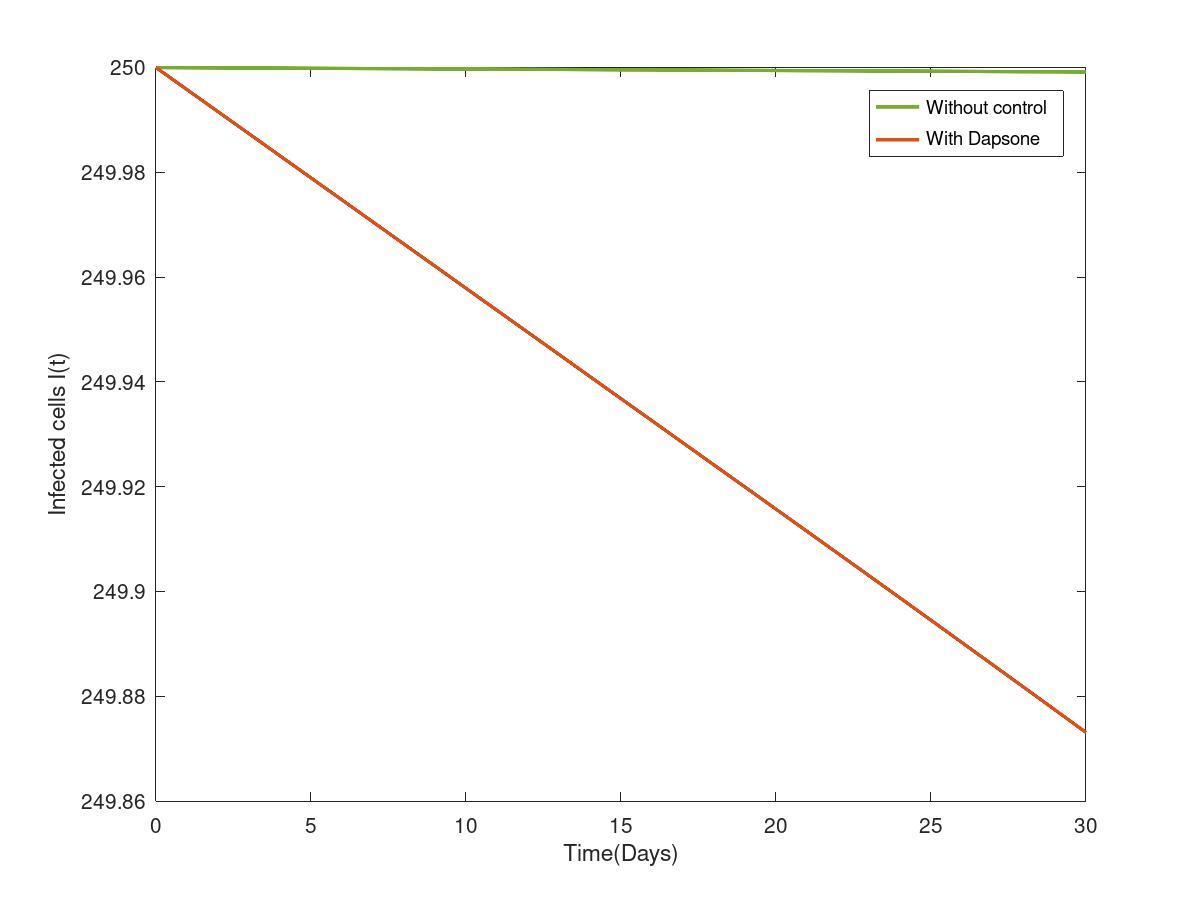}
        \caption{Graph 2}
         \label{fig:graph11}
    \end{subfigure}
    \hfill
    \begin{subfigure}{0.30\textwidth}
        \includegraphics[width=\textwidth]{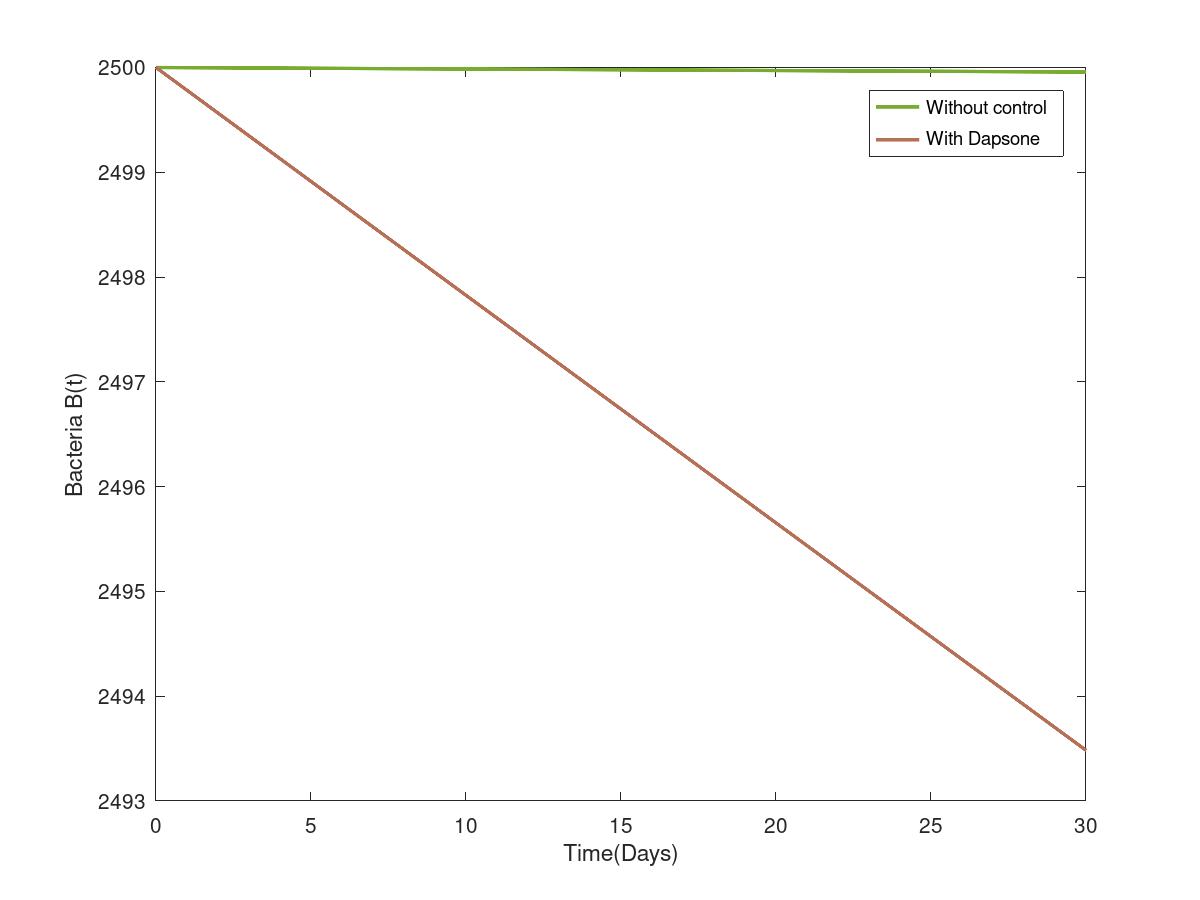}
        \caption{Graph 3}
         \label{fig:graph12}
    \end{subfigure}
    \hfill
    \begin{subfigure}{0.30\textwidth}
        \includegraphics[width=\textwidth]{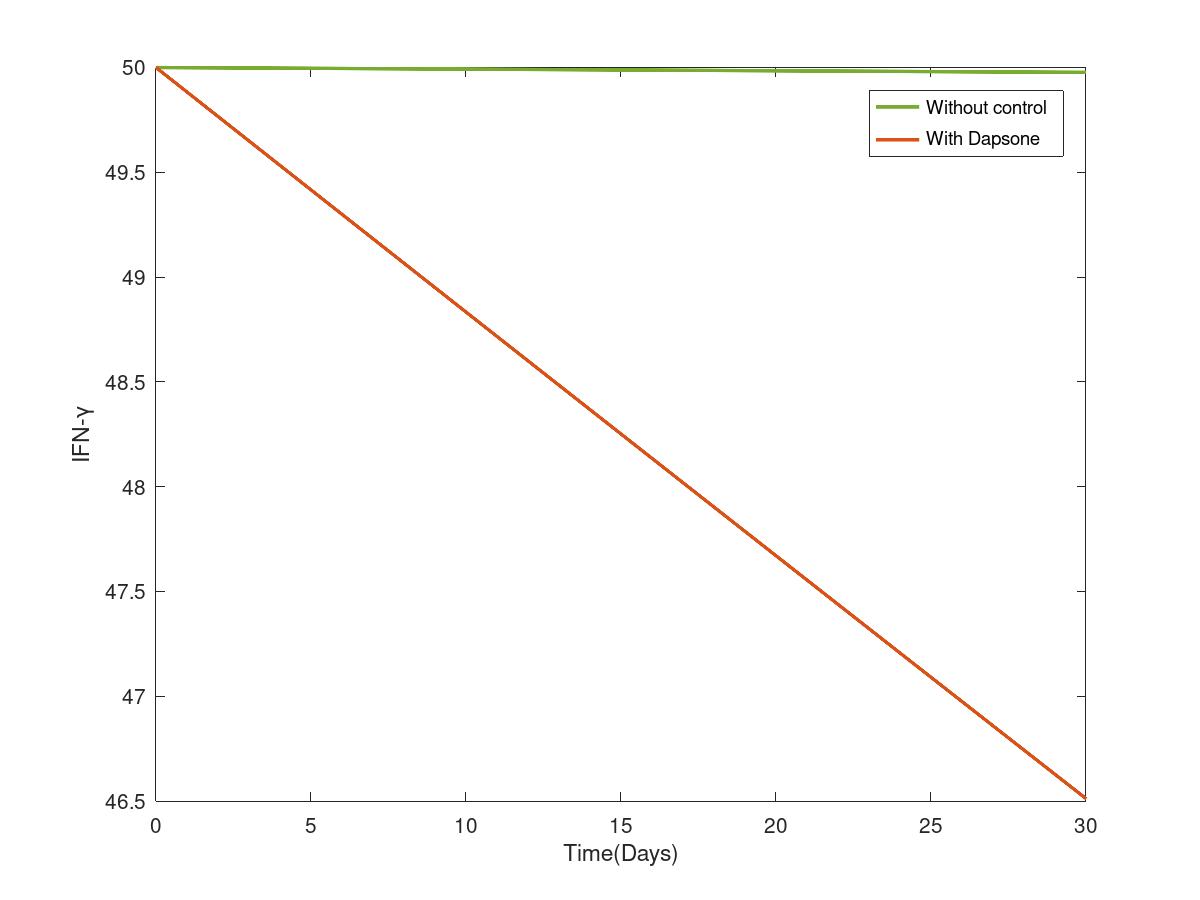}
        \caption{Graph 4}
         \label{fig:graph13}
    \end{subfigure}
    \hfill
    \begin{subfigure}{0.30\textwidth}
        \includegraphics[width=\textwidth]{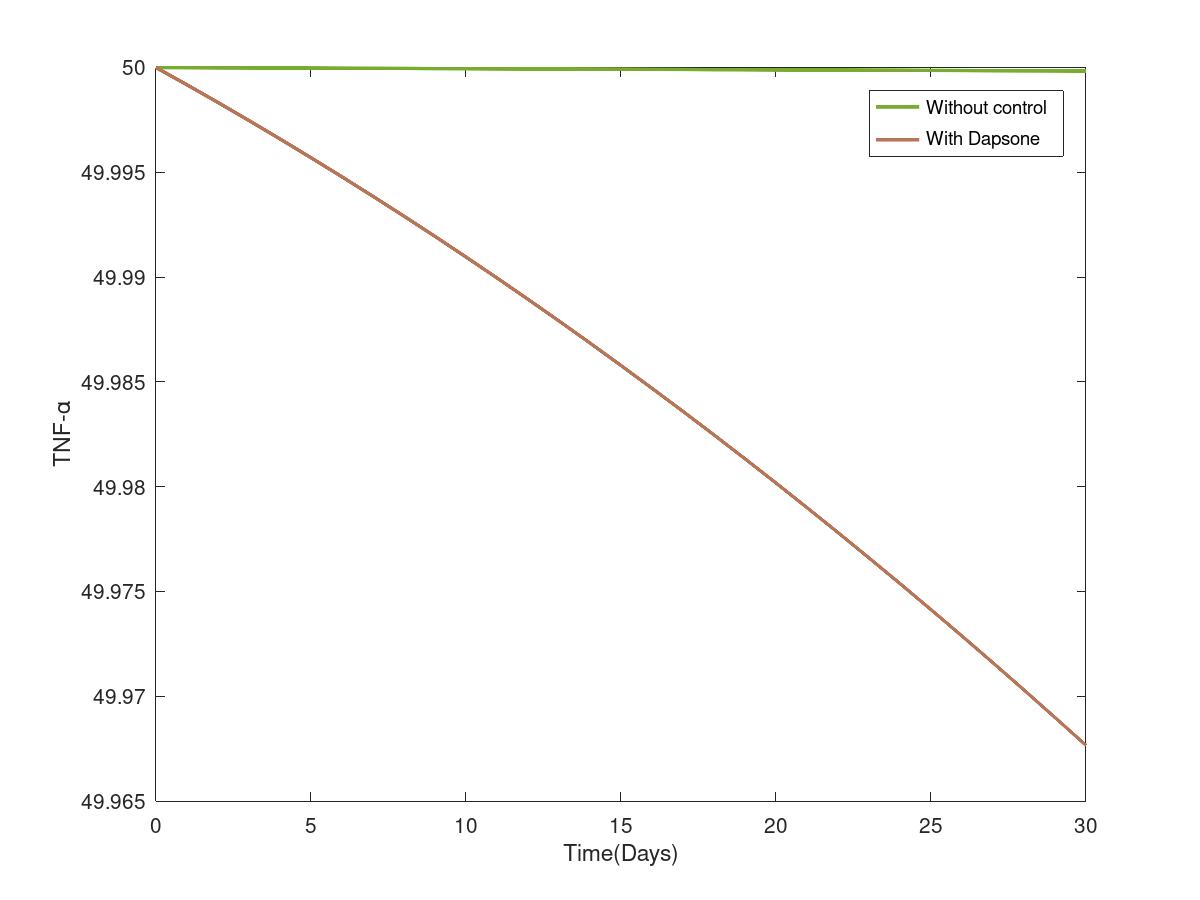}
        \caption{Graph 5}
         \label{fig:graph14}
    \end{subfigure}
    \hfill
    \begin{subfigure}{0.30\textwidth}
        \includegraphics[width=\textwidth]{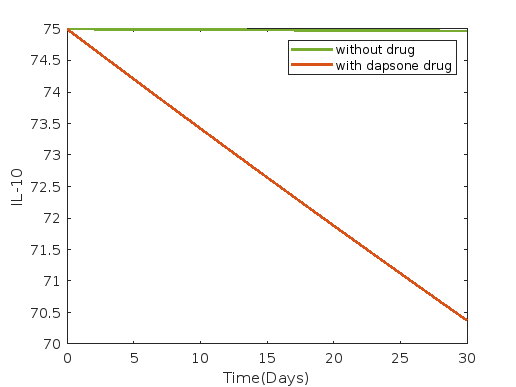}
        \caption{Graph 6}
         \label{fig:graph15}
    \end{subfigure}
    \hfill
    \begin{subfigure}{0.30\textwidth}
        \includegraphics[width=\textwidth]{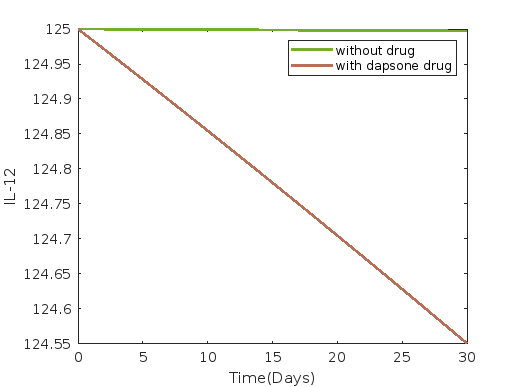}
        \caption{Graph 7}
         \label{fig:graph16}
    \end{subfigure}
    \hfill
    \begin{subfigure}{0.30\textwidth}
        \includegraphics[width=\textwidth]{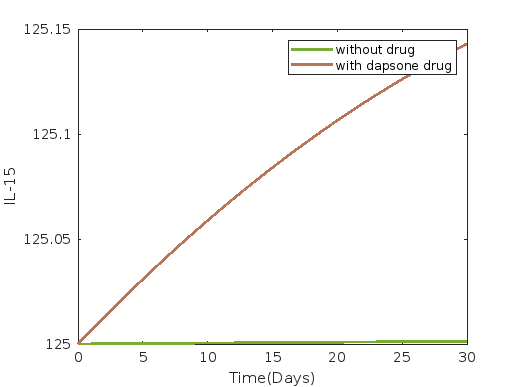}
        \caption{Graph 8}
         \label{fig:graph17}
    \end{subfigure}
    \hfill
    \begin{subfigure}{0.30\textwidth}
        \includegraphics[width=\textwidth]{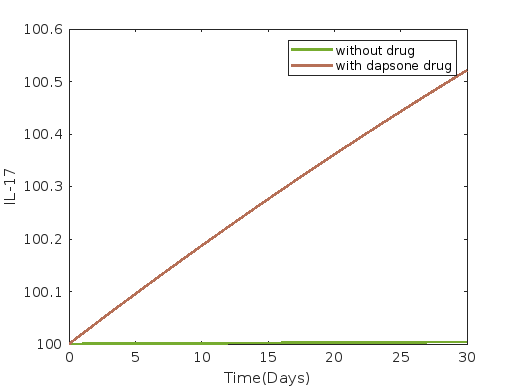}
        \caption{Graph 9}
         \label{fig:graph18}
    \end{subfigure}
    \caption{Plots depicting the influence of dapsone drug for one month  }
    \label{fig:Dapsone}
\end{figure}

\begin{figure}[htbp]
    \centering
    \begin{subfigure}{0.30\textwidth}
        \includegraphics[width=\textwidth]{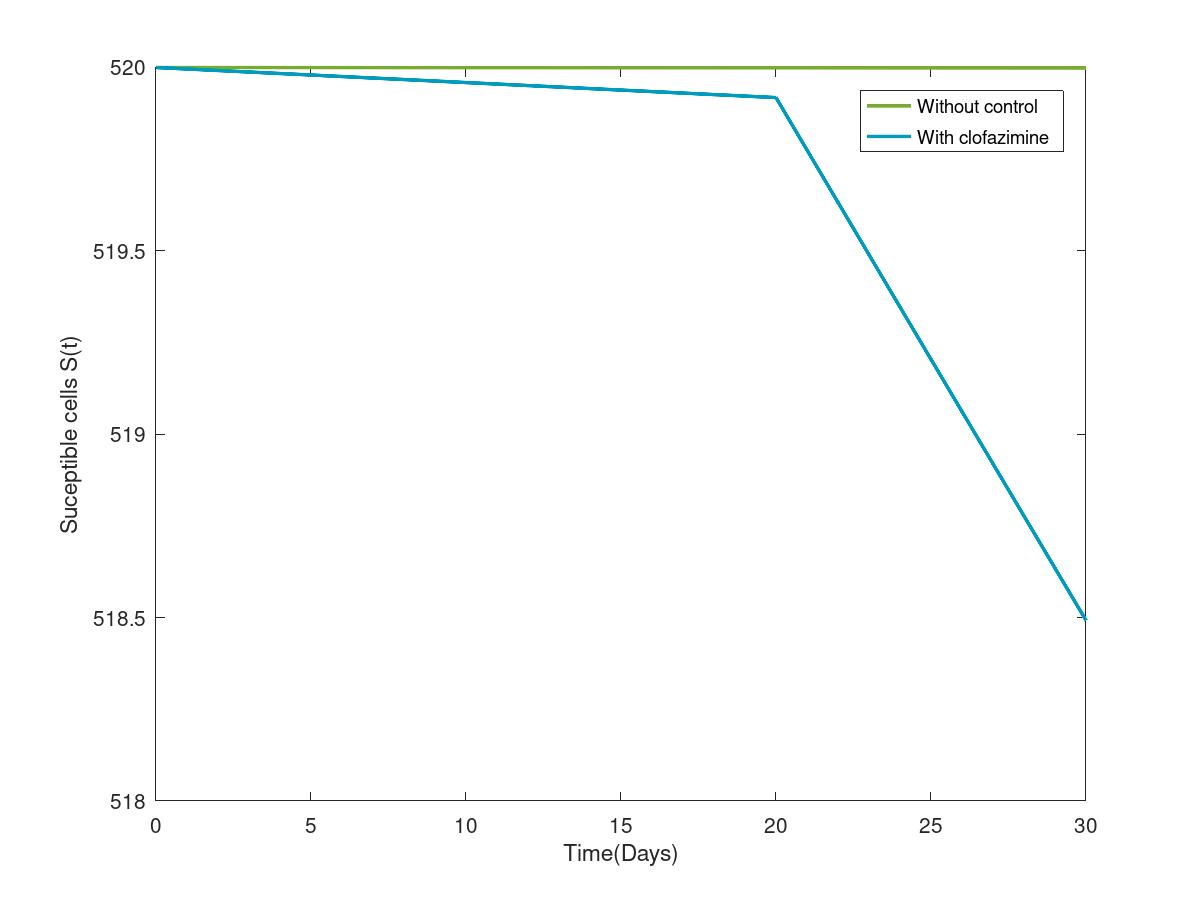}
        \caption{Graph 1}
        \label{fig:graph19}
    \end{subfigure}
    \hfill
    \begin{subfigure}{0.30\textwidth}
        \includegraphics[width=\textwidth]{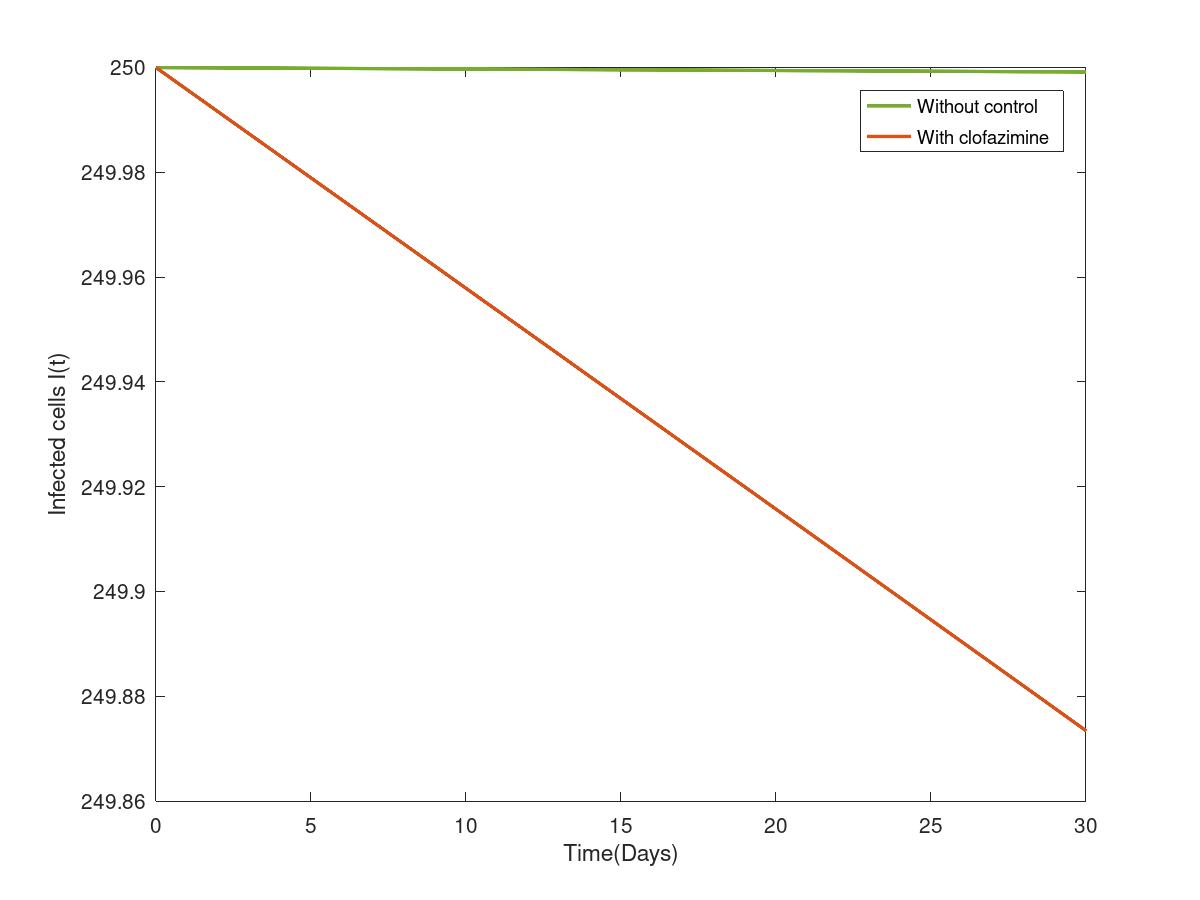}
        \caption{Graph 2}
        \label{fig:graph20}
    \end{subfigure}
    \hfill
    \begin{subfigure}{0.30\textwidth}
        \includegraphics[width=\textwidth]{B_c.png}
        \caption{Graph 3}
        \label{fig:graph21}
    \end{subfigure}
    \hfill
    \begin{subfigure}{0.30\textwidth}
        \includegraphics[width=\textwidth]{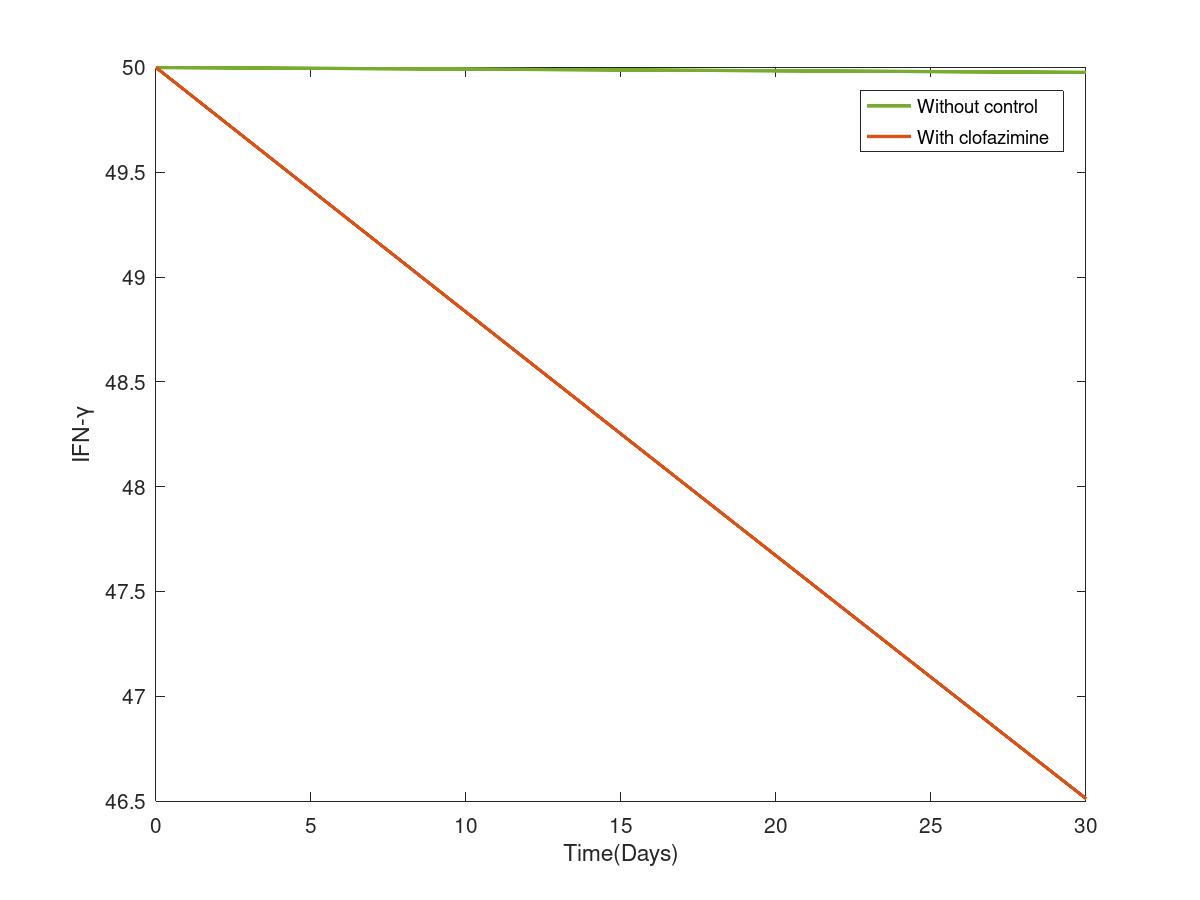}
        \caption{Graph 4}
        \label{fig:graph22}
    \end{subfigure}
    \hfill
    \begin{subfigure}{0.30\textwidth}
        \includegraphics[width=\textwidth]{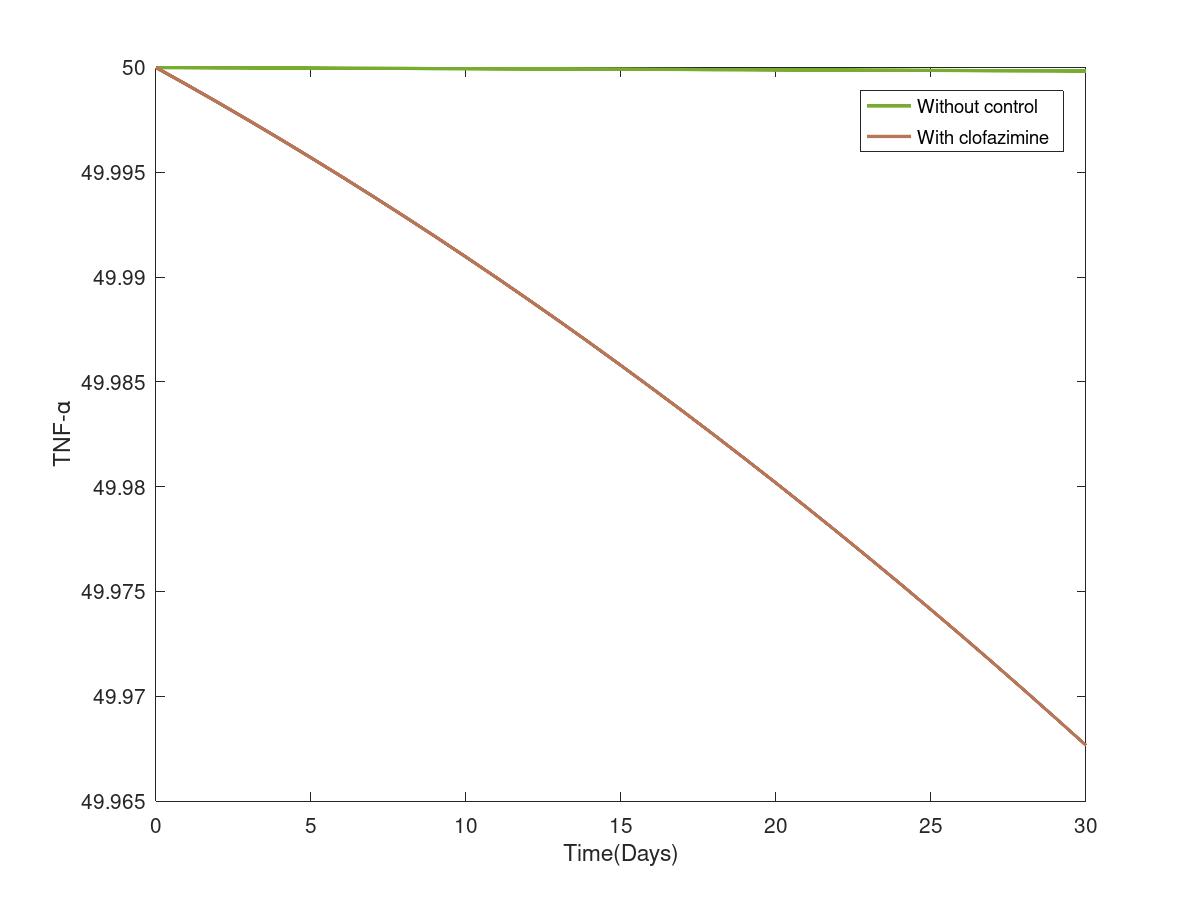}
        \caption{Graph 5}
        \label{fig:graph23}
    \end{subfigure}
    \hfill
    \begin{subfigure}{0.30\textwidth}
        \includegraphics[width=\textwidth]{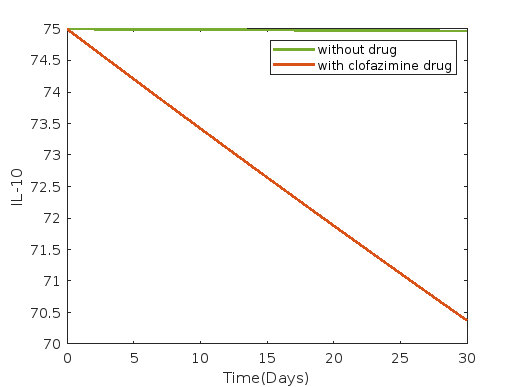}
        \caption{Graph 6}
        \label{fig:graph24}
    \end{subfigure}
    \hfill
    \begin{subfigure}{0.30\textwidth}
        \includegraphics[width=\textwidth]{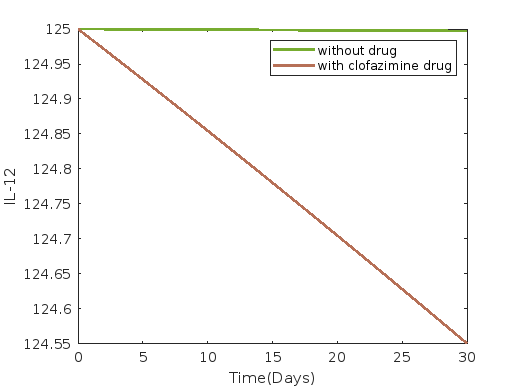}
        \caption{Graph 7}
        \label{fig:graph25}
    \end{subfigure}
    \hfill
    \begin{subfigure}{0.30\textwidth}
        \includegraphics[width=\textwidth]{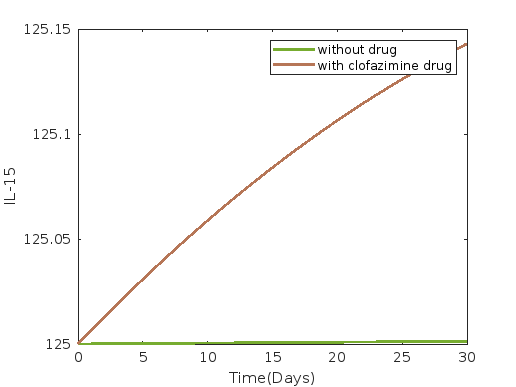}
        \caption{Graph 8}
        \label{fig:graph26}
    \end{subfigure}
    \hfill
    \begin{subfigure}{0.30\textwidth}
        \includegraphics[width=\textwidth]{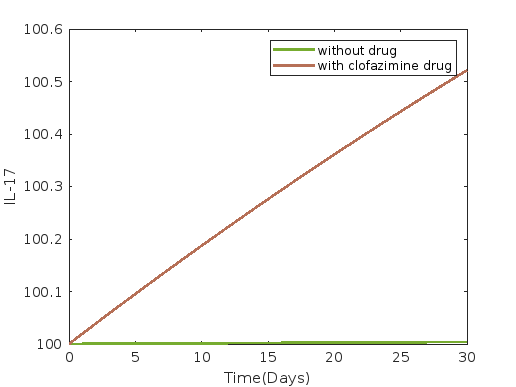}
        \caption{Graph 9}
        \label{fig:graph27}
    \end{subfigure}
    \caption{Plots depicting the influence of clofazimine drug for one month  }
    \label{fig:Clofazimine}
\end{figure}


\begin{figure}[htbp]
    \centering
    \begin{subfigure}{0.30\textwidth}
        \includegraphics[width=\textwidth]{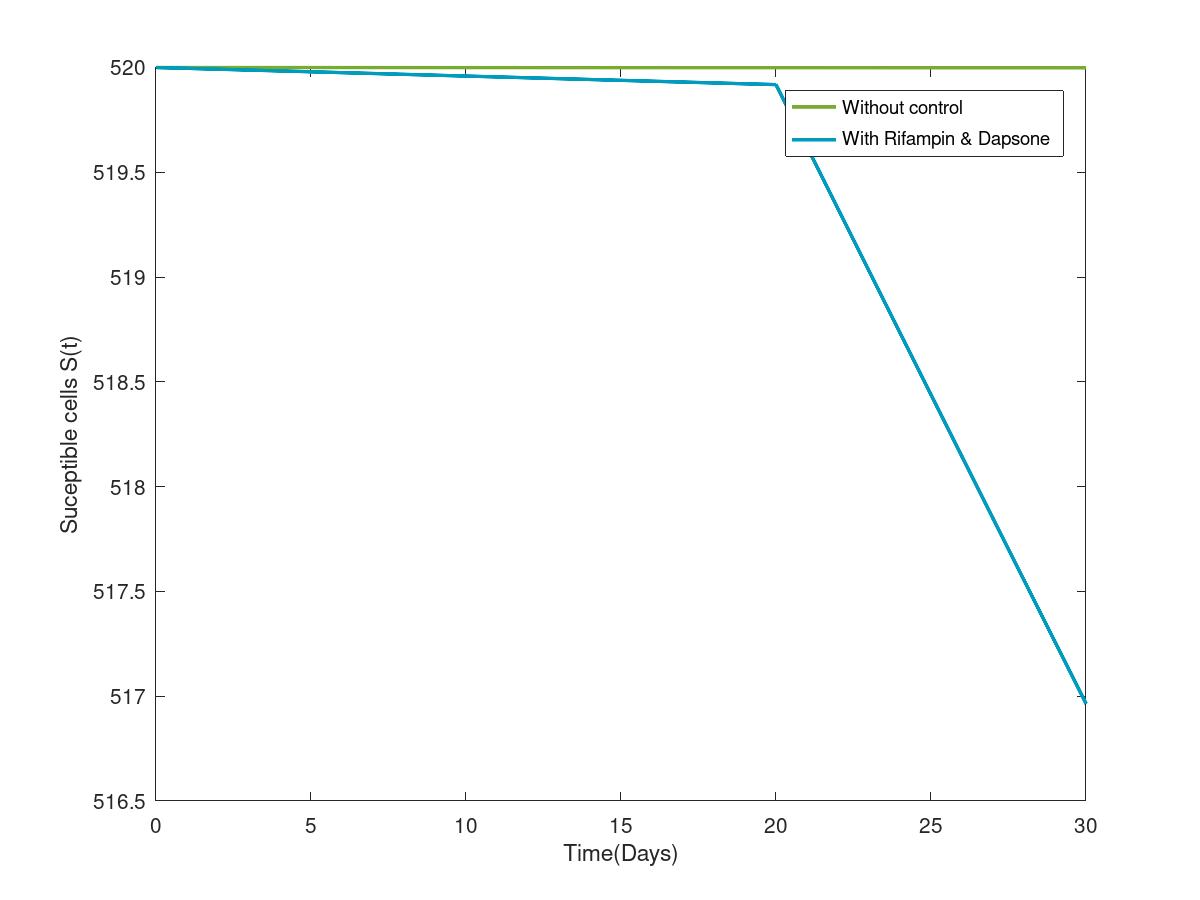}
        \caption{Graph 1}
        \label{fig:graph28}
    \end{subfigure}
    \hfill
    \begin{subfigure}{0.30\textwidth}
        \includegraphics[width=\textwidth]{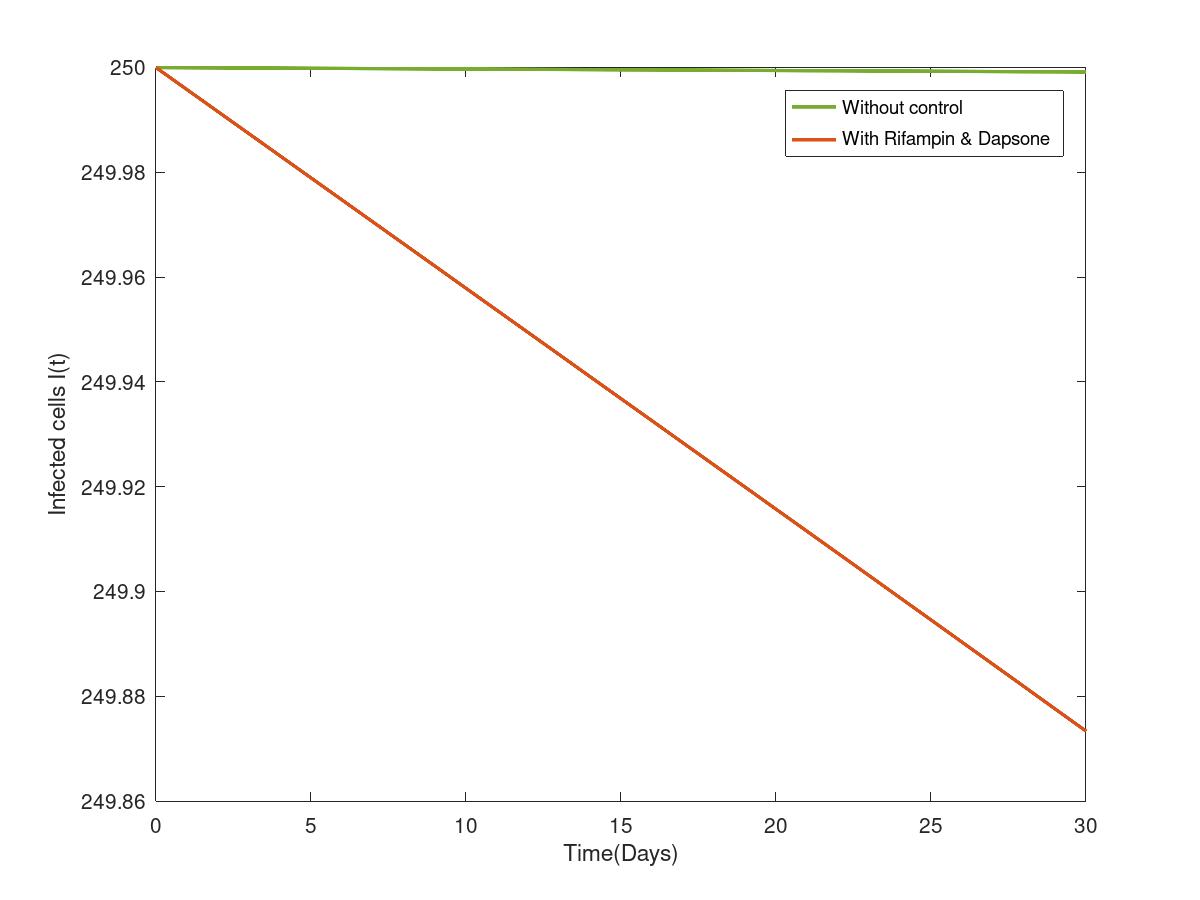}
        \caption{Graph 2}
        \label{fig:graph29}
    \end{subfigure}
    \hfill
    \begin{subfigure}{0.30\textwidth}
        \includegraphics[width=\textwidth]{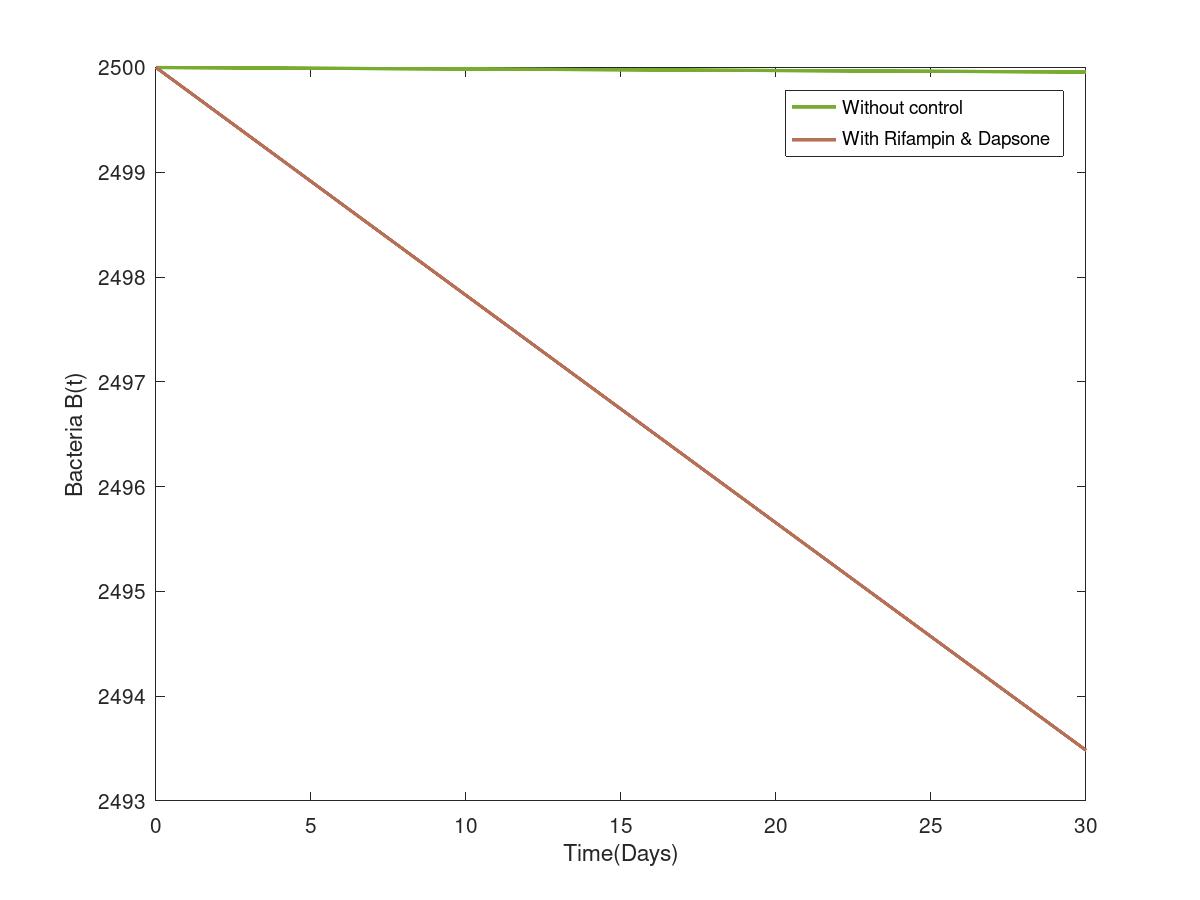}
        \caption{Graph 3}
         \label{fig:graph30}
    \end{subfigure}
    \hfill
    \begin{subfigure}{0.30\textwidth}
        \includegraphics[width=\textwidth]{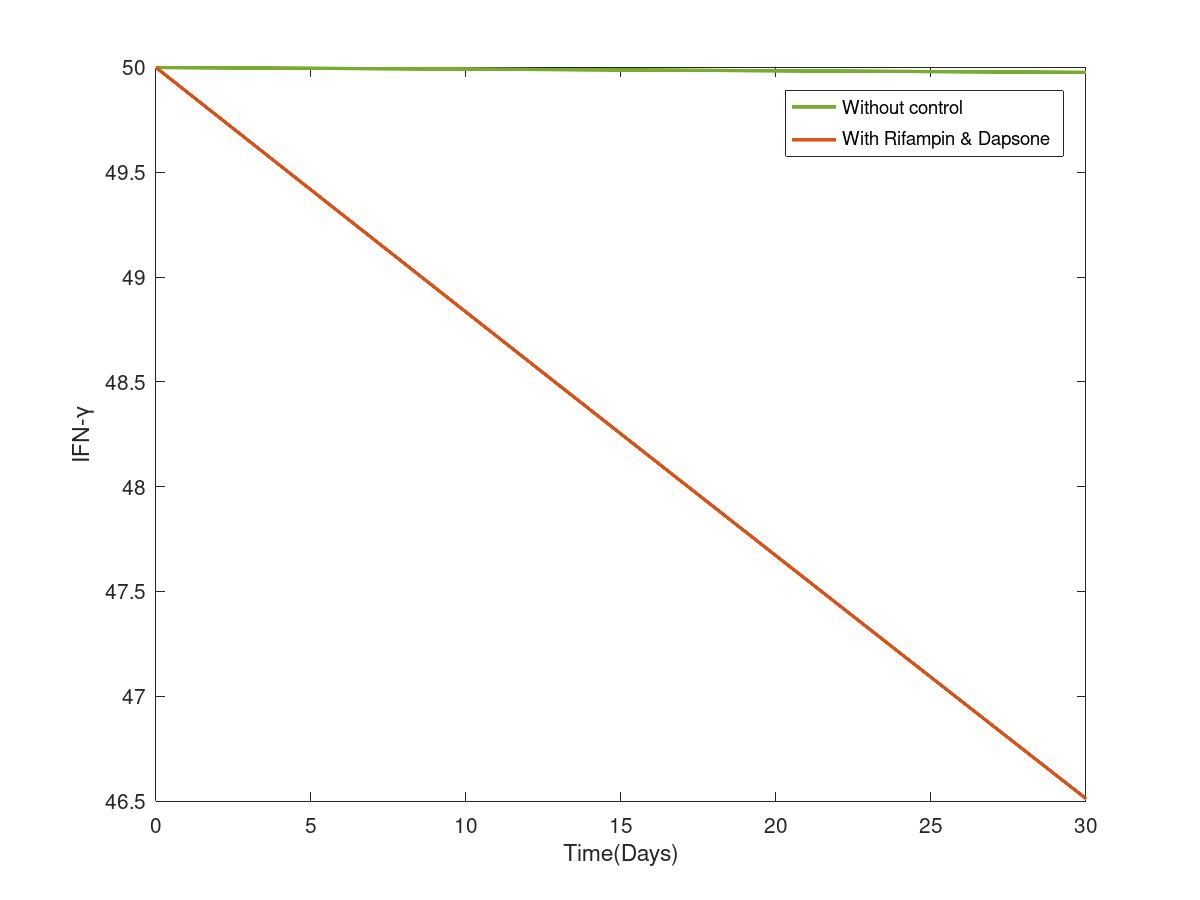}
        \caption{Graph 4}
        \label{fig:graph31}
    \end{subfigure}
    \hfill
    \begin{subfigure}{0.30\textwidth}
        \includegraphics[width=\textwidth]{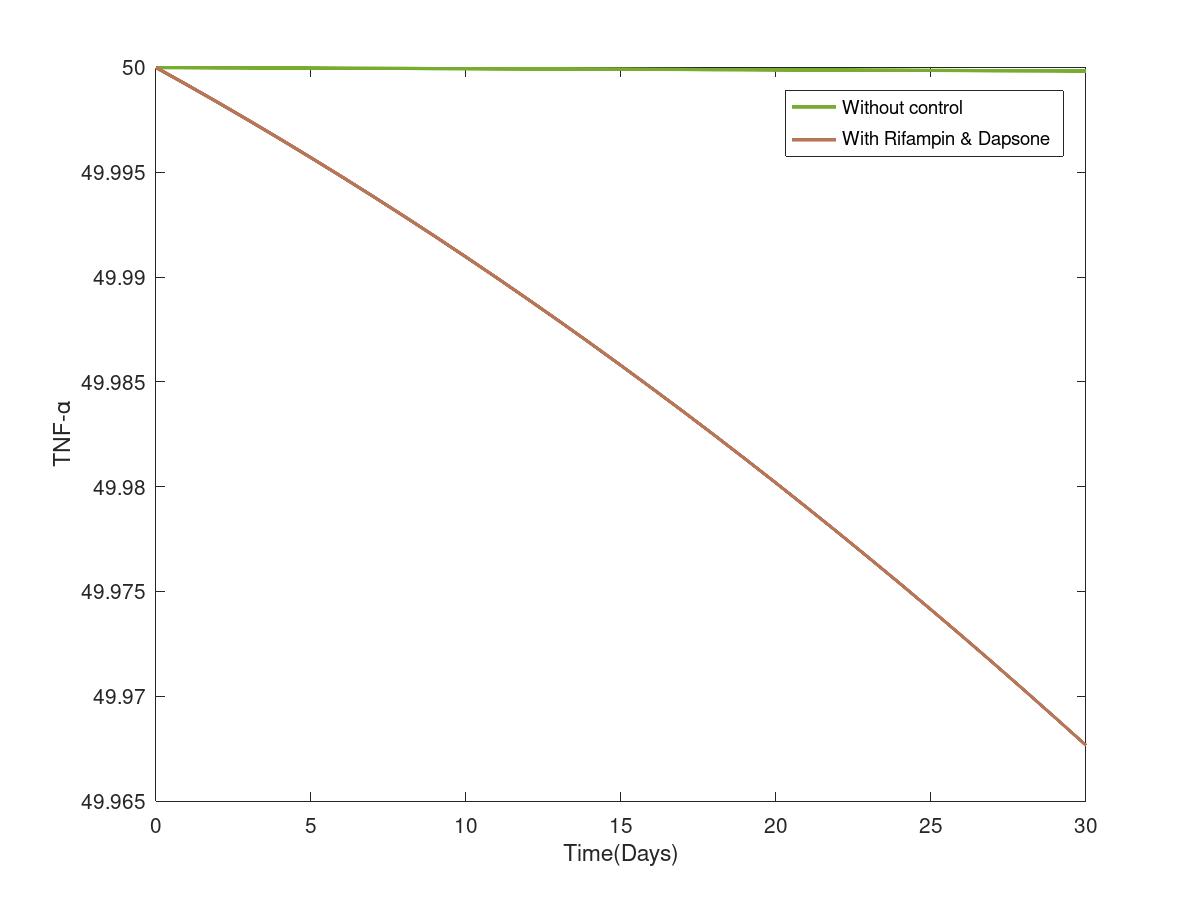}
        \caption{Graph 5}
         \label{fig:graph32}
    \end{subfigure}
    \hfill
    \begin{subfigure}{0.30\textwidth}
        \includegraphics[width=\textwidth]{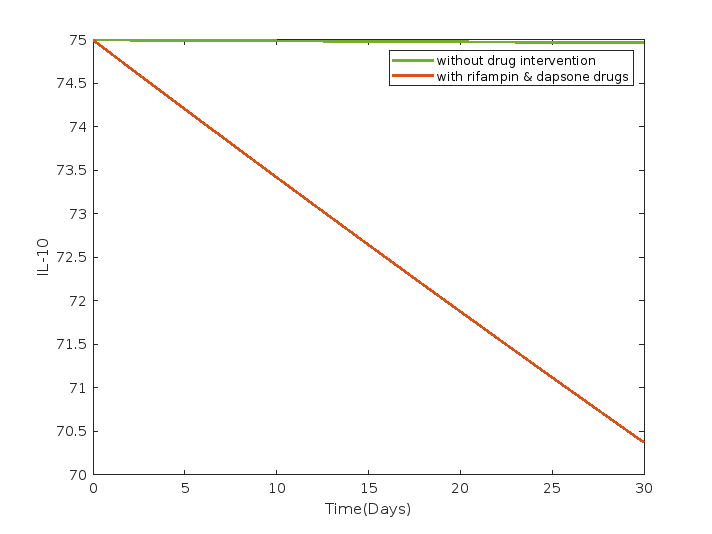}
        \caption{Graph 6}
         \label{fig:graph33}
    \end{subfigure}
    \hfill
    \begin{subfigure}{0.30\textwidth}
        \includegraphics[width=\textwidth]{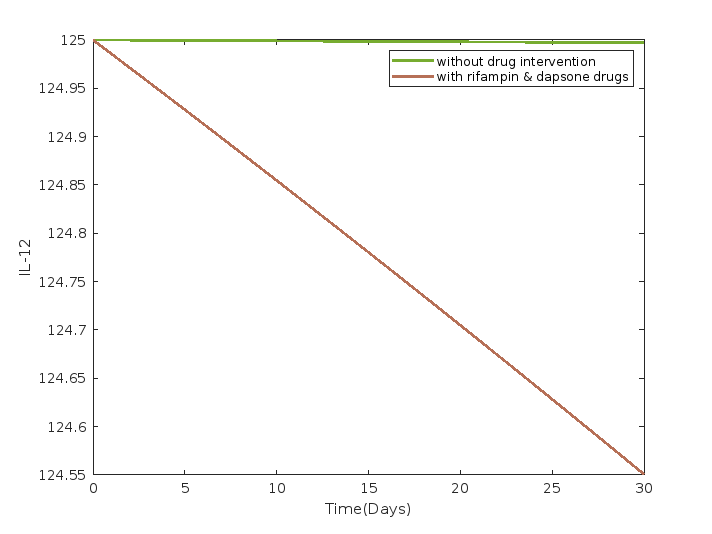}
        \caption{Graph 7}
         \label{fig:graph34}
    \end{subfigure}
    \hfill
    \begin{subfigure}{0.30\textwidth}
        \includegraphics[width=\textwidth]{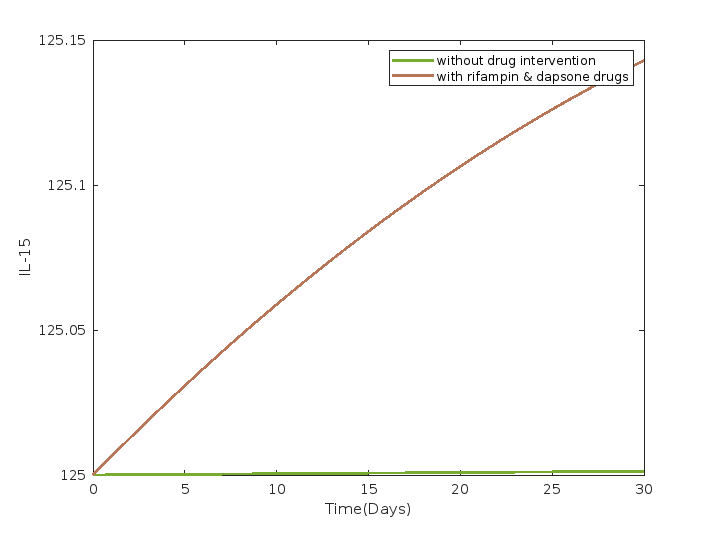}
        \caption{Graph 8}
         \label{fig:graph35}
    \end{subfigure}
    \hfill
    \begin{subfigure}{0.30\textwidth}
        \includegraphics[width=\textwidth]{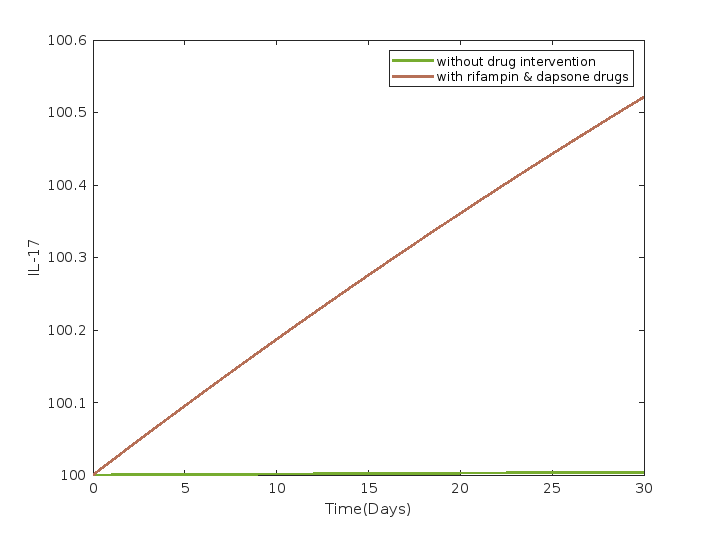}
        \caption{Graph 9}
         \label{fig:graph36}
    \end{subfigure}
    \caption{Plots depicting the influence of rifampin and dapsone drugs at a time for one month  }
    \label{fig:Rif&Dap}
\end{figure}

\begin{figure}[htbp]
    \centering
    \begin{subfigure}{0.30\textwidth}
        \includegraphics[width=\textwidth]{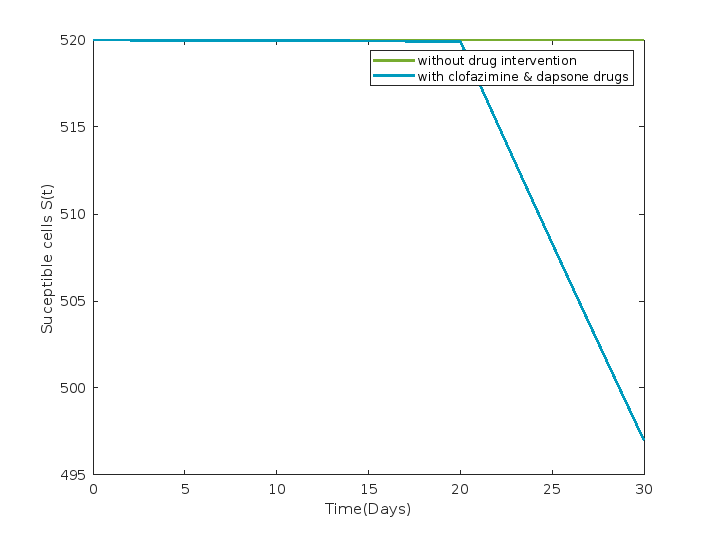}
        \caption{Graph 1}
         \label{fig:graph37}
    \end{subfigure}
    \hfill
    \begin{subfigure}{0.30\textwidth}
        \includegraphics[width=\textwidth]{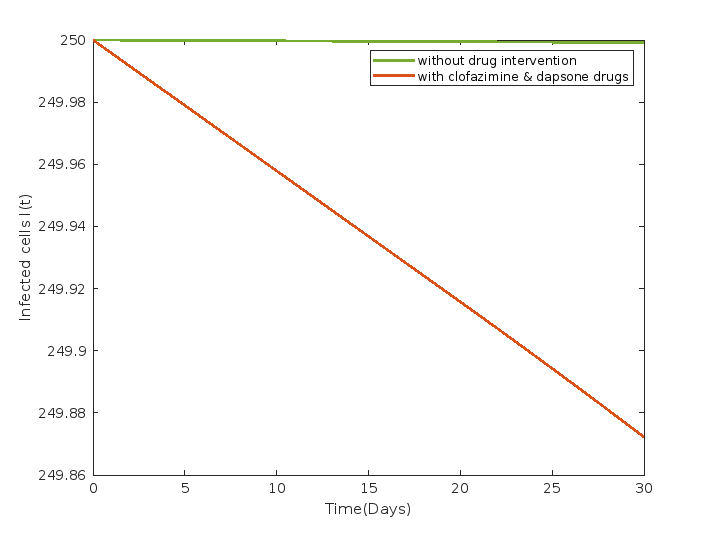}
        \caption{Graph 2}
         \label{fig:graph38}
    \end{subfigure}
    \hfill
    \begin{subfigure}{0.30\textwidth}
        \includegraphics[width=\textwidth]{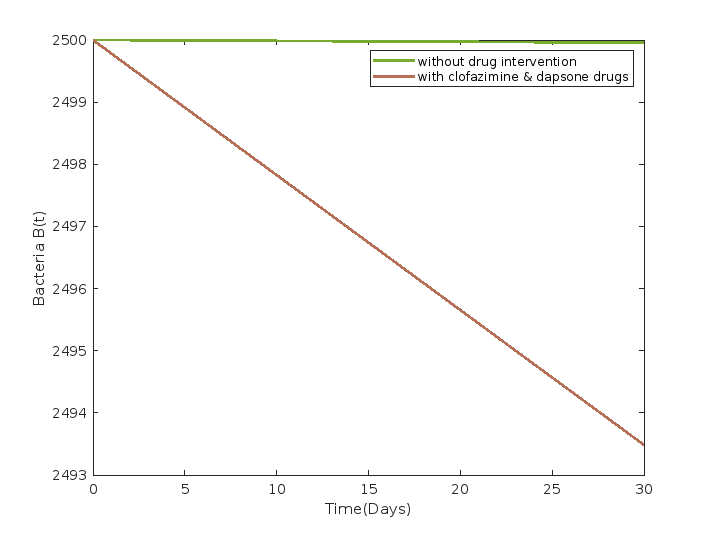}
        \caption{Graph 3}
         \label{fig:graph39}
    \end{subfigure}
    \hfill
    \begin{subfigure}{0.30\textwidth}
        \includegraphics[width=\textwidth]{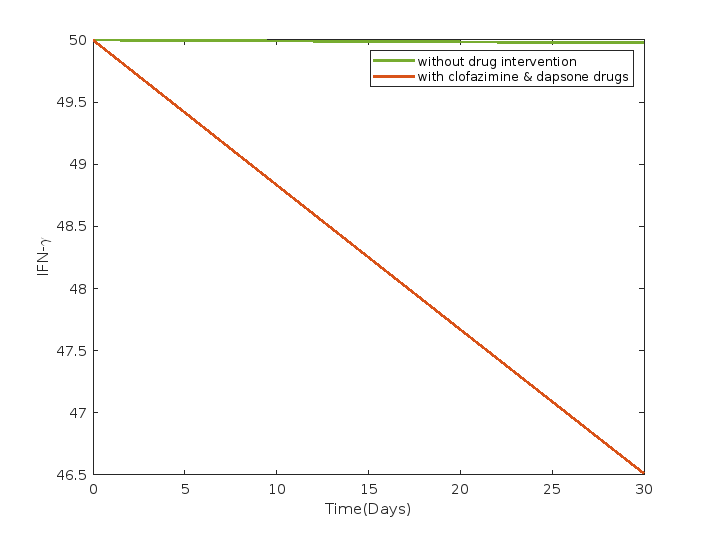}
        \caption{Graph 4}
         \label{fig:graph40}
    \end{subfigure}
    \hfill
    \begin{subfigure}{0.30\textwidth}
        \includegraphics[width=\textwidth]{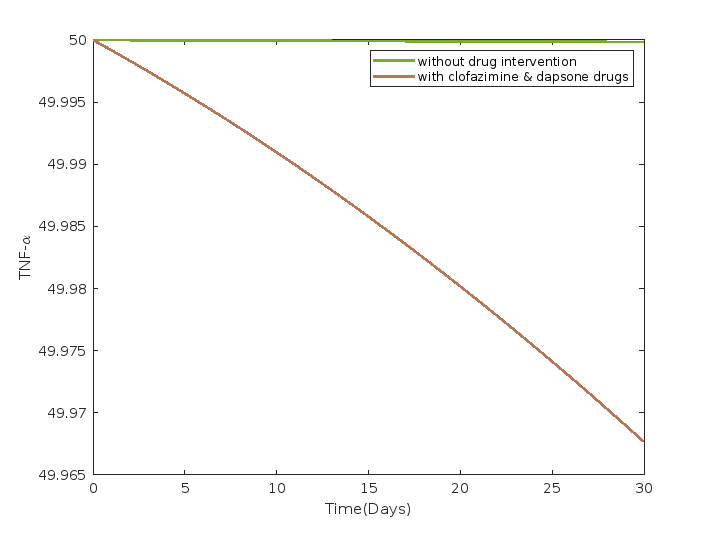}
        \caption{Graph 5}
         \label{fig:graph41}
    \end{subfigure}
    \hfill
    \begin{subfigure}{0.30\textwidth}
        \includegraphics[width=\textwidth]{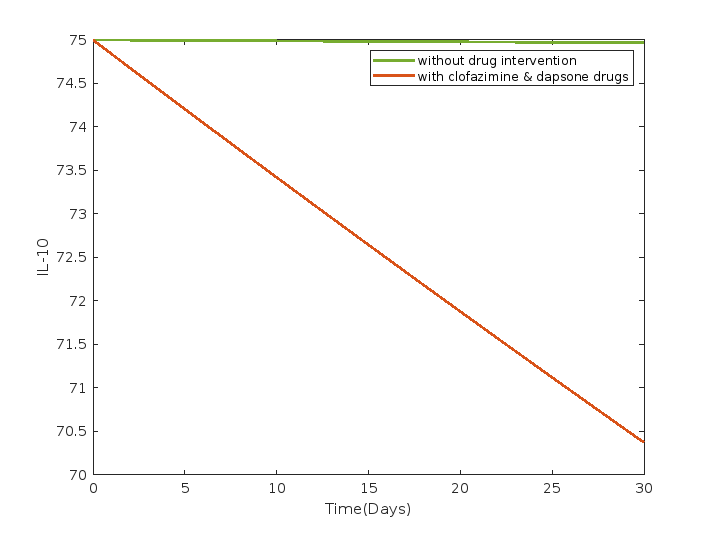}
        \caption{Graph 6}
         \label{fig:graph42}
    \end{subfigure}
    \hfill
    \begin{subfigure}{0.30\textwidth}
        \includegraphics[width=\textwidth]{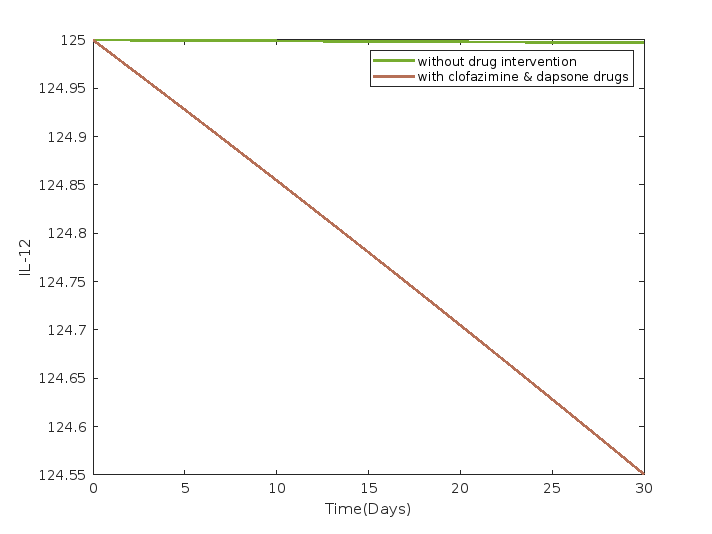}
        \caption{Graph 7}
         \label{fig:graph43}
    \end{subfigure}
    \hfill
    \begin{subfigure}{0.30\textwidth}
        \includegraphics[width=\textwidth]{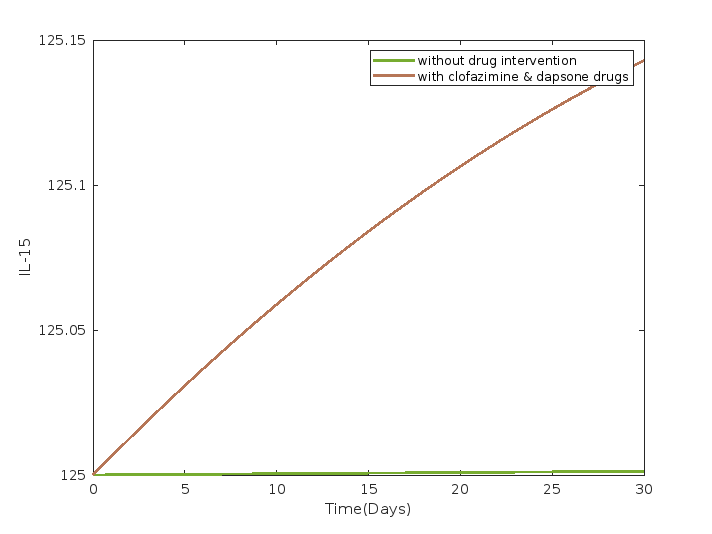}
        \caption{Graph 8}
         \label{fig:graph44}
    \end{subfigure}
    \hfill
    \begin{subfigure}{0.30\textwidth}
        \includegraphics[width=\textwidth]{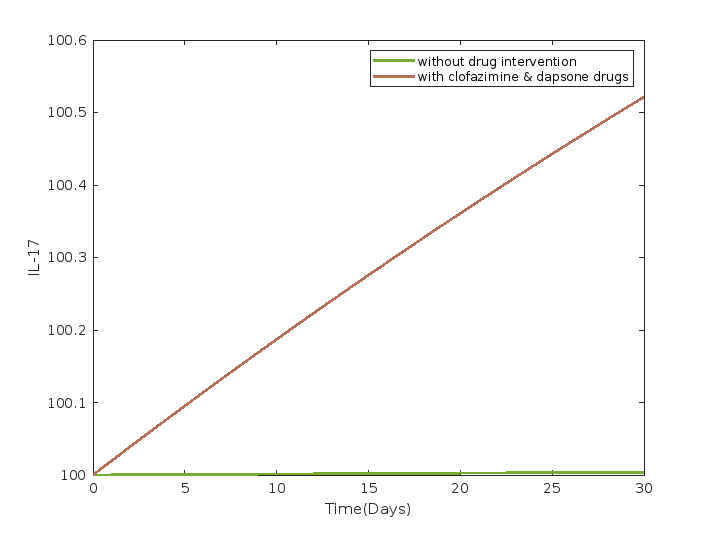}
        \caption{Graph 9}
         \label{fig:graph45}
    \end{subfigure}
    \caption{Plots depicting the influence of clofazimine and dapsone drugs at a time for one month  }
    \label{fig:Clo&Dap}
\end{figure}

\begin{figure}[htbp]
    \centering
    \begin{subfigure}{0.30\textwidth}
        \includegraphics[width=\textwidth]{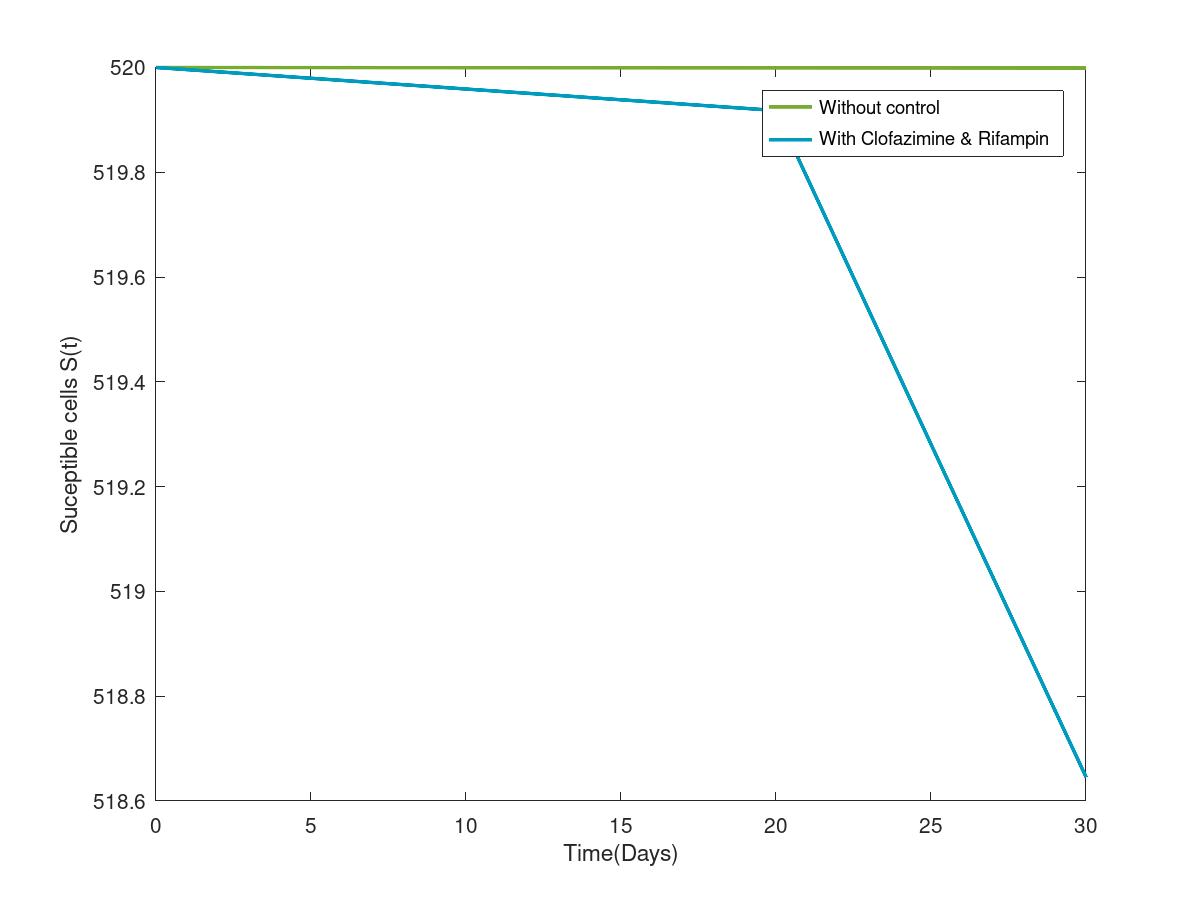}
        \caption{Graph 1}
        \label{fig:graph46}
    \end{subfigure}
    \hfill
    \begin{subfigure}{0.30\textwidth}
        \includegraphics[width=\textwidth]{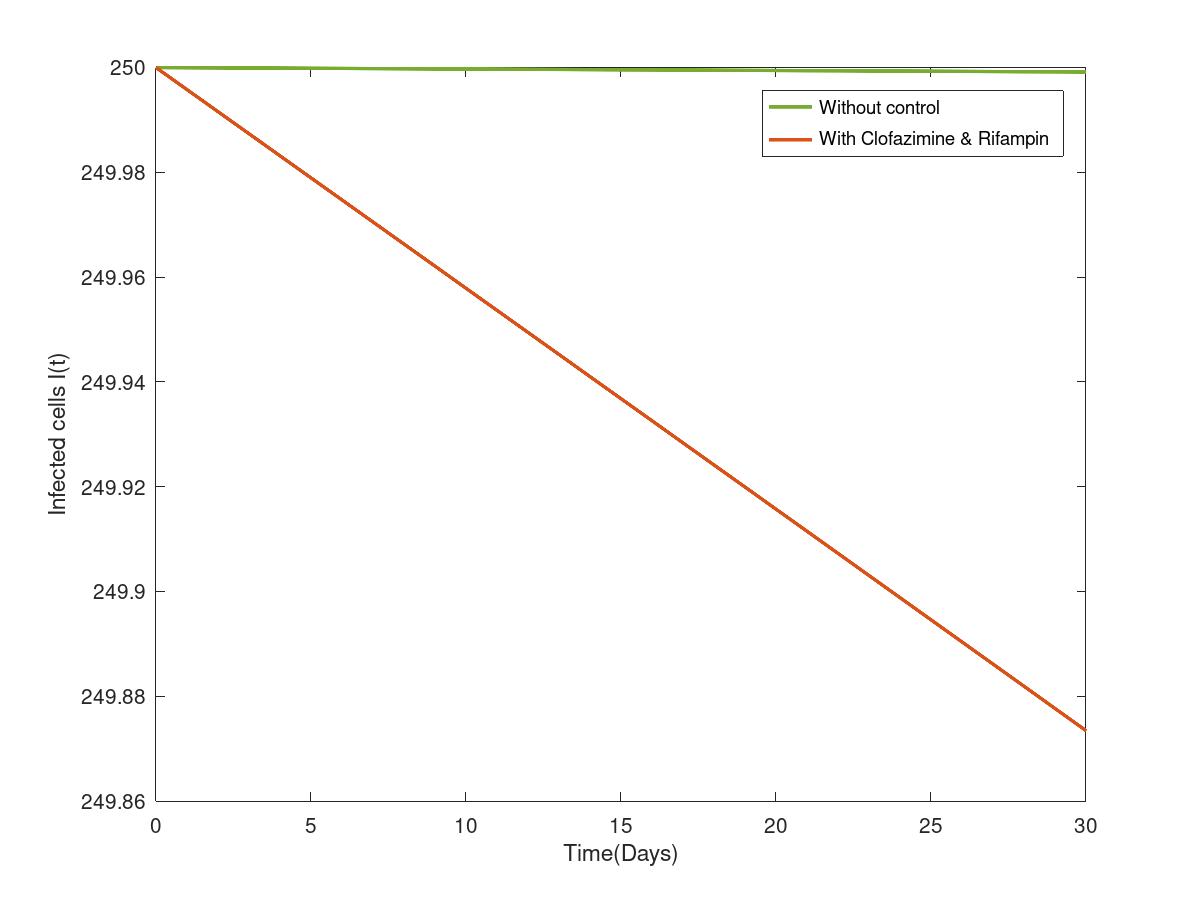}
        \caption{Graph 2}
        \label{fig:graph47}
    \end{subfigure}
    \hfill
    \begin{subfigure}{0.30\textwidth}
        \includegraphics[width=\textwidth]{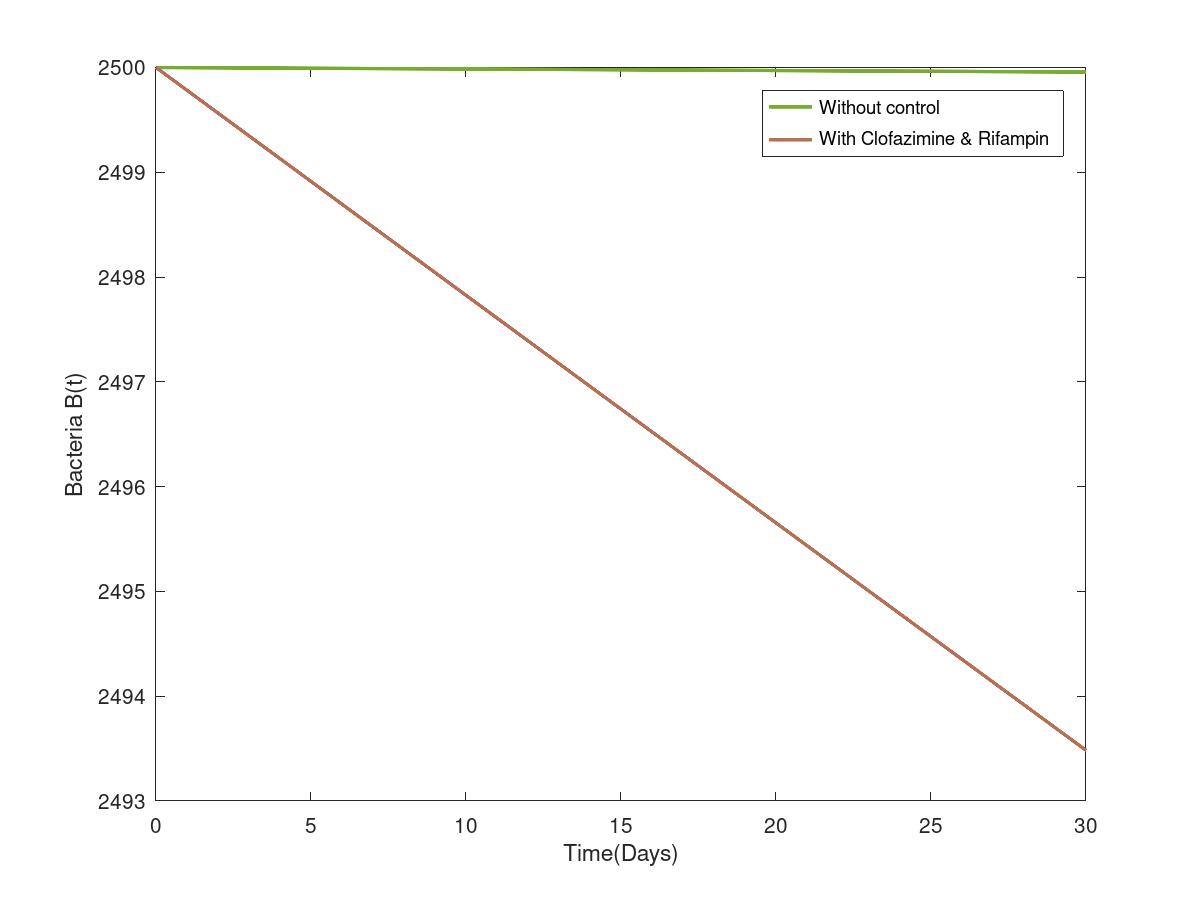}
        \caption{Graph 3}
        \label{fig:graph48}
    \end{subfigure}
    \hfill
    \begin{subfigure}{0.30\textwidth}
        \includegraphics[width=\textwidth]{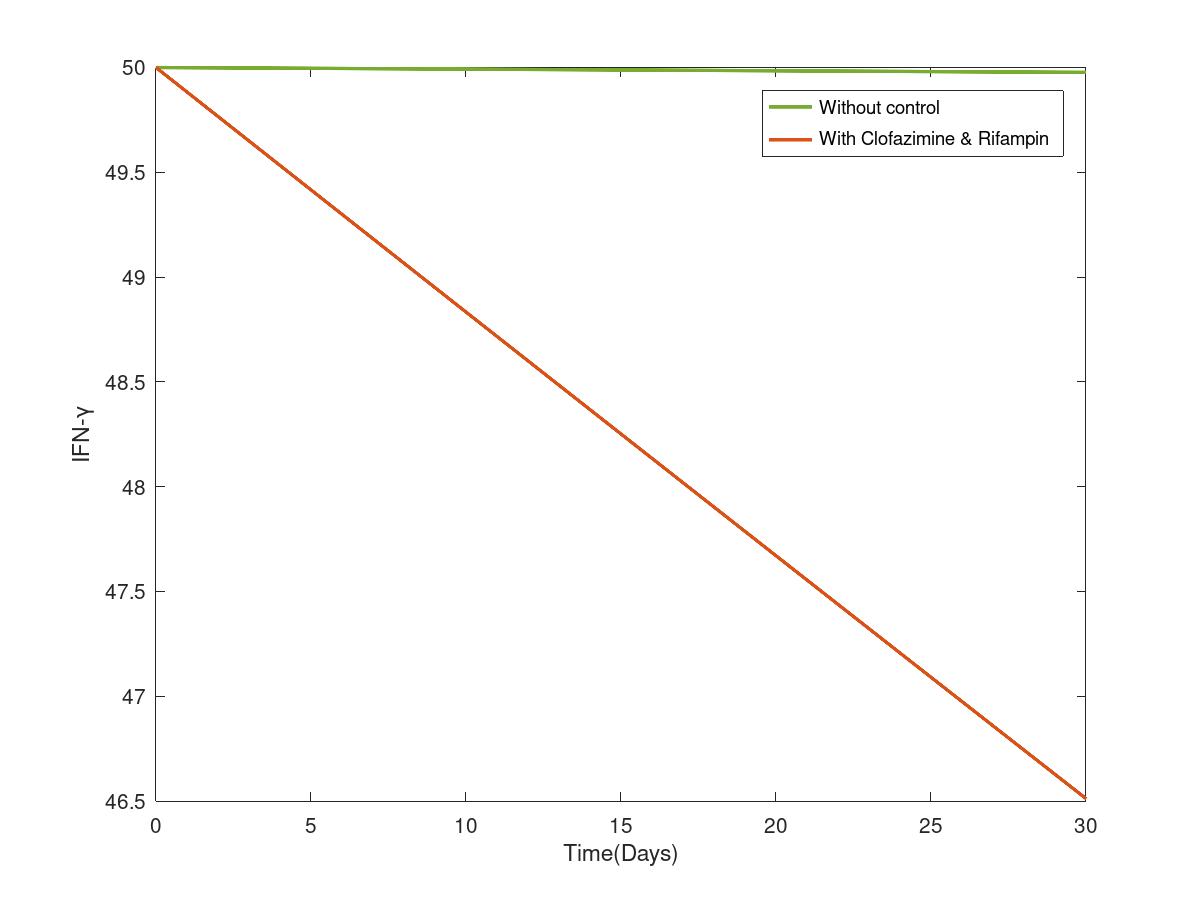}
        \caption{Graph 4}
        \label{fig:graph49}
    \end{subfigure}
    \hfill
    \begin{subfigure}{0.30\textwidth}
        \includegraphics[width=\textwidth]{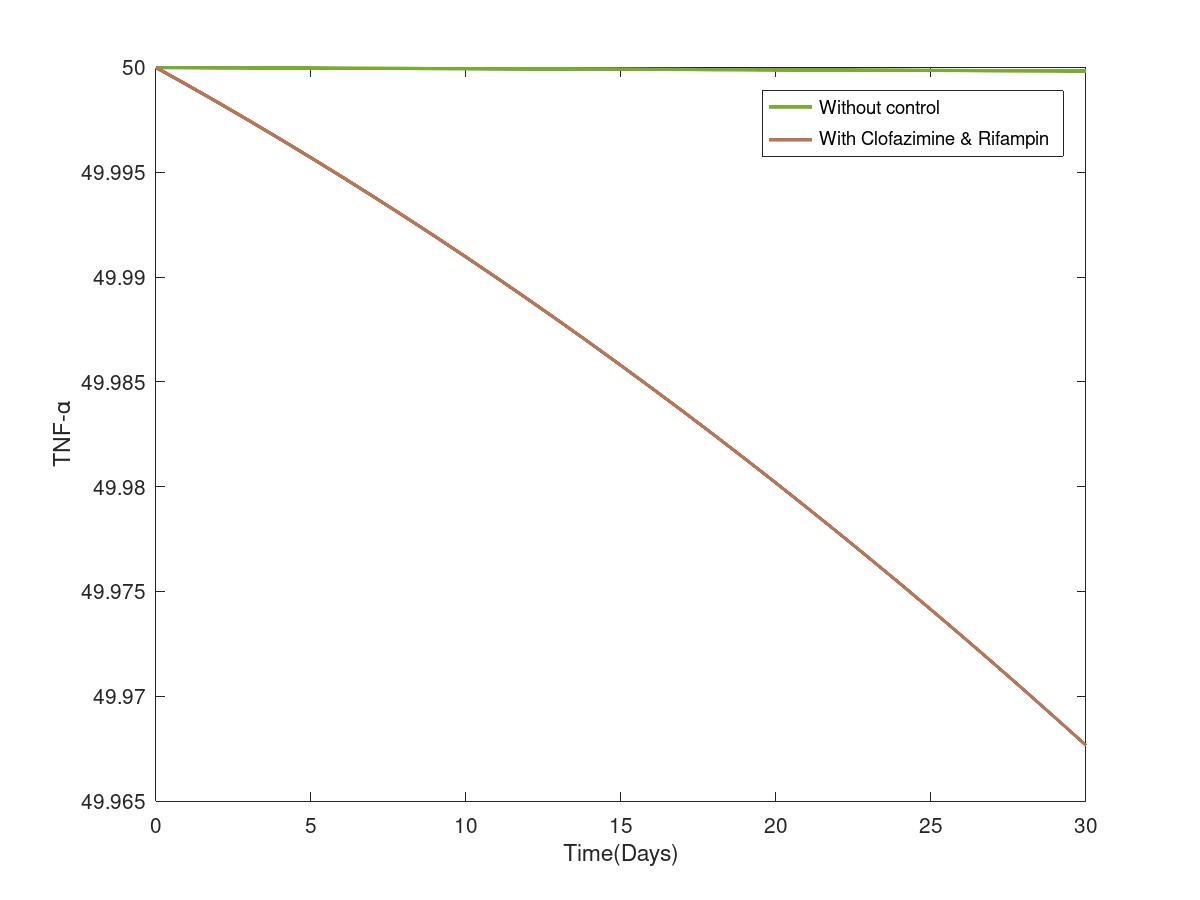}
        \caption{Graph 5}
        \label{fig:graph50}
    \end{subfigure}
    \hfill
    \begin{subfigure}{0.30\textwidth}
        \includegraphics[width=\textwidth]{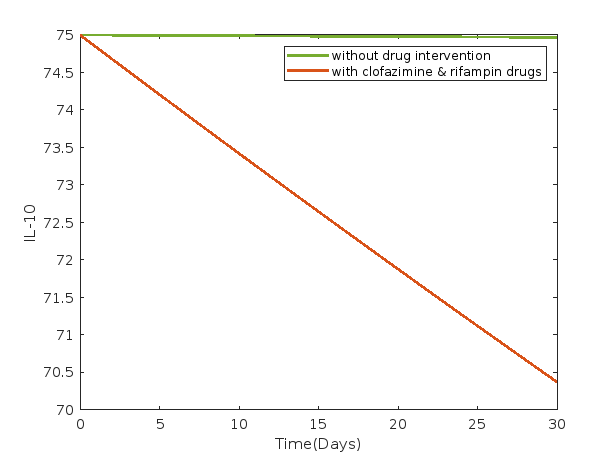}
        \caption{Graph 6}
        \label{fig:graph51}
    \end{subfigure}
    \hfill
    \begin{subfigure}{0.30\textwidth}
        \includegraphics[width=\textwidth]{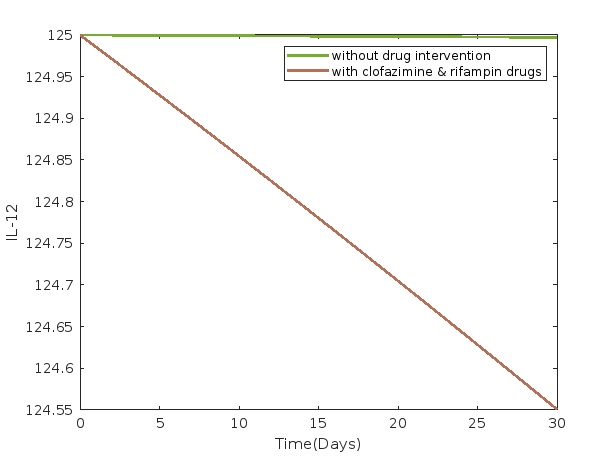}
        \caption{Graph 7}
        \label{fig:graph52}
    \end{subfigure}
    \hfill
    \begin{subfigure}{0.30\textwidth}
        \includegraphics[width=\textwidth]{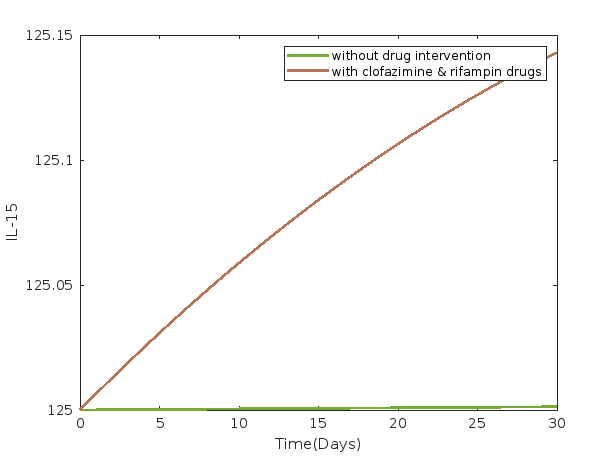}
        \caption{Graph 8}
        \label{fig:graph53}
    \end{subfigure}
    \hfill
    \begin{subfigure}{0.30\textwidth}
        \includegraphics[width=\textwidth]{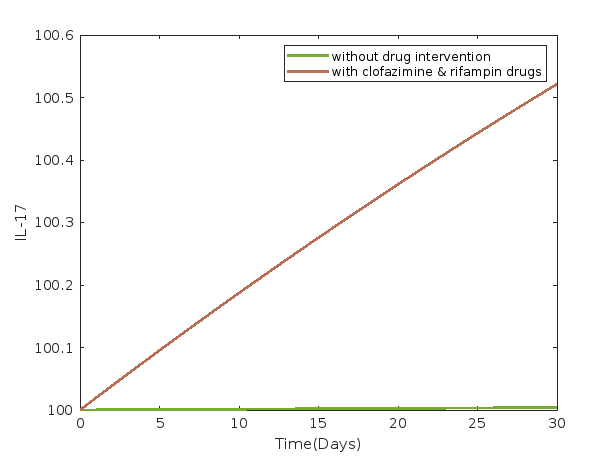}
        \caption{Graph 9}
        \label{fig:graph54}
    \end{subfigure}
    \caption{Plots depicting the influence of rifampin and Clofazimine drugs at a time for one month  }
    \label{fig:Rif&Clo}
\end{figure}

\begin{figure}[htbp]
    \centering
    \begin{subfigure}{0.30\textwidth}
        \includegraphics[width=\textwidth]{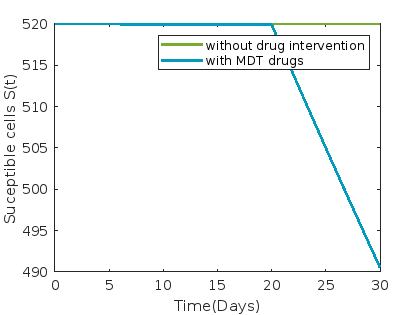}
        \caption{Graph 1}
        \label{fig:graph55}
    \end{subfigure}
    \hfill
    \begin{subfigure}{0.30\textwidth}
        \includegraphics[width=\textwidth]{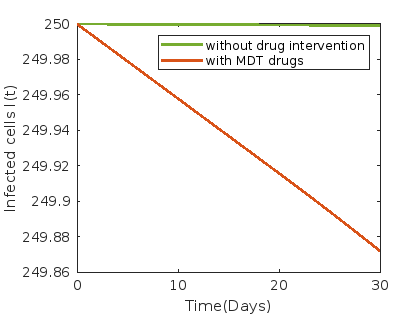}
        \caption{Graph 2}
        \label{fig:graph56}
    \end{subfigure}
    \hfill
    \begin{subfigure}{0.30\textwidth}
        \includegraphics[width=\textwidth]{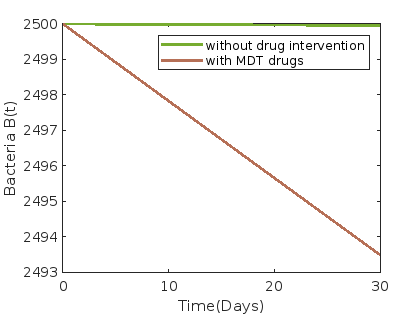}
        \caption{Graph 3}
        \label{fig:graph57}
    \end{subfigure}
    \hfill
    \begin{subfigure}{0.30\textwidth}
        \includegraphics[width=\textwidth]{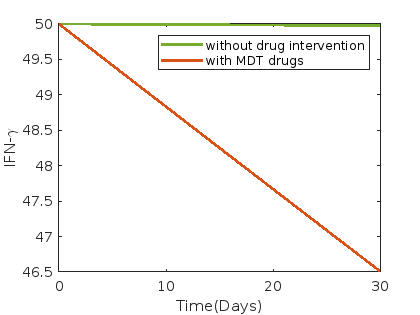}
        \caption{Graph 4}
        \label{fig:graph58}
    \end{subfigure}
    \hfill
    \begin{subfigure}{0.30\textwidth}
        \includegraphics[width=\textwidth]{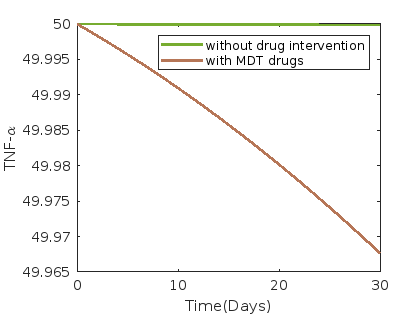}
        \caption{Graph 5}
        \label{fig:graph59}
    \end{subfigure}
    \hfill
    \begin{subfigure}{0.30\textwidth}
        \includegraphics[width=\textwidth]{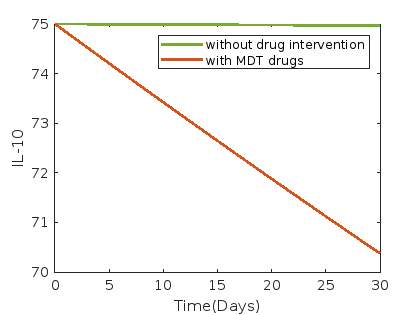}
        \caption{Graph 6}
        \label{fig:graph60}
    \end{subfigure}
    \hfill
    \begin{subfigure}{0.30\textwidth}
        \includegraphics[width=\textwidth]{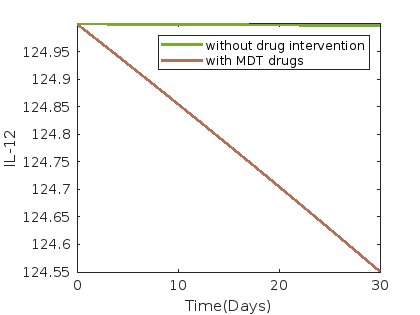}
        \caption{Graph 7}
        \label{fig:graph61}
    \end{subfigure}
    \hfill
    \begin{subfigure}{0.30\textwidth}
        \includegraphics[width=\textwidth]{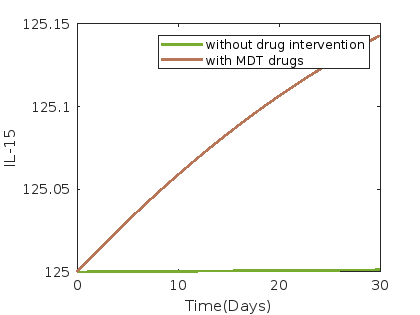}
        \caption{Graph 8}
        \label{fig:graph62}
    \end{subfigure}
    \hfill
    \begin{subfigure}{0.30\textwidth}
        \includegraphics[width=\textwidth]{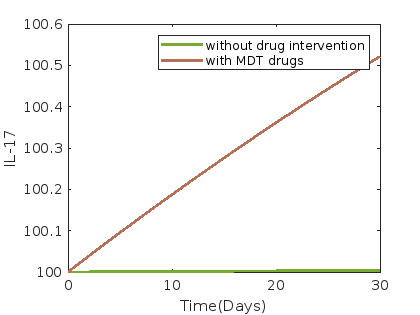}
        \caption{Graph 9}
        \label{fig:graph63}
    \end{subfigure}
    \caption{Plots depicting the influence of rifampin, clofazimine and dapsone drugs at a time for one month  }
    \label{fig:mdt}
\end{figure}

\begin{table}[htbp]
   \centering
    \begin{tabular}{|c|c|c|c|c|}
    \hline
        \textbf{compartments} & \textbf{without drugs } & \textbf{with rifampin } & \textbf{with dapsone } & \textbf{with clofazimine }\\  
        \hline
$S(t)$ & 519.999587   & 519.869137 & 518.629671 & 519.692631 \\
\hline
$I(t)$ & 249.999579  & 249.936781 & 249.936732 & 249.936774 \\
\hline
$B(t)$ & 2499.978250  & 2496.740464 & 2496.740464 & 2496.740464 \\
\hline
$I_{\gamma}(t)$ & 49.988312  &  48.252862 & 48.252862 & 48.252862 \\
\hline
$T_{\alpha}(t)$ & 49.999918 & 49.985089  & 49.985089 & 49.985089 \\
\hline
$I_{10}(t)$ & 74.984018 & 72.659750  &  72.659750 & 72.659750 \\
\hline
$I_{12}(t)$ & 124.998569 & 124.778566  & 124.778566 & 124.778566\\
\hline
$I_{15}(t)$ &  125.000643 & 125.079538 & 125.079538 & 125.079538 \\
\hline
$I_{17}(t)$ & 100.001938 & 100.270343 & 100.270343 & 100.270343\\
\hline
\end{tabular}
\caption{Average compartments values on individual administration of  rifampin, dapsone, clofazimine   over a 30-days period.}
\label{tab:avg_1d}
\end{table}

\begin{table}[htbp]
   \centering
    \begin{tabular}{|c|c|c|c|c|}
    \hline
        \textbf{compartments} & \textbf{without drugs } & \textbf{with rifampin } & \textbf{with dapsone } & \textbf{with clofazimine }\\  
        \hline
$S(t)$ & 519.999202   & 519.530281 & 513.249360 & 518.634813 \\
\hline
$I(t)$ & 249.999186  & 249.877716 & 249.877398 & 249.877671 \\
\hline
$B(t)$ & 2499.957950  & 2493.700671 & 2493.700671 & 2493.700671 \\
\hline
$I_{\gamma}(t)$ & 49.977404  &  46.627260  & 46.627260 & 46.627260 \\
\hline
$T_{\alpha}(t)$ & 49.999842 & 49.968988  & 49.968988 & 49.968988 \\
\hline
$I_{10}(t)$ & 74.969104 & 70.522049  &  70.522049 & 70.522049 \\
\hline
$I_{12}(t)$ & 124.997232 & 124.566354  & 124.566354 & 124.566354 \\
\hline
$I_{15}(t)$ &  125.001243 & 125.139782 & 125.139782 & 125.139782 \\
\hline
$I_{17}(t)$ & 100.003745 & 100.505891 & 100.505891 & 100.505891\\
\hline
\end{tabular}
\caption{30th-day compartments values on individual administration of  rifampin, dapsone, clofazimine over a 30-days period. }
\label{tab:last_1d}
\end{table}
\newpage

\begin{table}[htbp]
   \centering
    \begin{tabular}{|c|c|c|c|c|}
    \hline
        \textbf{compartments} & \textbf{without drugs} & \textbf{rifampin, dapsone } & \textbf{dapsone, clofazamine } & \textbf{clofazimine, rifampin } \\  
        \hline
$S(t)$ & 519.999587   & 517.960166 & 515.850690 & 519.099803 \\
\hline
$I(t)$ & 249.999579  & 249.936706 & 249.936623 & 249.936751 \\
\hline
$B(t)$ & 2499.978250  & 2496.740464 & 2496.740464 & 2496.740464 \\ 
\hline
$I_{\gamma}(t)$ & 49.988312  &  48.252862 & 48.252862 & 48.252862 \\
\hline
$T_{\alpha}(t)$ & 49.999918 & 49.985089  & 49.985089 & 49.985089 \\
\hline
$I_{10}(t)$ & 74.984018 & 72.659750  &  72.659750 & 72.659750 \\
\hline
$I_{12}(t)$ & 124.998569 & 124.778566  & 124.778566 & 124.778566\\
\hline
$I_{15}(t)$ &  125.000643 & 125.079538 & 125.079538 & 125.079538 \\
\hline
$I_{17}(t)$ & 100.001938 & 100.270343 & 100.270342 & 100.270343\\
\hline
\end{tabular}
\caption{Average compartments values on combined administration of   rifampin+dapsone, dapsone+clofazamine and clofazamine+rifampin   over a 30-days period.}
\label{tab:avg_2d}
\end{table}

\begin{table}[htbp]
   \centering
    \begin{tabular}{|c|c|c|c|c|}
    \hline
        \textbf{compartments} & \textbf{without drugs} & \textbf{rifampin, dapsone } & \textbf{dapsone, clofazamine } & \textbf{clofazimine, rifampin } \\  
        \hline
$S(t)$ & 519.999202   & 509.865071 & 499.232056 & 515.630278 \\
\hline
$I(t)$ & 249.999186  & 249.877226  & 249.876684 & 249.877519 \\
\hline
$B(t)$ & 2499.957950  & 2493.700671 & 2493.700671 & 2493.700671 \\ 
\hline
$I_{\gamma}(t)$ & 49.977404  &  46.627260 & 46.627261 & 46.627260 \\
\hline
$T_{\alpha}(t)$ & 49.999842 & 49.968988  & 49.968987 & 49.968988 \\
\hline
$I_{10}(t)$ & 74.969104 & 70.522049  &  70.522049 & 70.522049 \\
\hline
$I_{12}(t)$ & 124.997232 & 124.566354  & 124.566354 & 124.566354\\
\hline
$I_{15}(t)$ &  125.001243 & 125.139782 & 125.139781 & 125.139782 \\
\hline
$I_{17}(t)$ & 100.003745 & 100.505891 & 100.505890 & 100.505891\\
\hline
\end{tabular}
\caption{30th-day compartments values on combined administration of   rifampin+dapsone, dapsone+clofazamine and clofazamine+rifampin   over a 30-days period.}
\label{tab:last_2d}
\end{table}

\begin{table}[htbp]
    \centering
    \begin{tabular}{|c|c|c|c|c|}
        \hline
        \multirow{2}{*}{\textbf{compartments}} & \multicolumn{2}{c|}{\textbf{without drugs}} & \multicolumn{2}{c|}{\textbf{rifampin, dapsone and clofazimine }} \\
        \cline{2-5}
         &  Average & 30th day & Average & 30th day \\
        \hline
        $S(t)$ & 519.999587   & 519.999202  & 514.688830 & 493.397175 \\ 
        \hline
       $I(t)$ & 249.999579  &  249.999186 & 249.936577 & 249.876385 \\ 
       \hline
       $B(t)$ & 2499.978250  & 2499.957950 & 2496.740464 & 2493.700671 \\
       \hline
       $I_{\gamma}(t)$ & 49.988312  &  49.977404 & 48.252862 & 46.627261\\ 
       \hline
       $T_{\alpha}(t)$ & 49.999918 & 49.999842 & 49.985089 & 49.968987 \\
       \hline
       $I_{10}(t)$ & 74.984018 & 74.969104 & 72.659750 & 70.522049\\
       \hline
       $I_{12}(t)$ & 124.998569 & 124.997232 & 124.778566 & 124.566354 \\ 
       \hline
       $I_{15}(t)$ & 125.000643 & 125.001243 & 125.079538 & 125.139781 \\
       \hline
       $I_{17}(t)$ & 100.001938 & 100.003745 & 100.270342 & 100.505890\\
       \hline
    \end{tabular}
    \caption{Average and 30th day  compartments values on    MDT drug administration over a 30 days period.}
    \label{tab:Avg&last_3d}
\end{table}

\newpage

\section{Optimal Control studies incorporating second drug dosage} \label{sec5}
In this section, we introduce a delay into our model (\ref{sec2equ1}) - (\ref{sec2equ11}) to simulate a two-month duration for administering the drug, with the second dosage given 30 days after the first dosage based on tables \ref{tab:treatment1} and \ref{tab:treatment2} as per the clinical and medical guidelines. This model however can be extrapolated to  real-administration scenario based on the WHO 2018 guidelines for Multi Drug therapy (MDT) consisting of drugs rifampin, dapsone and clofazimine with certain dosage   administered every 30 days over a period of 12 months for the treatment of leprosy. The duration may vary based on whether it's paucibacillary leprosy (6 months) or multibacillary leprosy (12 months). Request to kindly refer tables \ref{tab:treatment1} and \ref{tab:treatment2}.\\

Motivated by the above now in our model (\ref{sec5equ1}) - (\ref{sec5equ11}), we introduce  a delay of $\tau$ = 30 days.  \\

We consider controls $D_{11}$, $D_{21}$, $D_{31}$ at $(t - \tau)$ for the first 30 days and controls $D_{12}$, $D_{22}$, $D_{32}$ associated with the next 30 days over a period of 60 days.
\subsection{The Delay Model} \label{sec5a}
The control set for this is given by 
$$U=\Big\{D_{ij}(t) \ \big| \ D_{ij}(t)\in[0,D_{ij}max],  1\leq i \leq 3, 1\leq j \leq 2,  t\in[0,T]\Big\},$$
and the revised objective function and the control system are provided as follows: \\
{\textit{Cost functional}}:
\begin{align}
  \begin{split} \label{costfdelay}
   \mathcal{J}_{min}\big(I,B,D_{11},D_{21},D_{31},D_{12},D_{22},D_{32}\big) \ &= \int_{0}^{T} \Big(I(t) + B(t)+P(D^{2}_{11}(t-\tau) + D^{2}_{12}(t))\\
   \ & + Q (D^{2}_{21}(t - \tau) + D^{2}_{22}(t)) +  R(D_{31}^{2}(t-\tau) + D_{32}^{2}(t)) \Big) dt 
  \end{split}
\end{align}
{\textit{Control system}}:
\begin{align}
     \frac{dc_{1}}{dt}  \ &=  \ \frac{D_{11}(t-\tau) + D_{21}(t - \tau) + D_{31}(t-\tau)}{V_{1}} + \frac{D_{12}(t) + D_{2}(t) + D_{3}(t)}{V_{1}} \ - \big( k_{12}  + k_{1}\big)c_{1}  \label{sec5equ1} \\
      \frac{dc_{2}}{dt} \ & = \ k_{12}\frac{V_{1}}{V_{2}} c_{1} \ - k_{2}c_{2}  \label{sec5equ2} \\
    \frac{dS}{dt} \ & = \ \omega \ - \beta S B  - \gamma S  - \mu_{1} S - (\mu_{d_1} + \mu_{d_2} +\mu_{d_3}) c_{1}(t-\tau-\tau_{d})S - (\mu_{d_1} + \mu_{d_2} +\mu_{d_3}) c_{1}(t-\tau_{d}) S\label{sec5equ3} \\
	  \frac{dI}{dt} \ &= \ \beta SB \ - \delta I  - \mu_{1} I - (\eta k_{d_1} + \eta  k_{d_2} + \eta  k_{d_3} ) \cdot (c_{2} - C_{min})  \cdot  H\left[(c_{2} - C_{min})\right] \cdot I \label{sec5equ4}\\ 
   	\frac{dB}{dt} \ &= \ \alpha I   \ - y B - \mu_{2} B  -  (k_{d_1} + k_{d_2} + k_{d_3})\cdot (c_{2} - C_{min} )  \cdot  H\left[(c_{2} - C_{min})\right] \cdot B\label{sec5equ5}\\
     \frac{dI_{\gamma}}{dt} \ &= \ \alpha_{I_{\gamma}} I   \ - \left[\delta_{T_{\alpha}}^{I_{\gamma}}T_{\alpha} + \delta_{I_{12}}^{I_{\gamma}}I_{12} +\delta_{I_{15}}^{I_{\gamma}}I_{15} +\delta_{I_{17}}^{I_{\gamma}}I_{17}\right]I - \mu_{I_{\gamma}}\big(I_{\gamma} - Q_{I_{\gamma}}\big) \label{sec5equ6}\\
     \frac{dT_{\alpha}}{dt} \ &= \ \beta_{T_{\alpha}}I_{\gamma} I   \  - \mu_{T_{\alpha}}\big(T_{\alpha} - Q_{T_{\alpha}}\big) \label{sec5equ7}\\
     \frac{dI_{10}}{dt} \ &= \ \alpha_{I_{10}} I   \ - \delta_{I_{\gamma}}^{I_{10}}I_{\gamma} - \mu_{I_{10}}\big(I_{10} - Q_{I_{10}}\big) \label{sec5equ8}\\
     \frac{dI_{12}}{dt} \ &= \ \beta_{I_{12}} I_{\gamma} I \ -  \mu_{I_{12}}\big(I_{12} - Q_{I_{12}}\big) \label{sec5equ9}\\
     \frac{dI_{15}}{dt} \ &= \  \beta_{I_{15}} I_{\gamma}I \ -  \mu_{I_{15}}\big(I_{15} - Q_{I_{15}}\big) \label{sec5equ10}\\
     \frac{dI_{17}}{dt} \ &= \  \beta_{I_{17}} I_{\gamma}I \ -  \mu_{I_{17}}\big(I_{17} - Q_{I_{17}}\big) \label{sec5equ11}
\end{align}
Here, the Lagrangian is the integrand of the cost functional (\ref{costfdelay}) and is given by
\begin{align}
   \begin{split}\label{Lagradly}
   L(I,B,D_{11},D_{21},D_{31},D_{12},D_{22},D_{32})\ & = I(t) + B(t)+P(D^{2}_{11}(t-\tau) + D^{2}_{12}(t))\\
   \ & + Q (D^{2}_{21}(t - \tau) + D^{2}_{22}(t)) +  R(D_{31}^{2}(t-\tau) + D_{32}^{2}(t)) 
   \end{split}
\end{align}
The set of admissible solutions for the above optimal control problem will be 
$$\Omega =\Big\{(I,B,D_{11},D_{21},D_{31},D_{12},D_{22},D_{32}) \ \big| I,B \ satisfy \ (\ref{sec5equ1})-(\ref{sec5equ11}) \ \forall (D_{11},D_{21},D_{31},D_{12},D_{22},D_{32}) \in U \Big\}.$$
The existence of optimal control can be shown similarly as in section 4.\\
The Hamiltonian of the control system (\ref{sec5equ1}) - (\ref{sec5equ11}) is as follows
\begin{align}
\begin{split}
    \mathcal{H}(I,B,D_{11},D_{21},D_{31},D_{12},D_{22},D_{32},\lambda) & = L(I,B,D_{11},D_{21},D_{31},D_{12},D_{22},D_{32}) + \lambda_{1}\frac{dc_1}{dt} + \lambda_{2}\frac{dc_2}{dt} \\
    \ & + \lambda_{3}\frac{dS}{dt} + \lambda_{4}\frac{dI}{dt} + \lambda_{5}\frac{dB}{dt} + \lambda_{6}\frac{dI_{\gamma}}{dt} + \lambda_{7}\frac{dT_{\alpha}}{dt} + \lambda_{8}\frac{dI_{10}}{dt}\\
    \ & + \lambda_{9}\frac{dI_{12}}{dt} +\lambda_{10}\frac{dI_{15}}{dt} +\lambda_{11}\frac{dI_{17}}{dt}
\end{split}
\end{align}
where $\lambda = (\lambda_{1},\lambda_{2},\lambda_{3},\lambda_{4},\lambda_{5},\lambda_{6},\lambda_{7},\lambda_{8},\lambda_{9},\lambda_{10},\lambda_{11})$ is co-state variable or adjoint vector. \\

Since we have $D^{*} = (D_{11}^{*},D_{12}^{*},D_{21}^{*},D_{22}^{*},D_{31}^{*},D_{32}^{*})$ and $X^{*} = (x_{1},x_{2},x_{3},x_{4},x_{5},x_{6},x_{7},x_{8},x_{9},x_{10},x_{11})$ as optimal control and state variable respectively, using Pontryagin maximum principle there exists an optimal co-state variable say $\lambda^{*}$
which satisfies the canonical equation
\begin{equation}\label{co-stdly}
       \frac{d\lambda_{j}}{dt} = -\frac{\partial \mathcal{H}(X^{*},D^{*},\lambda^{*})}{\partial x_{j}}
\end{equation}
Using the above equation we get the below system of ODE's for co-state variables as follows\\
\begin{align}
    \frac{d\lambda_{1}}{dt} &= (k_{12} + k_1)\lambda_{1} - k_{12}\left(\frac{V_1}{V_2}\right)\lambda_{2} + 2(\mu_{d_1} + \mu_{d_2} + \mu_{d_3})S\lambda_{3} \\
   \frac{d\lambda_{2}}{dt} &= k_2\lambda_{2} + \eta(k_{d_1}+ k_{d_2} + k_{d_3}) H[(c_2-c_{min})] \cdot I\lambda_{4} + (k_{d_1} + k_{d_2} + k_{d_3})  H[(c_2-c_{min})] \cdot B\lambda_{5} \\
    \frac{d\lambda_{3}}{dt} &= \big(\beta B + \mu_1 + \gamma + (\mu_{d_1} + \mu_{d_2} + \mu_{d_3})c_1(t - \tau - \tau_d) + (\mu_{d_1} + \mu_{d_2} + \mu_{d_3})c_1(t - \tau_d) \big)\lambda_{3} - \beta B \lambda_{4} \\
    \begin{split}
    \frac{d\lambda_{4}}{dt} &= \big(\delta + \mu_{1} + (\eta k_{d_1} + \eta k_{d_2} + \eta k_{d_3})(c_2 - C_{min}) H[(c_2 - C_{min})]\big)\lambda_{4} - \alpha \lambda_5 - \alpha_{I_{\gamma}}\lambda_{6} \\
    \ & + \left(\delta_{T_{\alpha}}^{I_{\gamma}}T_{\alpha} + \delta_{I_{12}}^{I_{\gamma}}I_{12} +\delta_{I_{15}}^{I_{\gamma}}I_{15} +\delta_{I_{17}}^{I_{\gamma}}I_{17}\right) \lambda_{6} + \beta_{T_{\alpha}}I_{\gamma} \lambda_{7} - \alpha_{I_{10}} \lambda_{8} - \beta_{I_{12}} I_{\gamma} \lambda_{9} - \beta_{I_{15}} I_{\gamma} \lambda_{10} - \beta_{I_{17}} I_{\gamma} \lambda_{11} - 1 
    \end{split} \\
   \frac{d\lambda_{5}}{dt} &= \beta S \lambda_{3} - \beta S \lambda_{4} + \big(y + \mu_{2} + (k_{d_1} + k_{d_2} + k_{d_3})(c_2 - C_{min}) H[(c_2 - C_{min})]\big)\lambda_{5} - 1 \\
    \frac{d\lambda_{6}}{dt} &= \mu_{I_{\gamma}}\lambda_{6} - \beta_{T_\alpha}I\lambda_{7} + \delta_{I_{10}}^{I_\gamma} I\lambda_8 - \beta_{I_{12}}I\lambda_{9} - \beta_{I_{15}}I\lambda_{10} - \beta_{I_{17}}I\lambda_{11} \\
    \frac{d\lambda_{7}}{dt} &= \delta_{T_{\alpha}}^{I_\gamma} I\lambda_6 + \mu_{T_{\alpha}}\lambda_{7} \\
    \frac{d\lambda_{8}}{dt} &= \mu_{I_{10}} \lambda_{8} \\
    \frac{d\lambda_{9}}{dt} &= \delta_{I_{12}}^{I_\gamma} I\lambda_6  + \mu_{I_{12}}\lambda_{9}  \\
    \frac{d\lambda_{10}}{dt} &= \delta_{I_{15}}^{I_\gamma} I\lambda_6  + \mu_{I_{15}}\lambda_{10} \\
    \frac{d\lambda_{11}}{dt} &= \delta_{I_{17}}^{I_\gamma} I\lambda_6  + \mu_{I_{17}}\lambda_{11}
\end{align}
and the transversality condition $\lambda_{i}(T) = \frac{\partial \phi}{\partial x_i}\big|_{t=T} = 0 $ for all $i = 1,2,3,...,11$ where in this case, the terminal cost function, represented by $\phi$, is constantly zero. \\

Now we use $Newton's \ Gradient \ method$ from \cite{edge1976function} to obtain the optimal value of the controls. For this recursive formula is employed to update the control at each step of the numerical simulation as follows

\begin{align}\label{updatecontrol3}
    D_{ij}^{k+1}(t) = D_{ij}^{k}(t) + \theta_{k}d_k
\end{align}
Here,  $D_{ij}^{k}(t)$ represents the control value at the $k^{th}$ iteration at a given time $t$, $d_k$ signifies the direction, and $\theta_k$ denotes the step size. The direction $d_k$ can be evaluated as negative of gradient of the objective function i.e $d_k = - g_{ij}(D_{ij}^{k})$ ,where $g_{ij}(D_{ij}^{k}) = \frac{\partial \mathcal{H}}{\partial D_{ij}}\big|_{D_{ij}^{k}(t)}$ as mentioned in\cite{edge1976function}.The step size $\theta_k$ is determined at each iteration using a linear search technique aimed at minimizing the Hamiltonian,$\mathcal{H}$.Therefore (\ref{updatecontrol3}) can become as 
\begin{align}\label{updatecontrol4}
    D_{ij}^{k+1}(t) = D_{ij}^{k}(t) - \theta_{k}\frac{\partial \mathcal{H}}{\partial D_{ij}}\Big|_{D_{ij}^{k}(t)}
\end{align}
To implement the aforementioned approach, we need to compute the gradient for each control, denoted as $g_{ij}(D_{ij}^{k})$ , which are listed as follows
\begin{align*}
        g_{11}(D_{11})\ & = 2PD_{11}(t - \tau) + \frac{\lambda_{1}}{V_1}\\
        g_{12}(D_{12})\ &= 2PD_{12}(t) + \frac{\lambda_{1}}{V_1}\\
        g_{21}(D_{21})\ &= 2QD_{21}(t - \tau) + \frac{\lambda_{1}}{V_1}\\
        g_{22}(D_{22})\ &= 2QD_{22}(t) + \frac{\lambda_{1}}{V_1}\\
        g_{31}(D_{31})\ &= 2RD_{31}(t - \tau) + \frac{\lambda_{1}}{V_1} \\
        g_{32}(D_{32})\ &= 2RD_{32}(t) + \frac{\lambda_{1}}{V_1} 
\end{align*}

\subsection{Numerical simulations} \label{sec5b}
Here, we employ the parameter values, initial conditions, and numerical simulation methods identical to those used in section \ref{sec4b}. Additionally, we introduce one extra control associated with each drug, with a delay of $\tau = 30$ days. Therefore, the values of weights associated with these controls remain the same as in section \ref{sec4b}.
We proceed to numerically simulate the populations of  $c_1,$  $c_2,$ $S,$ $I,$ and $B$ along with cytokines levels, employing
single, double and triple control interventions of MDT for 60 days with a delay of 30 days for second dosage.

\subsection{Findings} \label{sec5c}
In this section, we analyze the results from the simulations described earlier. Figures \ref{fig:rifampin_dly} - \ref{fig:mdt_dly} depict the dynamics of the $S$, $I$, $B$, $I_{\gamma}$, $T_{\alpha}$, $I_{10}$, $I_{12}$, $I_{15}$, and $I_{17}$ compartments in our model (\ref{sec5equ3})–(\ref{sec5equ11}). Additionally, these figures illustrate the control flow associated with under different drug administration scenarios over a 60-day period with a single delay of 30 days. Each panel represents a compartment and compares its dynamics without drug intervention, with single dosgae drug interventions and two dosage drug interventions. \\

Figure \ref{fig:rifampin_dly} depicts the dynamics under rifampin administration, while figures \ref{fig:dapsone_dly} and \ref{fig:clofazimine_dly} show the dynamics under dapsone and clofazimine administration, respectively. Average and 60th-day values of each compartment without control and with single drug intervention such as rifampin, dapsone, and clofazimine are presented in tables \ref{tab:avg1_dly} and \ref{tab:last1_dly} respectively, with a  30-day delay for second dosage. \\

In all instances of single drug intervention, involving the administration of rifampin, dapsone, and clofazimine, it is evident from figures \ref{fig:rifampin_dly}, \ref{fig:dapsone_dly}, and \ref{fig:clofazimine_dly} that various compartments, including susceptible cells $S(t)$, infected cells $I(t)$, and bacterial load $B(t)$, as well as  $I_{\gamma}(t)$,  $T_{\alpha}(t)$, $I_{10}(t)$, and  $I_{12}(t)$, exhibit a decreasing trend when compared to scenarios without drug intervention and delay. Conversely, compartments  $I_{15}(t)$ and  $I_{17}(t)$ demonstrate an increasing trend.\\

Dapsone exhibits the most substantial reduction in  both susceptible and infected cells whereas rifampin shows the least reduction when administered as a single drug intervention.\\ 

The results from the figures \ref{fig:rifampin_dly}, \ref{fig:dapsone_dly}, and \ref{fig:clofazimine_dly}  depictit that  the compartments bacterial load, IFN-$\gamma$, TNF-${\alpha}$, IL-10, IL-12, IL-15, and IL-17 show consistency across all single drug interventions. However, differences are noticeable in the values presented in tables \ref{tab:avg_1d} and \ref{tab:last_1d}. Additionally, these tables clearly indicate a reduction in infected cells.\\

From tables \ref{tab:avg_1d} and \ref{tab:last_1d}, 18, 19 it is evident that, similar to susceptible and infected cells, dapsone is also more effective in reducing bacterial load, TNF-${\alpha}$, IL-10, and IL-12, whereas rifampin shows the least reduction when administered as a single drug intervention.\\

IFN-$\gamma$ experiences the most significant reduction with rifampin and the least with dapsone. Additionally, the increment of IL-15 and IL-17 is greater with rifampin and least with dapsone. \\

Figures \ref{fig:Rif&Dap_dly}, \ref{fig:Dap&Clo_dly} and \ref{fig:Clo&Rif_dly} illustrate the dynamics under combinations of two drugs: rifampin and dapsone, clofazimine and dapsone, and rifampin and clofazimine, respectively. Average and 60th-day values of each compartment without control and with a two-drug combined intervention such as rifampin and dapsone, clofazimine and dapsone, and rifampin and clofazimine are presented in tables \ref{tab:avg2_dly} and \ref{tab:last2_dly} respectively, where a  30-day delay is incorporated for second dosage. \\

In all instances of two-drug combination intervention, involving the administration of drug combinations such as rifampin and dapsone, dapsone and clofazimine, and rifampin and clofazimine, it is evident from figures \ref{fig:Rif&Dap_dly}, \ref{fig:Dap&Clo_dly}, and \ref{fig:Clo&Rif_dly} that various compartments, including susceptible cells $S(t)$, infected cells $I(t)$, and bacterial load $B(t)$, as well as  $I_{\gamma}(t)$,  $T_{\alpha}(t)$,  $I_{10}(t)$, and  $I_{12}(t)$, exhibit a decreasing trend when compared to scenarios without drug intervention and delay. Conversely, compartments  $I_{15}(t)$ and  $I_{17}(t)$ demonstrate an increasing trend. \\

The two-drug combination of dapsone and clofazimine exhibits the most substantial reduction in both susceptible and infected cells, whereas the two-drug combination of rifampin and clofazimine shows the least reduction when administered as a two-drug intervention. \\

The results from figures \ref{fig:Rif&Dap_dly}, \ref{fig:Dap&Clo_dly}, and \ref{fig:Clo&Rif_dly}  depicting the compartments bacterial load, IFN-$\gamma$, TNF-${\alpha}$, IL-10, IL-12, IL-15, and IL-17, show consistency across all two-drug interventions. However, differences are noticeable in the values presented in tables \ref{tab:avg2_dly} and \ref{tab:last2_dly}. Additionally, these tables clearly indicate a reduction in infected cells. \\

From tables \ref{tab:avg2_dly} and \ref{tab:last2_dly}, it is evident that, similar to susceptible and infected cells, dapsone and clofazimine is also more effective in reducing bacterial load, TNF-${\alpha}$, IL-10, and IL-12, whereas rifampin and clofazimine shows the least reduction when administered as a two-drug intervention. \\

IFN-$\gamma$ experiences the most significant reduction with rifampin and clofazimine, and the least with dapsone and clofazimine. Additionally, the increment of IL-15 and IL-17 is greater with rifampin and clofazimine, and less with dapsone and clofazimine in two-drug combination. \\

Figure \ref{fig:mdt_dly} illustrates the dynamics under the administration of MDT drugs, comprising rifampin, clofazimine, and dapsone. \\

Average and 60th-day values of each compartment without control and with MDT drug intervention such as rifampin, dapsone, and clofazimine are presented in table \ref{tab:Avg&last3_dly}, where a  30-day delay is incorporated for second drug dosage. \\

In case of MDT drug (rifampin, clofazimine, and dapsone) interventions, we observe from figure \ref{fig:mdt_dly} that the compartments susceptible cells $S(t)$, infected cells $I(t)$, and bacterial load $B(t)$, as well as for  $I_{\gamma}(t)$,  $T_{\alpha}(t)$,  $I_{10}(t)$, and  $I_{12}(t)$, show a decreasing trend compared to the scenarios without drug intervention and delay. Conversely, compartments  $I_{15}(t)$, and $I_{17}(t)$, show an increasing trend.\\

Optimal drug values for individual drug administration, combination of two drugs, and MDT drug administration with two dosages at the first day and next at the 31th day  over a 60-day period are presented in tables \ref{tab:result4}, \ref{tab:result5} and \ref{tab:result6}, respectively.

\begin{table}[ht!]
   \centering
    \begin{tabular}{|c|c|c|c|c|c|}
    \hline
        \multirow{2}{*}{\textbf{Single Drug}} & \multirow{2}{*}{\textbf{Monthly dosage(mg)}} & \multicolumn{2}{c|}{\textbf{First dosage for 60 days(mg)}} &  \multicolumn{2}{c|}{\textbf{Second dosage at 31th-day(mg)}} \\    
        \cline{3-6}
         & & \textbf{Initial}  & \textbf{Optimal} & \textbf{Initial}  & \textbf{Optimal}\\
        \hline
        Rifampin & 600   & 10  & 10 & 20 & 20.013 \\
        \hline
        Dapsone & 3000 & 50 & 50.002 & 100 & 100.009 \\
        \hline
        Clofazimine & 300 & 5 & 5.001 & 10  & 10.005 \\
        \hline
\end{tabular}
\caption{Dosage levels for individual  drug administration for 60-days with delay of 30 days for second drug dosage}
\label{tab:result4}
\end{table}

\begin{table}[ht!]
   \centering
    \begin{tabular}{|c|c|c|c|c|c|}
    \hline
        \multirow{2}{*}{\textbf{Two Drugs}} & \multirow{2}{*}{\textbf{Monthly dosage(mg)}} & \multicolumn{2}{c|}{\textbf{1st dose for 60 days(mg)}} &  \multicolumn{2}{c|}{\textbf{2nd dose at 31th-day(mg)}} \\    
         \cline{3-6}
         & & \textbf{Initial(mg)} & \textbf{Optimal(mg)} & \textbf{Initial(mg)} & \textbf{Optimal(mg)} \\  
        \hline
         Rifampin and Dapsone & 600 + 3000 & 10, 50  & 10.001, 50.001 & 20, 100  & 20.049, 99.949  \\
         \hline
         Dapsone and Clofazimine & 3000 + 300 & 50, 5  & 49.998, 5.00 & 100, 10  & 99.785, 9.996\\
         \hline
          Rifampin and Clofazimine & 600 + 300 & 10, 5 & 10.001, 4.999 & 20, 10 & 19.994, 10.037 \\
         \hline
    \end{tabular}
    \caption{Dosage levels for  combination of two drugs administration for 60-days with delay of 30 days for second drug dosage }
    \label{tab:result5}
\end{table}

\begin{table}[ht!]
   \centering
    \begin{tabular}{|c|c|c|c|c|c|}
        \hline
        \multirow{2}{*}{\textbf{Three Drugs}} & \multirow{2}{*}{\textbf{Monthly dosage(mg)}} & \multicolumn{2}{c|}{\textbf{1st dose for 60 days(mg)}} &  \multicolumn{2}{c|}{\textbf{2nd dose at 31th-day(mg)}} \\    
         \cline{3-6}
         & & \textbf{Initial(mg)} & \textbf{Optimal(mg)} & \textbf{Initial(mg)} & \textbf{Optimal(mg)} \\  
         \hline
         MDT & 600+3000+300 & 10, 50, 5  & 10.001, 50.002, 5.001 & 20, 100, 10  & 19.872, 99.985, 9.939 \\
         \hline
\end{tabular}
\caption{Drug dosage levels for  three drugs administrtion in MDT for 60-days with delay of 30 days for second drug dosage}
\label{tab:result6}
\end{table}

\newpage 

\begin{figure}[htbp]
	\centering
	\begin{subfigure}{0.30\textwidth}
    	\includegraphics[width=\textwidth]{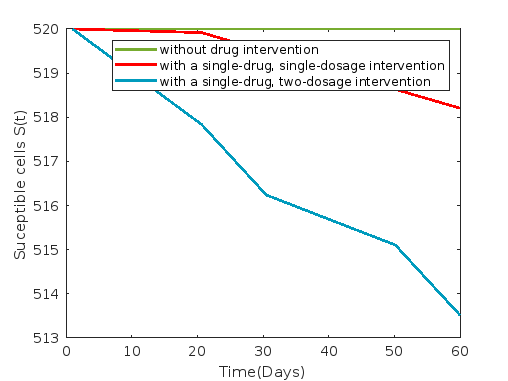}
    	\caption{Graph 1}
    	\label{fig:graph64}
	\end{subfigure}
	\hfill
	\begin{subfigure}{0.30\textwidth}
    	\includegraphics[width=\textwidth]{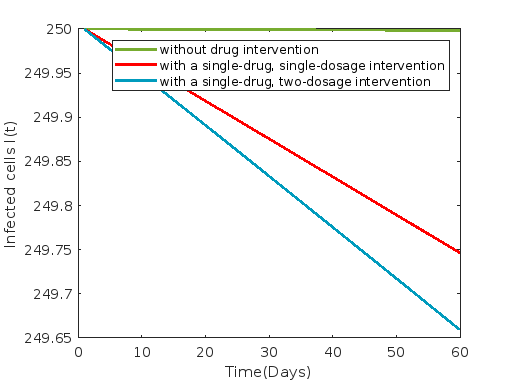}
    	\caption{Graph 2}
    	\label{fig:graph65}
	\end{subfigure}
	\hfill
	\begin{subfigure}{0.30\textwidth}
    	\includegraphics[width=\textwidth]{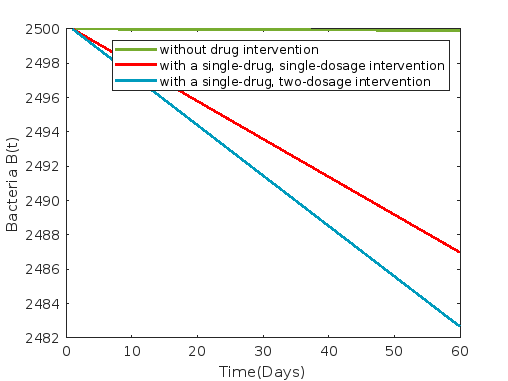}
    	\caption{Graph 3}
    	\label{fig:graph66}
	\end{subfigure}
	\hfill
	\begin{subfigure}{0.30\textwidth}
    	\includegraphics[width=\textwidth]{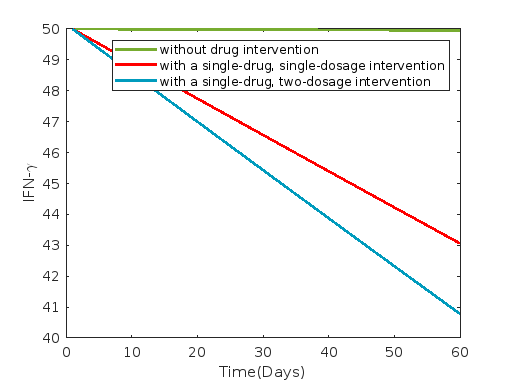}
    	\caption{Graph 4}
    	\label{fig:graph67}
	\end{subfigure}
	\hfill
	\begin{subfigure}{0.30\textwidth}
    	\includegraphics[width=\textwidth]{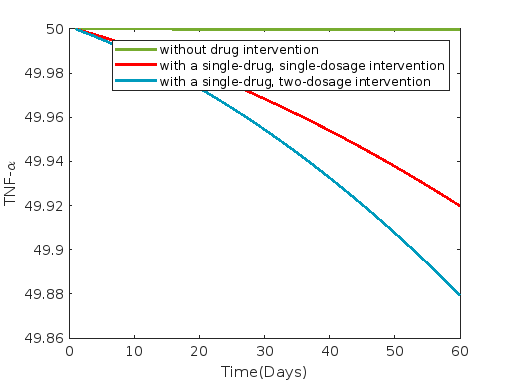}
    	\caption{Graph 5}
    	\label{fig:graph68}
	\end{subfigure}
	\hfill
	\begin{subfigure}{0.30\textwidth}
    	\includegraphics[width=\textwidth]{IL_10_R.png}
    	\caption{Graph 6}
    	\label{fig:graph69}
	\end{subfigure}
	\hfill
	\begin{subfigure}{0.30\textwidth}
    	\includegraphics[width=\textwidth]{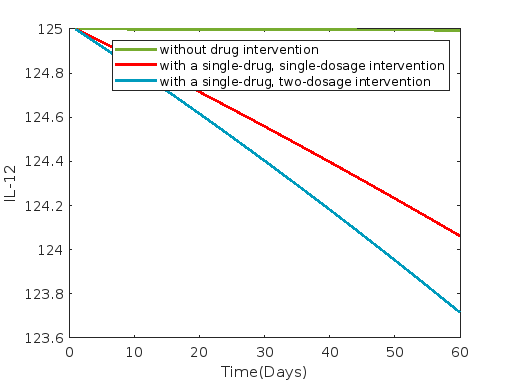}
    	\caption{Graph 7}
    	\label{fig:graph70}
	\end{subfigure}
	\hfill
	\begin{subfigure}{0.30\textwidth}
    	\includegraphics[width=\textwidth]{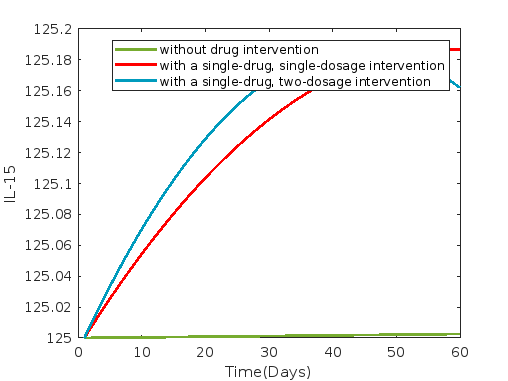}
    	\caption{Graph 8}
    	\label{fig:graph71}
	\end{subfigure}
	\hfill
	\begin{subfigure}{0.30\textwidth}
    	\includegraphics[width=\textwidth]{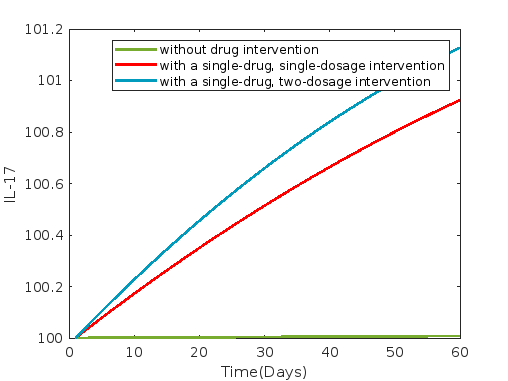}
    	\caption{Graph 9}
    	\label{fig:graph72}
	\end{subfigure}
	\hfill
    \begin{subfigure}{0.30\textwidth}
    	\includegraphics[width=\textwidth]{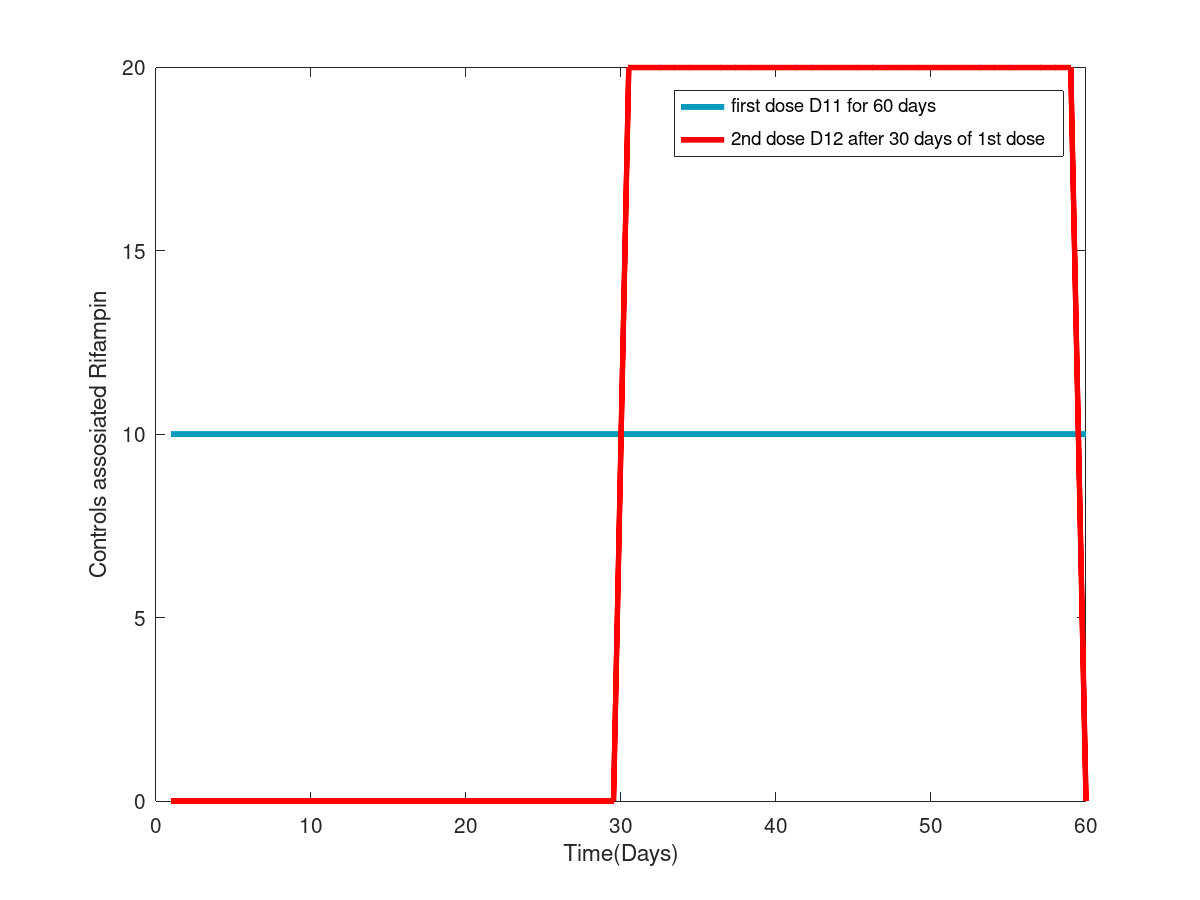}
    	\caption{Graph 10}
    	\label{fig:graph73}
	\end{subfigure}
	\hfill
	\caption{Plots depicting the influence of two dosages of rifampin over a period of 60 days with the second dose being administered on 31st day. }
	\label{fig:rifampin_dly}
\end{figure}

\begin{figure}[htbp]
	\centering
	\begin{subfigure}{0.30\textwidth}
    	\includegraphics[width=\textwidth]{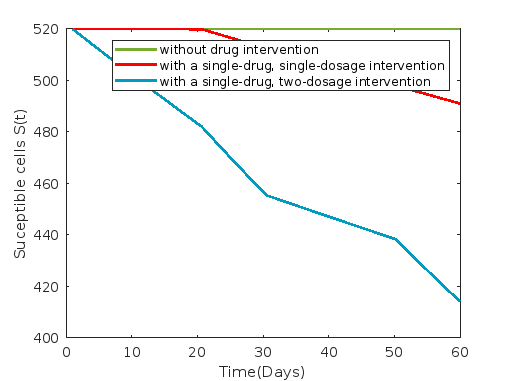}
    	\caption{Graph 1}
    	\label{fig:graph76}
	\end{subfigure}
	\hfill
	\begin{subfigure}{0.30\textwidth}
    	\includegraphics[width=\textwidth]{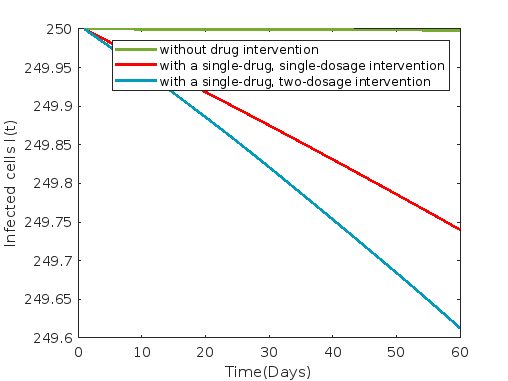}
    	\caption{Graph 2}
    	\label{fig:graph77}
	\end{subfigure}
	\hfill
	\begin{subfigure}{0.30\textwidth}
    	\includegraphics[width=\textwidth]{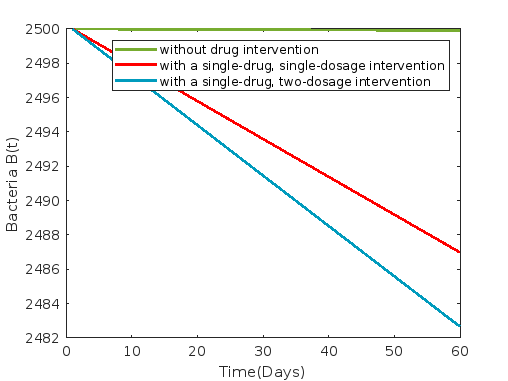}
    	\caption{Graph 3}
    	\label{fig:graph78}
	\end{subfigure}
	\hfill
	\begin{subfigure}{0.30\textwidth}
    	\includegraphics[width=\textwidth]{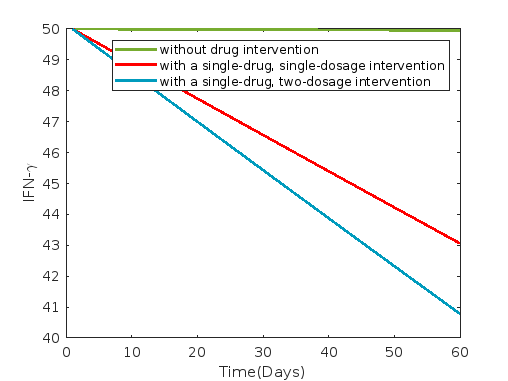}
    	\caption{Graph 4}
    	\label{fig:graph79}
	\end{subfigure}
	\hfill
	\begin{subfigure}{0.30\textwidth}
    	\includegraphics[width=\textwidth]{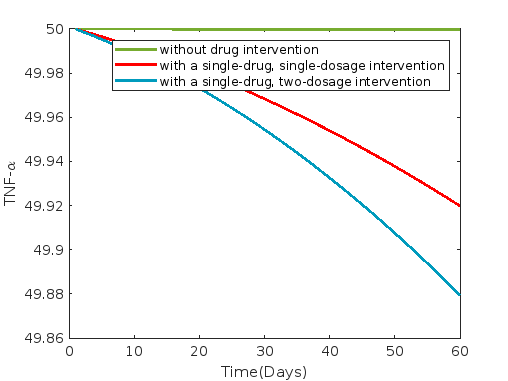}
    	\caption{Graph 5}
    	\label{fig:graph80}
	\end{subfigure}
	\hfill
	\begin{subfigure}{0.30\textwidth}
    	\includegraphics[width=\textwidth]{IL_10_D.png}
    	\caption{Graph 6}
    	\label{fig:graph81}
	\end{subfigure}
	\hfill
	\begin{subfigure}{0.30\textwidth}
    	\includegraphics[width=\textwidth]{IL_12_D.png}
    	\caption{Graph 7}
    	\label{fig:graph82}
	\end{subfigure}
	\hfill
	\begin{subfigure}{0.30\textwidth}
    	\includegraphics[width=\textwidth]{IL_15_D.png}
    	\caption{Graph 8}
    	\label{fig:graph83}
	\end{subfigure}
	\hfill
	\begin{subfigure}{0.30\textwidth}
    	\includegraphics[width=\textwidth]{IL_17_D.png}
    	\caption{Graph 9}
    	\label{fig:graph84}
	\end{subfigure}
	\hfill
    \begin{subfigure}{0.30\textwidth}
    	\includegraphics[width=\textwidth]{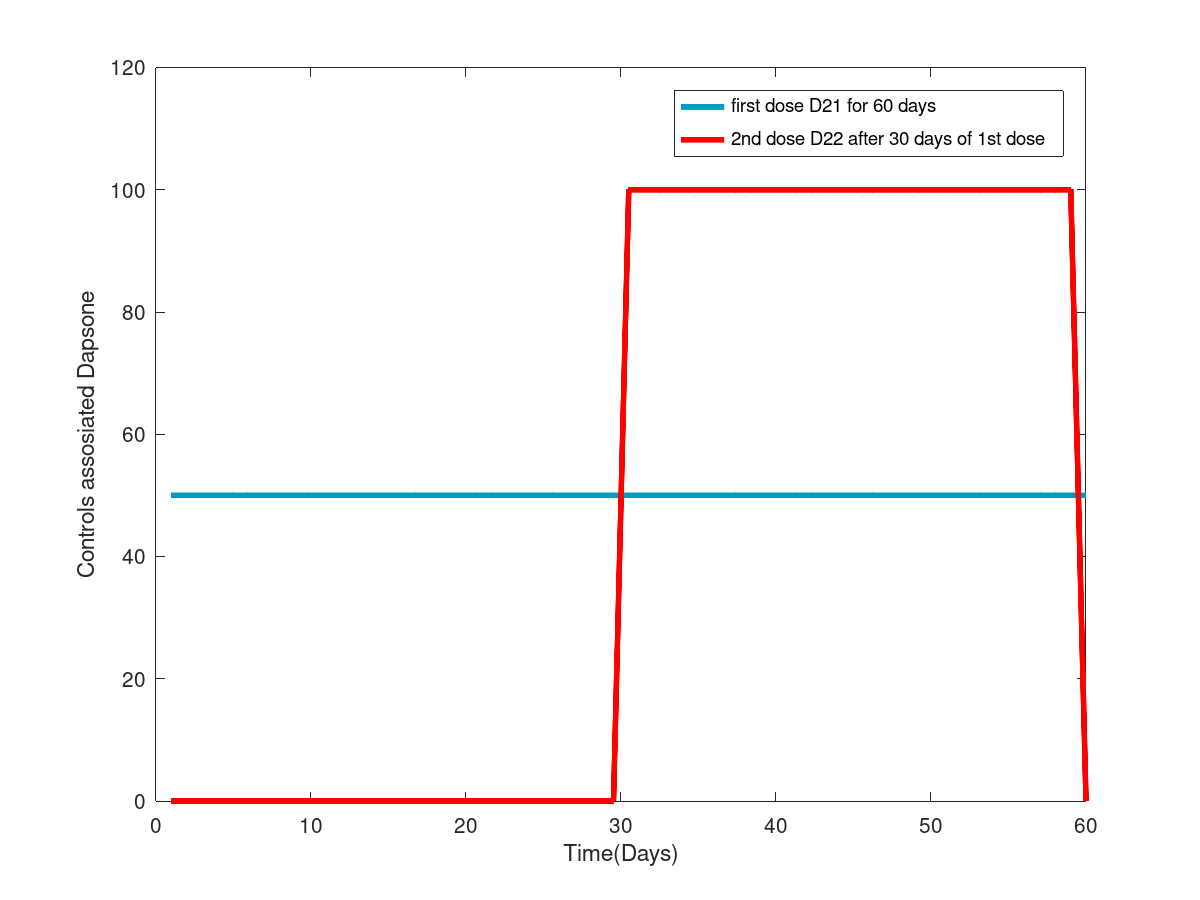}
    	\caption{Graph 10}
    	\label{fig:graph85}
	\end{subfigure}
	\hfill
	\caption{Plots depicting the influence of two dosages of dapsone over a period of 60 days with the second dose being administered on 31st day.}
 	\label{fig:dapsone_dly}
\end{figure}

\begin{figure}[htbp]
	\centering
	\begin{subfigure}{0.30\textwidth}
    	\includegraphics[width=\textwidth]{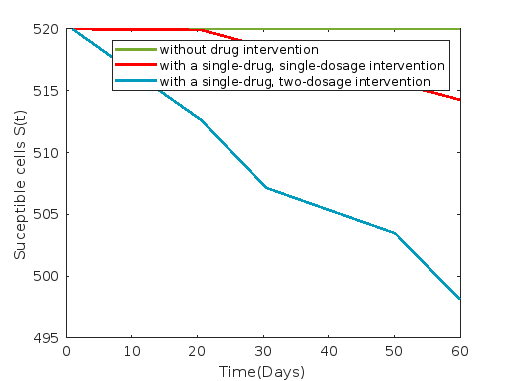}
    	\caption{Graph 1}
    	\label{fig:graph88}
	\end{subfigure}
	\hfill
	\begin{subfigure}{0.30\textwidth}
    	\includegraphics[width=\textwidth]{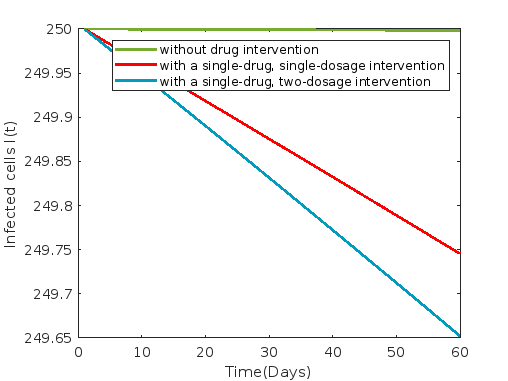}
    	\caption{Graph 2}
    	\label{fig:graph89}
	\end{subfigure}
	\hfill
	\begin{subfigure}{0.30\textwidth}
    	\includegraphics[width=\textwidth]{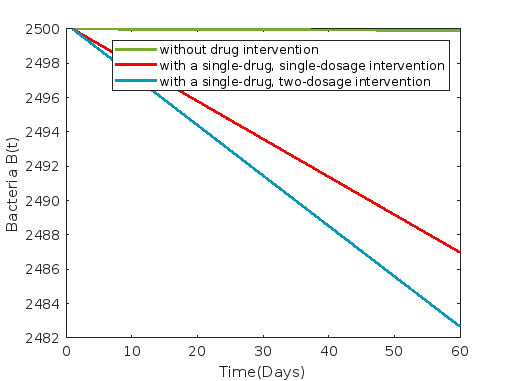}
    	\caption{Graph 3}
    	\label{fig:graph90}
	\end{subfigure}
	\hfill
	\begin{subfigure}{0.30\textwidth}
    	\includegraphics[width=\textwidth]{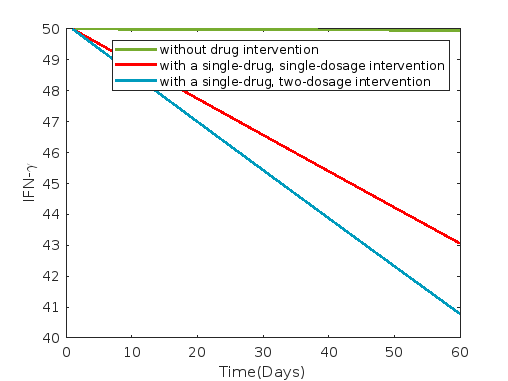}
    	\caption{Graph 4}
    	\label{fig:graph91}
	\end{subfigure}
	\hfill
	\begin{subfigure}{0.30\textwidth}
    	\includegraphics[width=\textwidth]{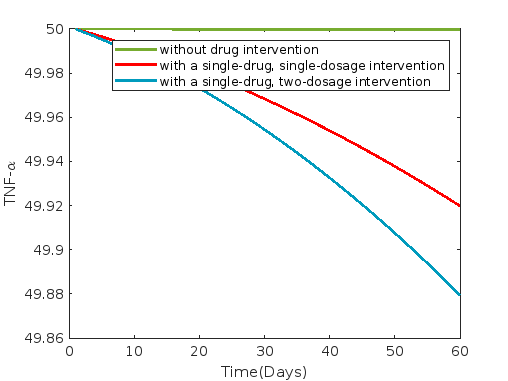}
    	\caption{Graph 5}
    	\label{fig:graph92}
	\end{subfigure}
	\hfill
	\begin{subfigure}{0.30\textwidth}
    	\includegraphics[width=\textwidth]{IL_10_C.png}
    	\caption{Graph 6}
    	\label{fig:graph93}
	\end{subfigure}
	\hfill
	\begin{subfigure}{0.30\textwidth}
    	\includegraphics[width=\textwidth]{IL_12_C.png}
    	\caption{Graph 7}
    	\label{fig:graph94}
	\end{subfigure}
	\hfill
	\begin{subfigure}{0.30\textwidth}
    	\includegraphics[width=\textwidth]{IL_15_C.png}
    	\caption{Graph 8}
    	\label{fig:graph95}
	\end{subfigure}
	\hfill
	\begin{subfigure}{0.30\textwidth}
    	\includegraphics[width=\textwidth]{IL_17_C.png}
    	\caption{Graph 9}
    	\label{fig:graph96}
	\end{subfigure}
	\hfill
    \begin{subfigure}{0.30\textwidth}
    	\includegraphics[width=\textwidth]{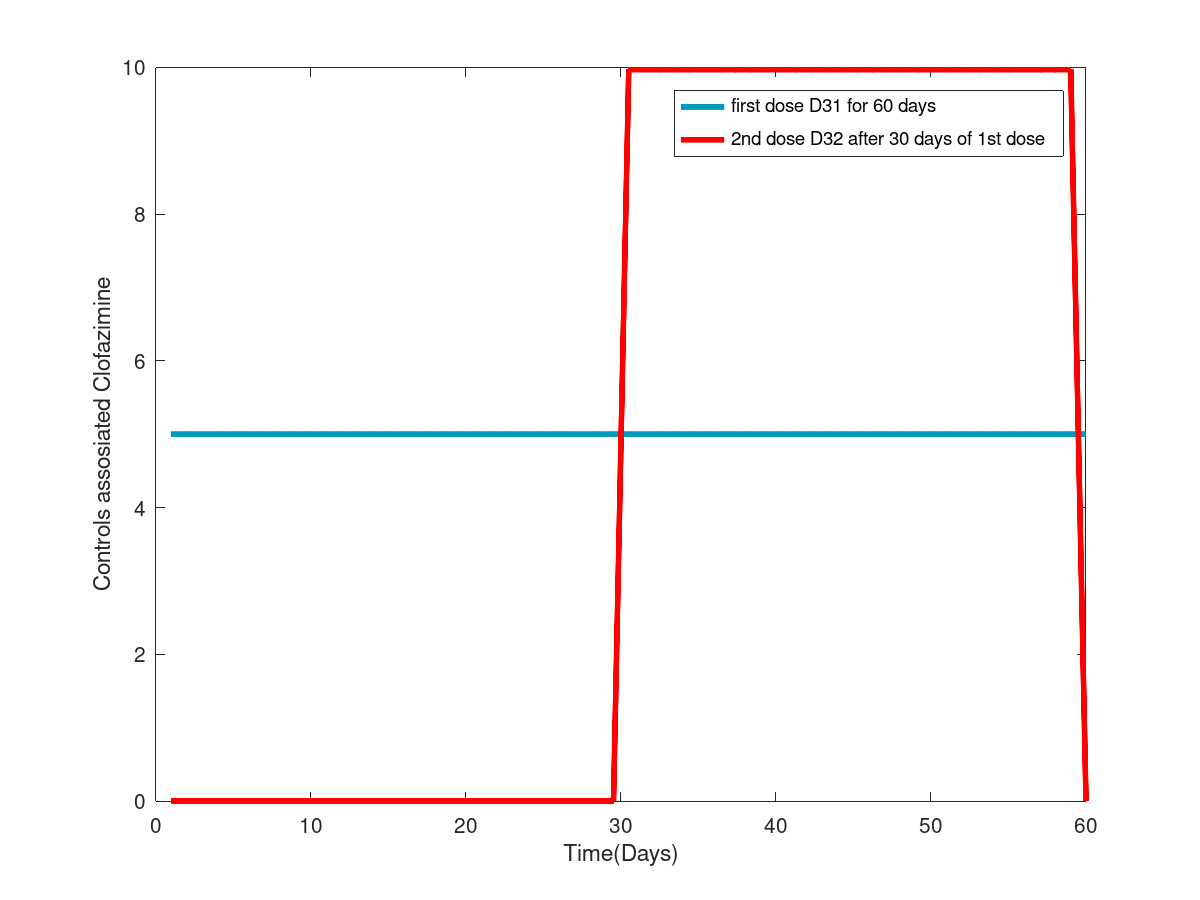}
    	\caption{Graph 10}
    	\label{fig:graph97}
	\end{subfigure}
	\hfill
	\caption{Plots depicting the influence of two dosages of clofazimine over a period of 60 days with the second dose being administered on 31st day. }
	\label{fig:clofazimine_dly}
\end{figure}

\begin{figure}[htbp]
    \centering
    \begin{subfigure}{0.30\textwidth}
        \includegraphics[width=\textwidth]{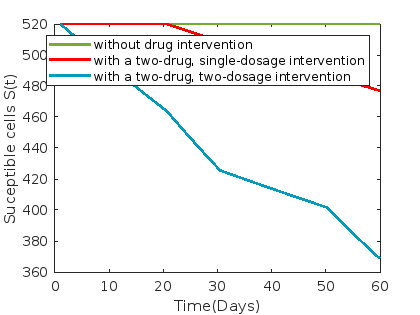}
        \caption{Graph 1}
        \label{fig:graph100}
    \end{subfigure}
    \hfill
    \begin{subfigure}{0.30\textwidth}
        \includegraphics[width=\textwidth]{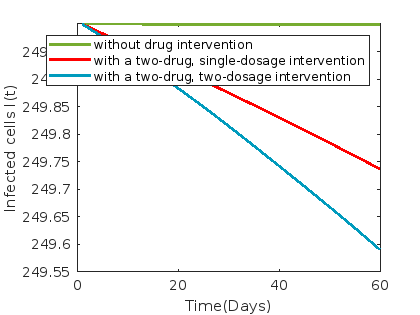}
        \caption{Graph 2}
        \label{fig:graph101}
    \end{subfigure}
    \hfill
    \begin{subfigure}{0.30\textwidth}
        \includegraphics[width=\textwidth]{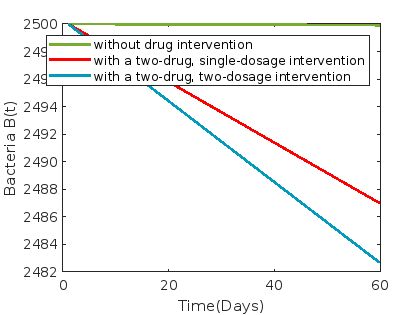}
        \caption{Graph 3}
        \label{fig:graph102}
    \end{subfigure}
    \hfill
    \begin{subfigure}{0.30\textwidth}
        \includegraphics[width=\textwidth]{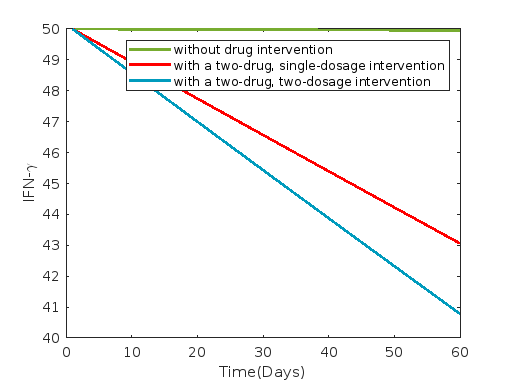}
        \caption{Graph 4}
        \label{fig:graph103}
    \end{subfigure}
    \hfill
    \begin{subfigure}{0.30\textwidth}
        \includegraphics[width=\textwidth]{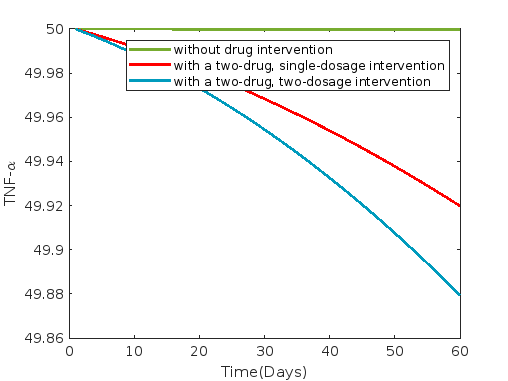}
        \caption{Graph 5}
        \label{fig:graph104}
    \end{subfigure}
    \hfill
    \begin{subfigure}{0.30\textwidth}
        \includegraphics[width=\textwidth]{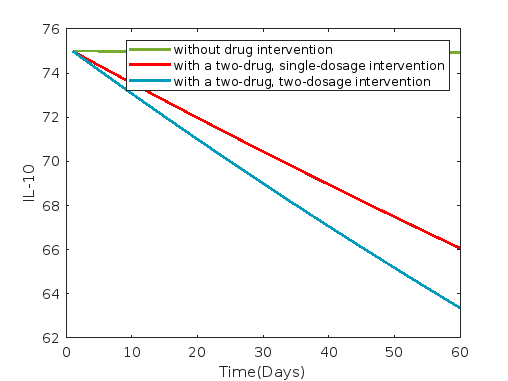}
        \caption{Graph 6}
        \label{fig:graph105}
    \end{subfigure}
    \hfill
    \begin{subfigure}{0.30\textwidth}
        \includegraphics[width=\textwidth]{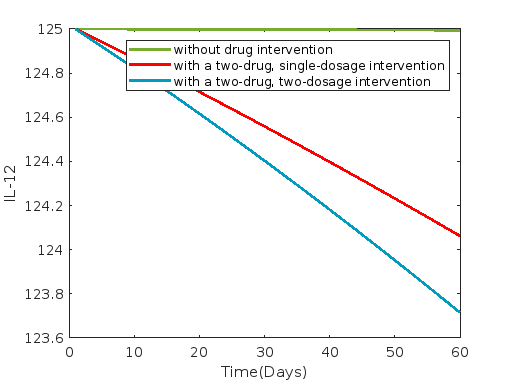}
        \caption{Graph 7}
        \label{fig:graph106}
    \end{subfigure}
    \hfill
    \begin{subfigure}{0.30\textwidth}
        \includegraphics[width=\textwidth]{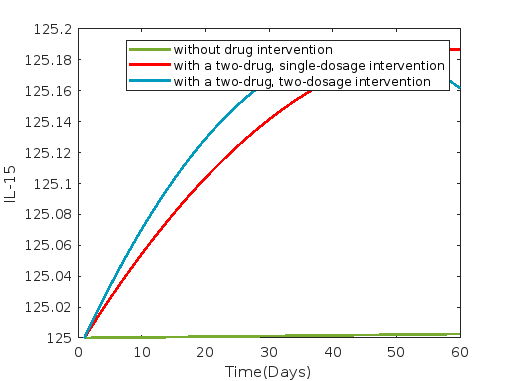}
        \caption{Graph 8}
        \label{fig:graph107}
    \end{subfigure}
    \hfill
    \begin{subfigure}{0.30\textwidth}
        \includegraphics[width=\textwidth]{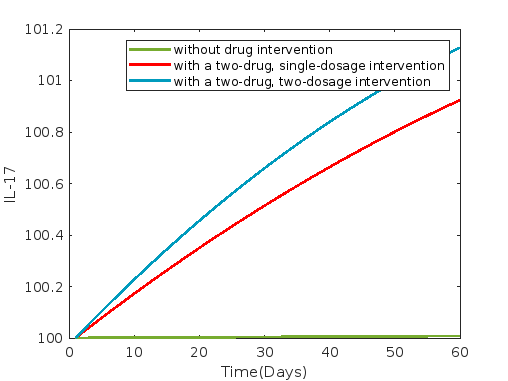}
        \caption{Graph 9}
        \label{fig:graph108}
    \end{subfigure}
    \hfill
    \begin{subfigure}{0.30\textwidth}
        \includegraphics[width=\textwidth]{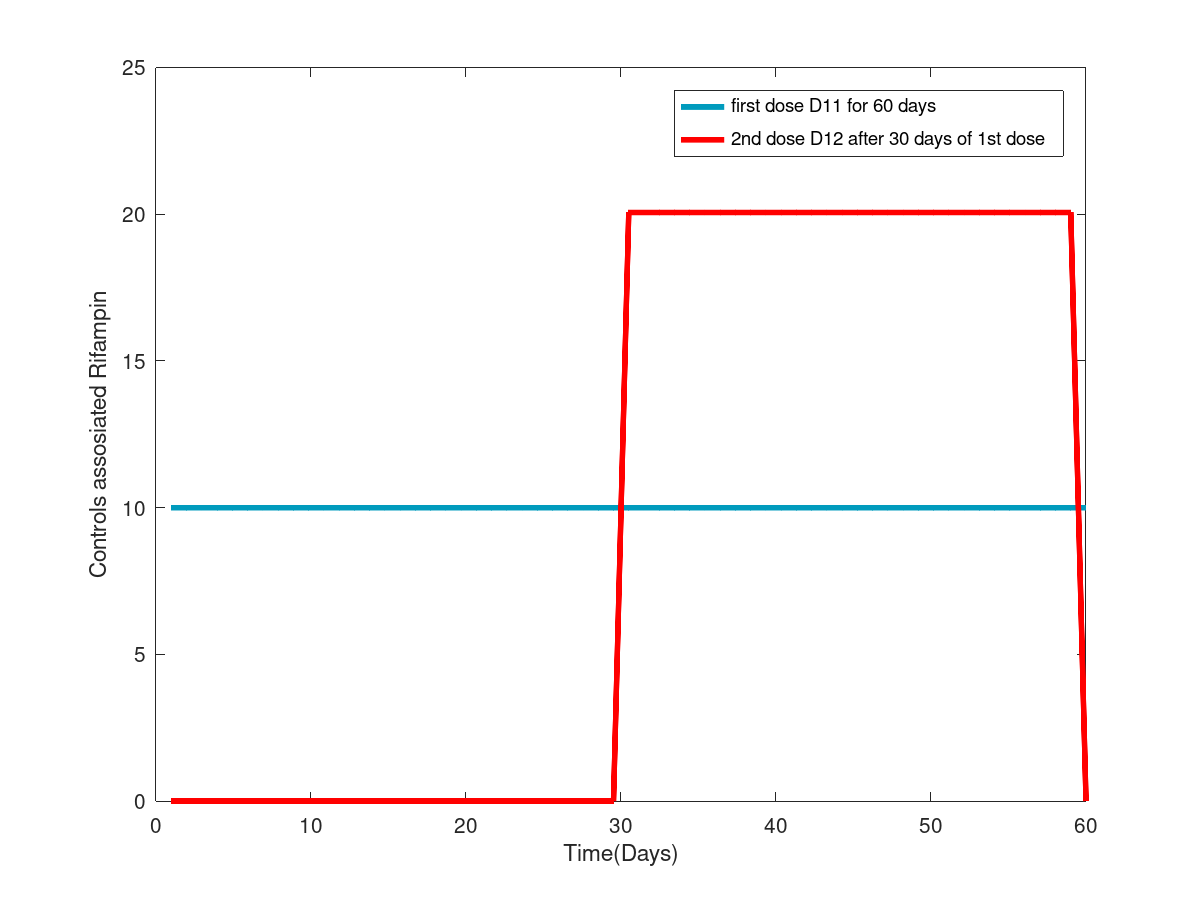}
        \caption{Graph 10}
        \label{fig:graph109}
    \end{subfigure}
    \hfill
    \begin{subfigure}{0.30\textwidth}
        \includegraphics[width=\textwidth]{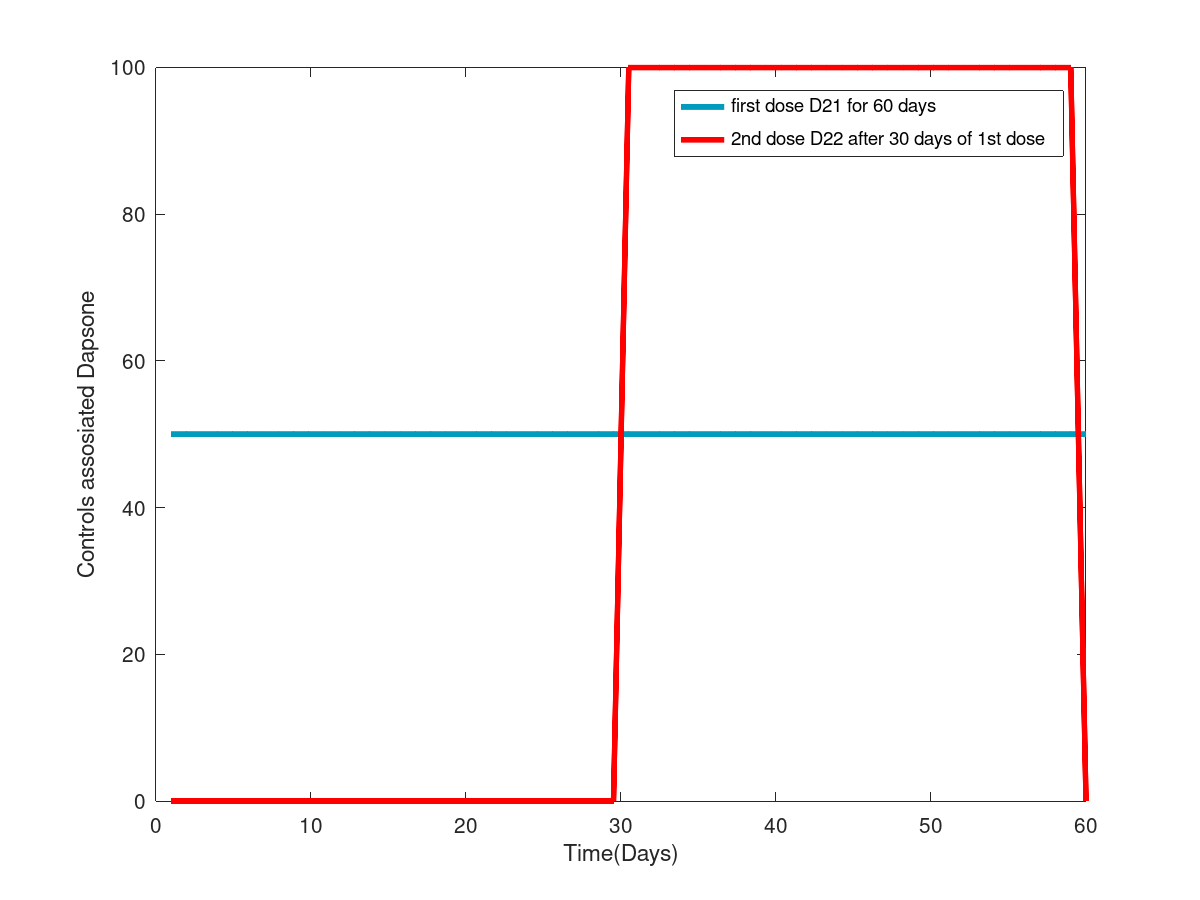}
        \caption{Graph 12}
        \label{fig:graph111}
    \end{subfigure}
    \hfill
    \caption{Plots depicting the combined influence of two dosages of rifampin and dapsone over a period of 60 days with the second dose being administered on 31st day. }
    \label{fig:Rif&Dap_dly}
\end{figure}

\begin{figure}[htbp]
    \centering
    \begin{subfigure}{0.30\textwidth}
        \includegraphics[width=\textwidth]{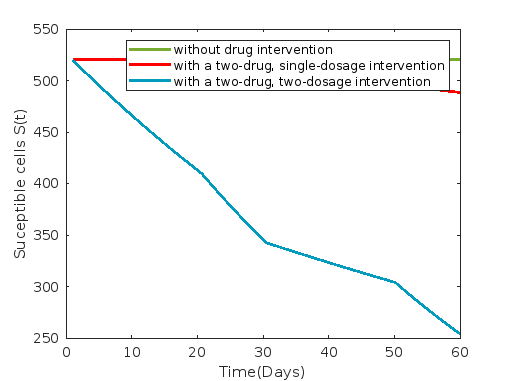}
        \caption{Graph 1}
        \label{fig:graph113}
    \end{subfigure}
    \hfill
    \begin{subfigure}{0.30\textwidth}
        \includegraphics[width=\textwidth]{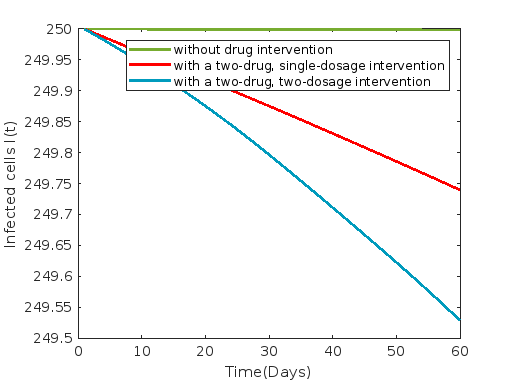}
        \caption{Graph 2}
        \label{fig:graph114}
    \end{subfigure}
    \hfill
    \begin{subfigure}{0.30\textwidth}
        \includegraphics[width=\textwidth]{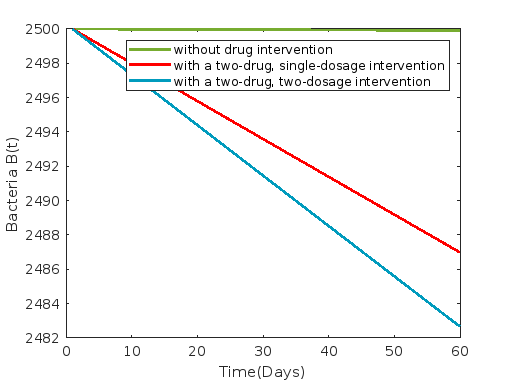}
        \caption{Graph 3}
        \label{fig:graph115}
    \end{subfigure}
    \hfill
    \begin{subfigure}{0.30\textwidth}
        \includegraphics[width=\textwidth]{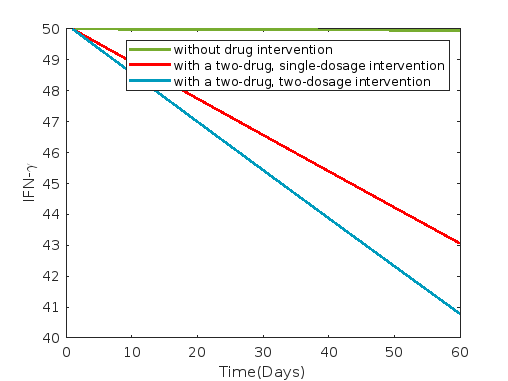}
        \caption{Graph 4}
        \label{fig:graph116}
    \end{subfigure}
    \hfill
    \begin{subfigure}{0.30\textwidth}
        \includegraphics[width=\textwidth]{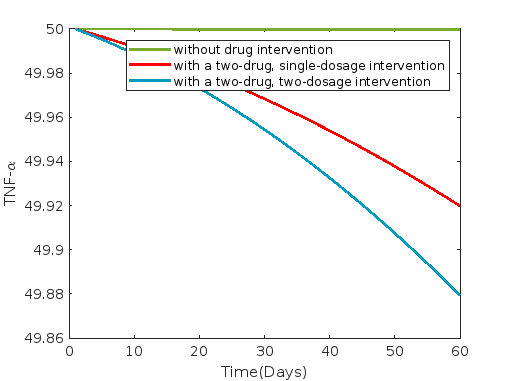}
        \caption{Graph 5}
        \label{fig:graph117}
    \end{subfigure}
    \hfill
    \begin{subfigure}{0.30\textwidth}
        \includegraphics[width=\textwidth]{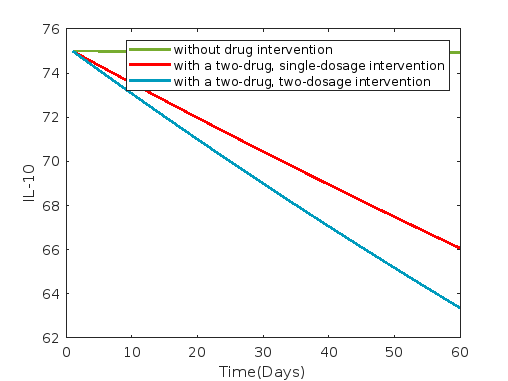}
        \caption{Graph 6}
        \label{fig:graph118}
    \end{subfigure}
    \hfill
    \begin{subfigure}{0.30\textwidth}
        \includegraphics[width=\textwidth]{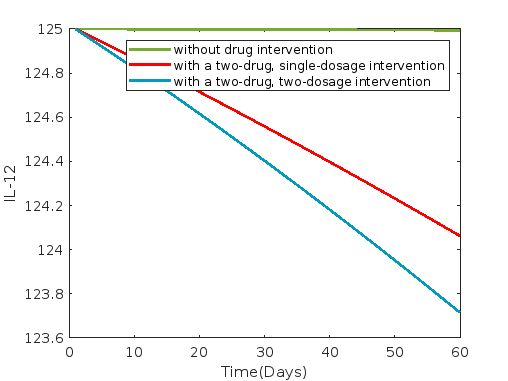}
        \caption{Graph 7}
        \label{fig:graph119}
    \end{subfigure}
    \hfill
    \begin{subfigure}{0.30\textwidth}
        \includegraphics[width=\textwidth]{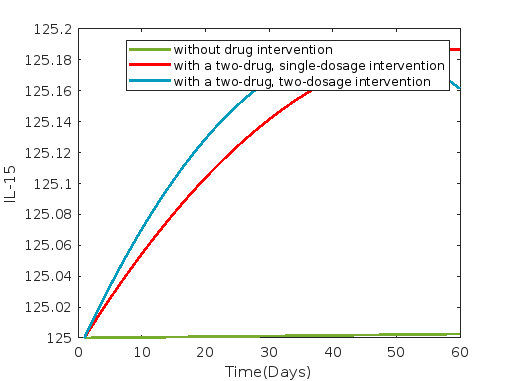}
        \caption{Graph 8}
        \label{fig:graph120}
    \end{subfigure}
    \hfill
    \begin{subfigure}{0.30\textwidth}
        \includegraphics[width=\textwidth]{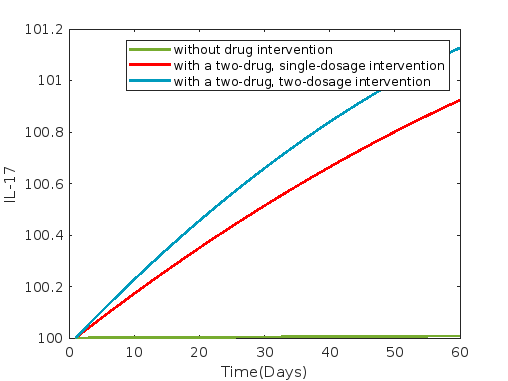}
        \caption{Graph 9}
        \label{fig:graph121}
    \end{subfigure}
    \hfill
    \begin{subfigure}{0.30\textwidth}
        \includegraphics[width=\textwidth]{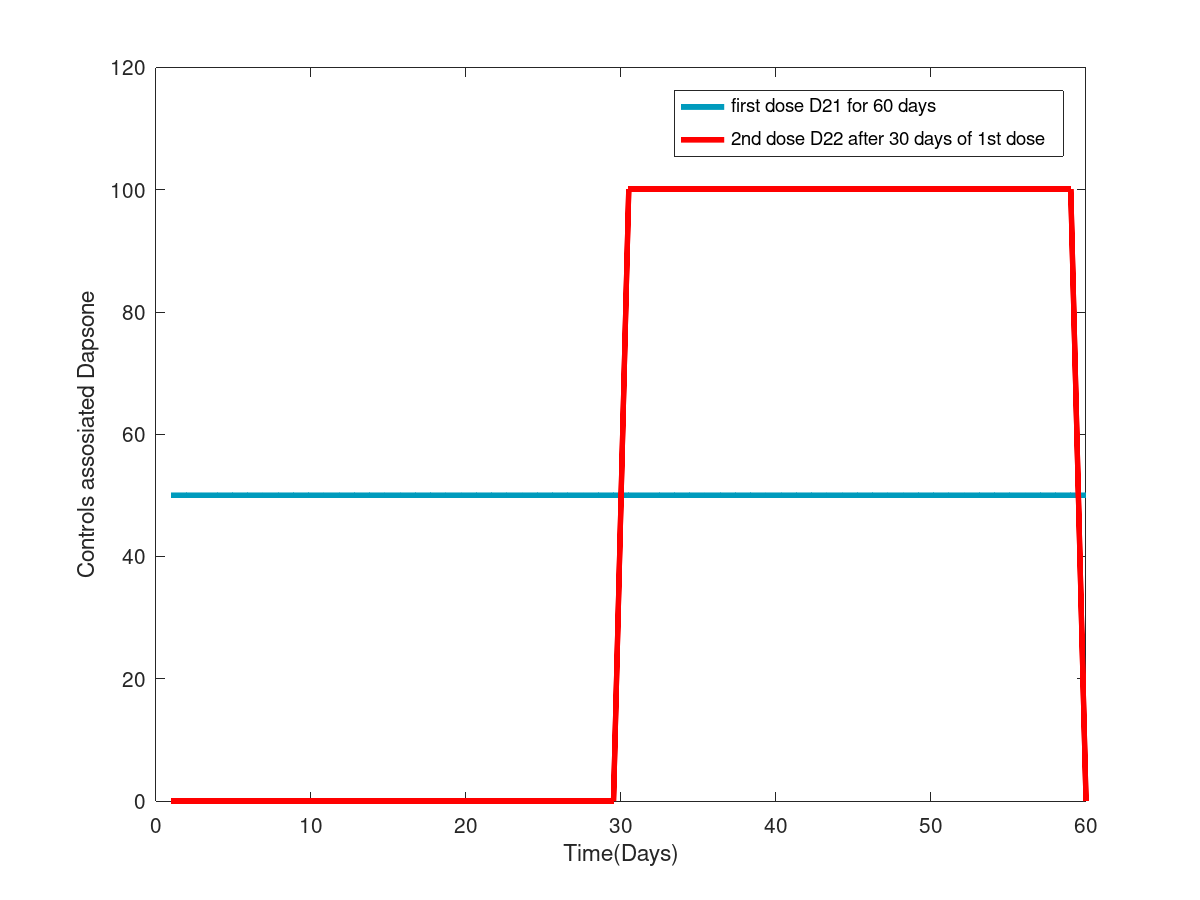}
        \caption{Graph 10}
        \label{fig:graph122}
    \end{subfigure}
    \hfill
    \begin{subfigure}{0.30\textwidth}
        \includegraphics[width=\textwidth]{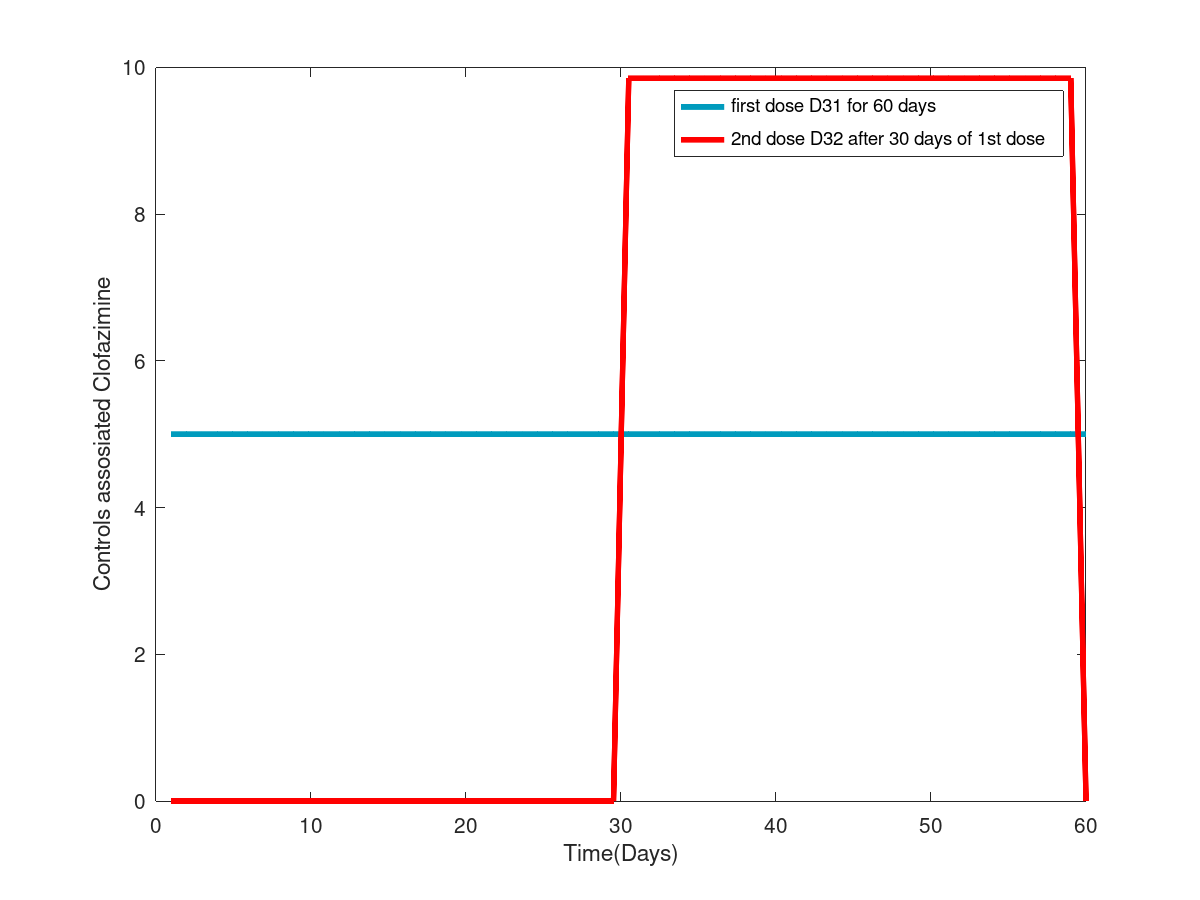}
        \caption{Graph 12}
        \label{fig:graph124}
    \end{subfigure}
    \hfill
    \caption{Plots depicting the combined influence of two dosages of dapsone and clofazimine over a period of 60 days with the second dose being administered on 31st day.} 
     \label{fig:Dap&Clo_dly}
\end{figure}

\begin{figure}[htbp]
    \centering
    \begin{subfigure}{0.30\textwidth}
        \includegraphics[width=\textwidth]{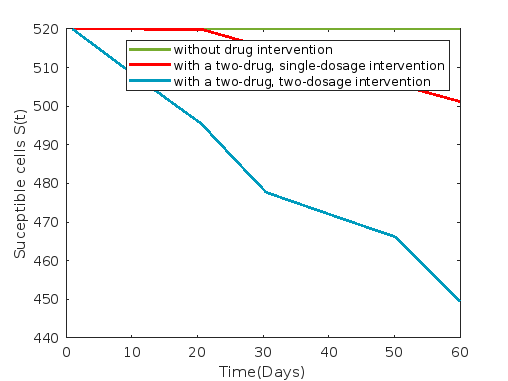}
        \caption{Graph 1}
        \label{fig:graph126}
    \end{subfigure}
    \hfill
    \begin{subfigure}{0.30\textwidth}
        \includegraphics[width=\textwidth]{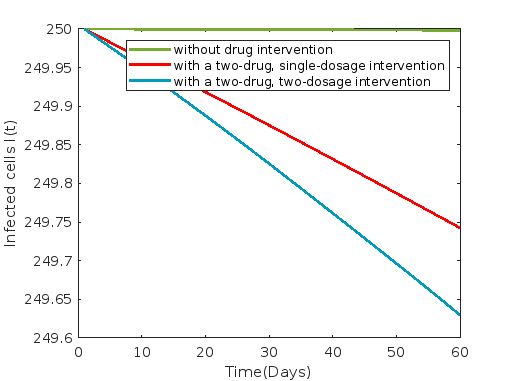}
        \caption{Graph 2}
        \label{fig:graph127}
    \end{subfigure}
    \hfill
    \begin{subfigure}{0.30\textwidth}
        \includegraphics[width=\textwidth]{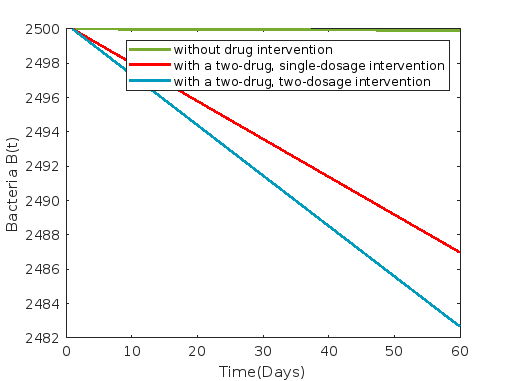}
        \caption{Graph 3}
        \label{fig:graph128}
    \end{subfigure}
    \hfill
    \begin{subfigure}{0.30\textwidth}
        \includegraphics[width=\textwidth]{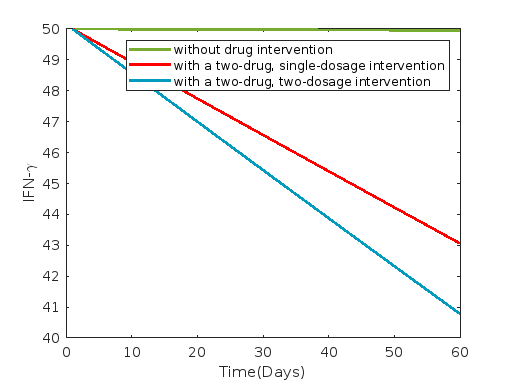}
        \caption{Graph 4}
        \label{fig:graph129}
    \end{subfigure}
    \hfill
    \begin{subfigure}{0.30\textwidth}
        \includegraphics[width=\textwidth]{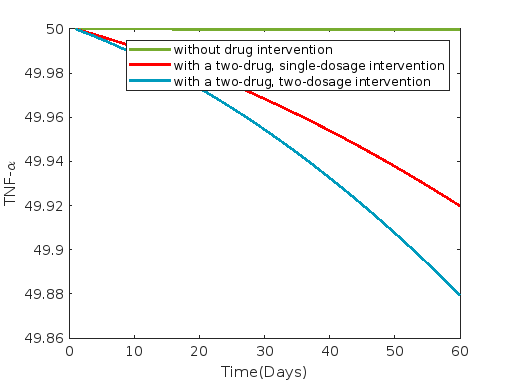}
        \caption{Graph 5}
        \label{fig:graph130}
    \end{subfigure}
    \hfill
    \begin{subfigure}{0.30\textwidth}
        \includegraphics[width=\textwidth]{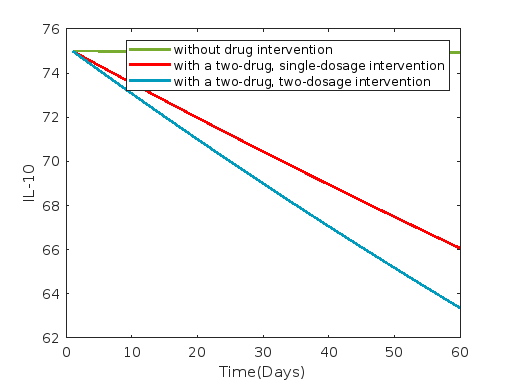}
        \caption{Graph 6}
        \label{fig:graph131}
    \end{subfigure}
    \hfill
    \begin{subfigure}{0.30\textwidth}
        \includegraphics[width=\textwidth]{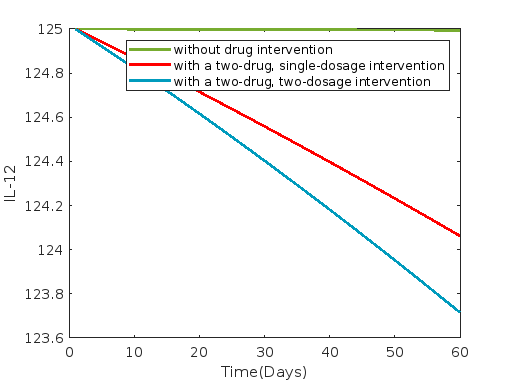}
        \caption{Graph 7}
        \label{fig:graph132}
    \end{subfigure}
    \hfill
    \begin{subfigure}{0.30\textwidth}
        \includegraphics[width=\textwidth]{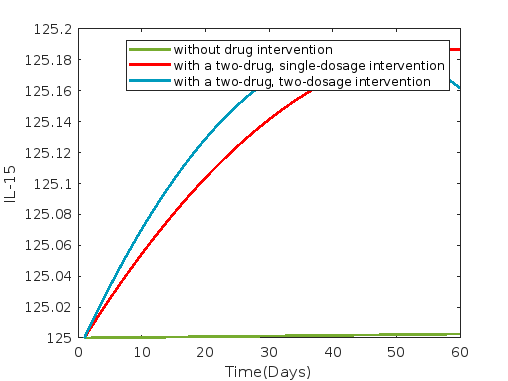}
        \caption{Graph 8}
        \label{fig:graph133}
    \end{subfigure}
    \hfill
    \begin{subfigure}{0.30\textwidth}
        \includegraphics[width=\textwidth]{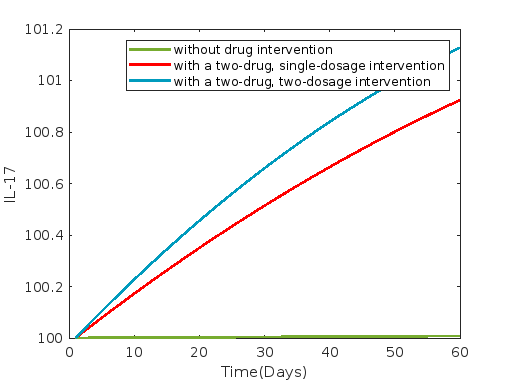}
        \caption{Graph 9}
        \label{fig:graph134}
    \end{subfigure}
    \hfill
     \begin{subfigure}{0.30\textwidth}
        \includegraphics[width=\textwidth]{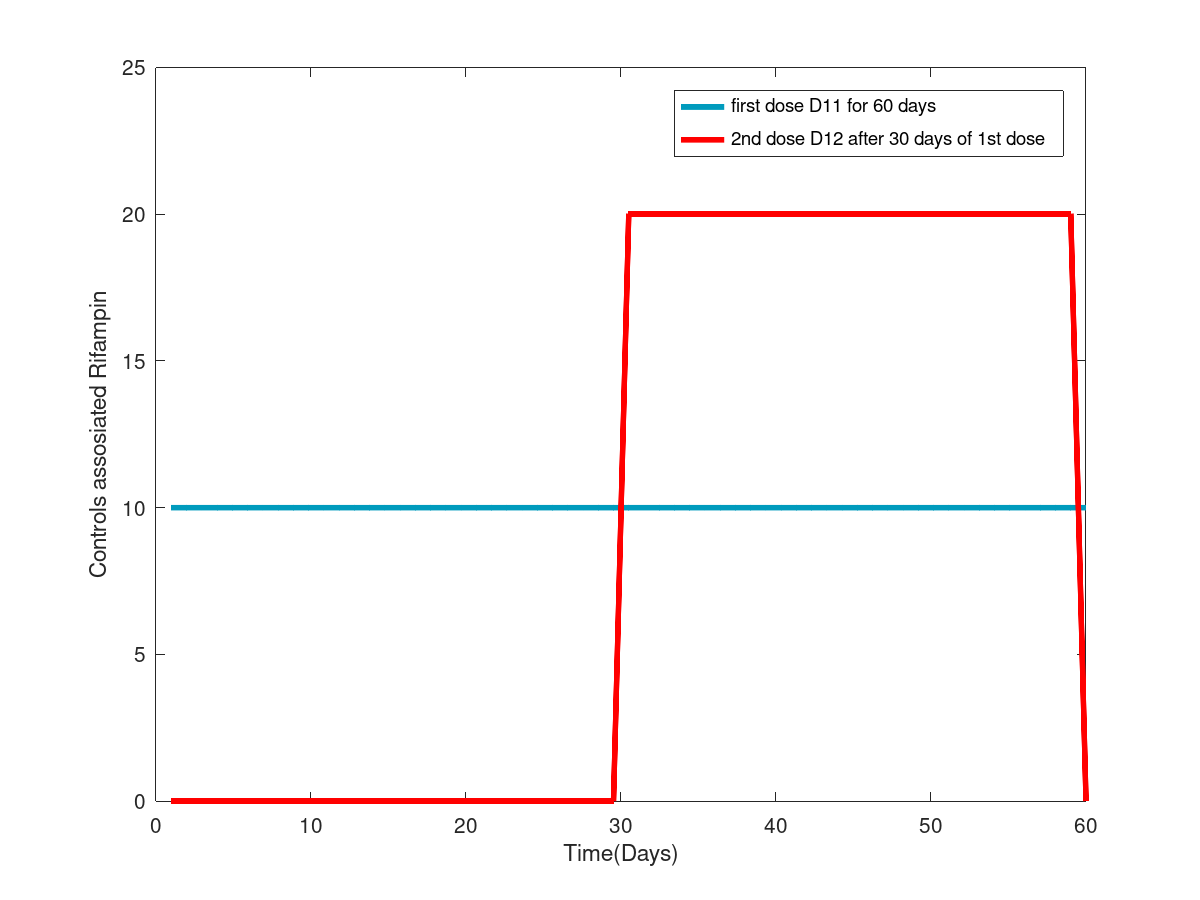}
        \caption{Graph 10}
        \label{fig:graph135}
    \end{subfigure}
    \hfill
     \begin{subfigure}{0.30\textwidth}
        \includegraphics[width=\textwidth]{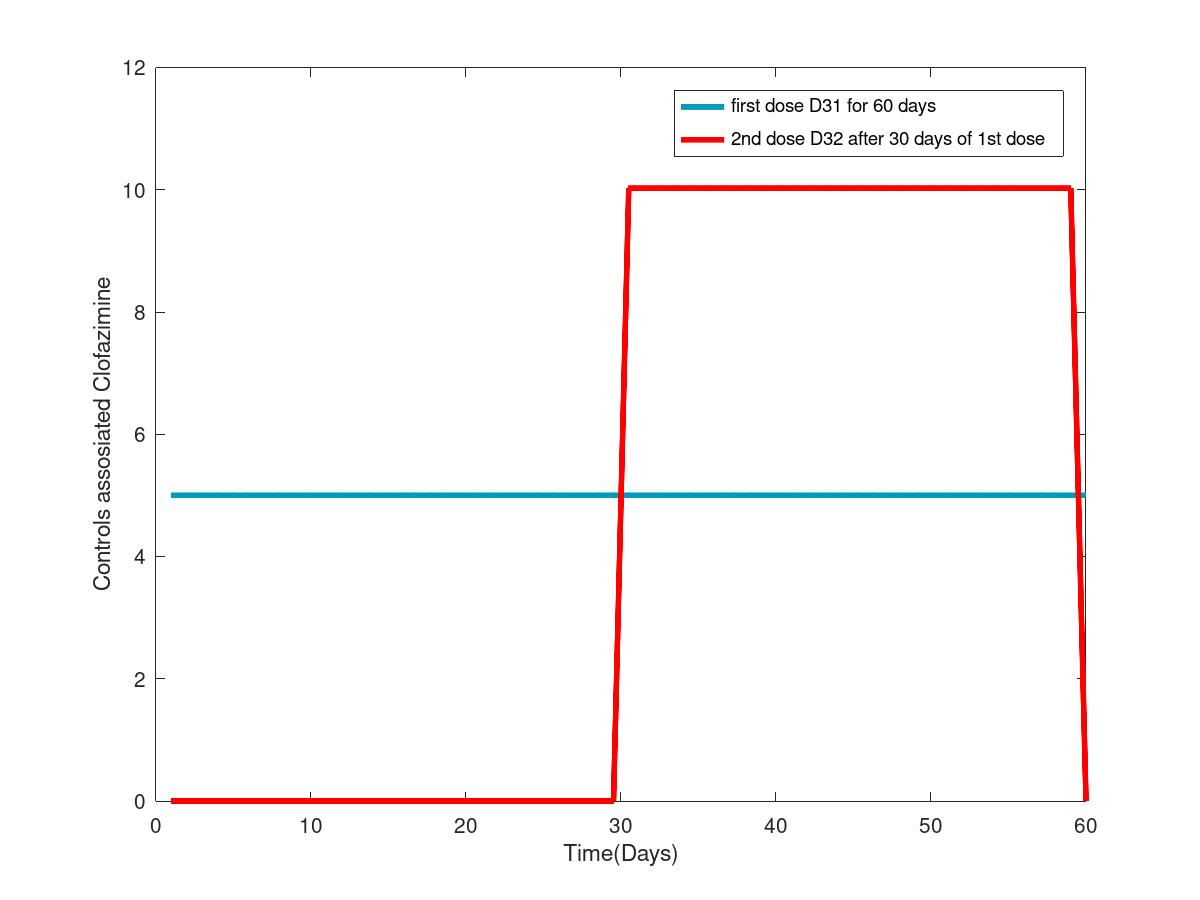}
        \caption{Graph 12}
        \label{fig:graph137}
    \end{subfigure}
    \hfill
    \caption{Plots depicting the combined influence of two dosages of    clofazimine and rifampin over a period of 60 days with the second dose being administered on 31st day. }
    \label{fig:Clo&Rif_dly}
\end{figure}

\begin{figure}[htbp]
    \centering
    \begin{subfigure}{0.30\textwidth}
        \includegraphics[width=\textwidth]{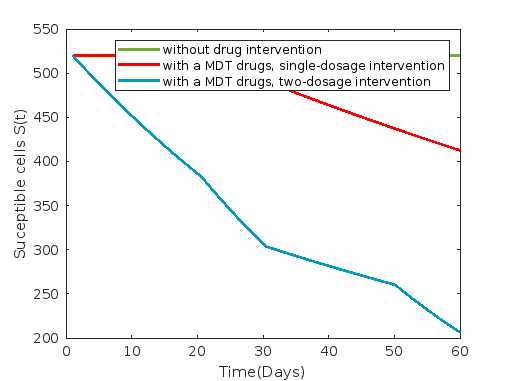}
        \caption{Graph 1}
        \label{fig:graph139}
    \end{subfigure}
    \hfill
    \begin{subfigure}{0.30\textwidth}
        \includegraphics[width=\textwidth]{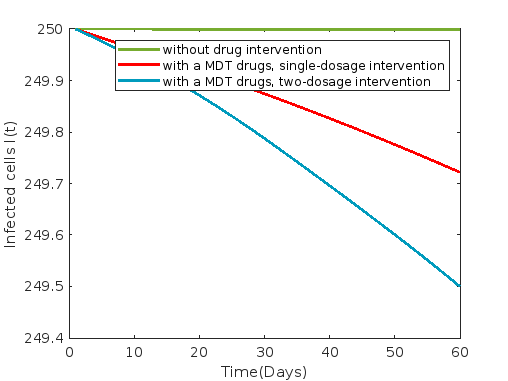}
        \caption{Graph 2}
        \label{fig:graph140}
    \end{subfigure}
    \hfill
    \begin{subfigure}{0.30\textwidth}
        \includegraphics[width=\textwidth]{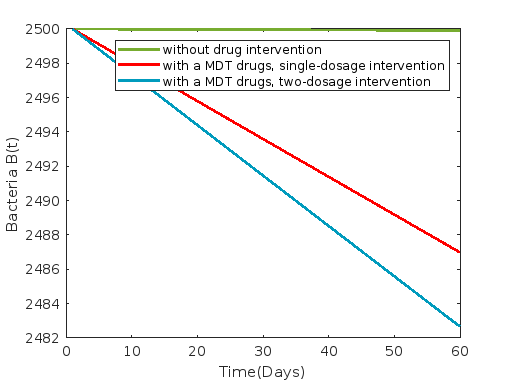}
        \caption{Graph 3}
        \label{fig:graph141}
    \end{subfigure}
    \hfill
    \begin{subfigure}{0.30\textwidth}
        \includegraphics[width=\textwidth]{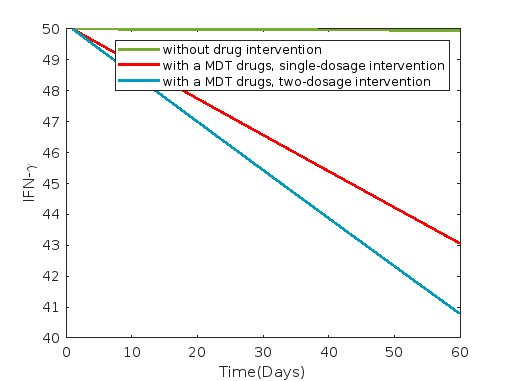}
        \caption{Graph 4}
        \label{fig:graph142}
    \end{subfigure}
    \hfill
    \begin{subfigure}{0.30\textwidth}
        \includegraphics[width=\textwidth]{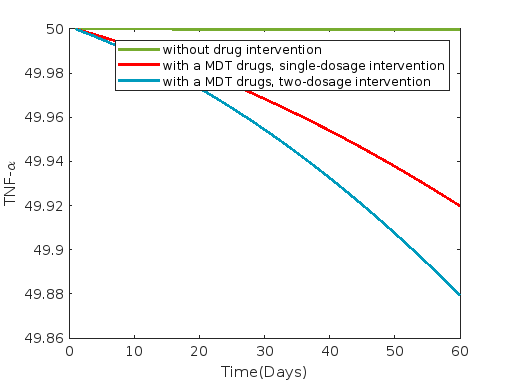}
        \caption{Graph 5}
        \label{fig:graph143}
    \end{subfigure}
    \hfill
    \begin{subfigure}{0.30\textwidth}
        \includegraphics[width=\textwidth]{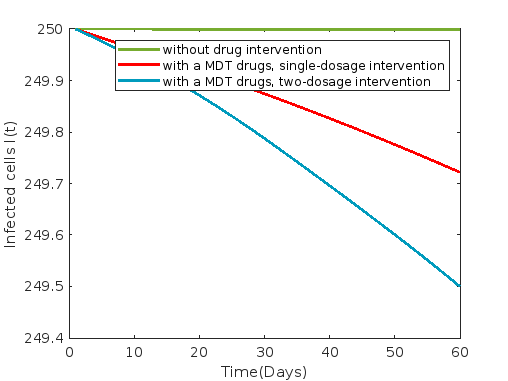}
        \caption{Graph 6}
        \label{fig:graph144}
    \end{subfigure}
    \hfill
    \begin{subfigure}{0.30\textwidth}
        \includegraphics[width=\textwidth]{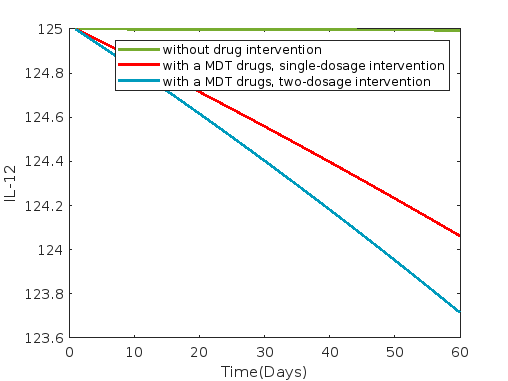}
        \caption{Graph 7}
        \label{fig:graph145}
    \end{subfigure}
    \hfill
    \begin{subfigure}{0.30\textwidth}
        \includegraphics[width=\textwidth]{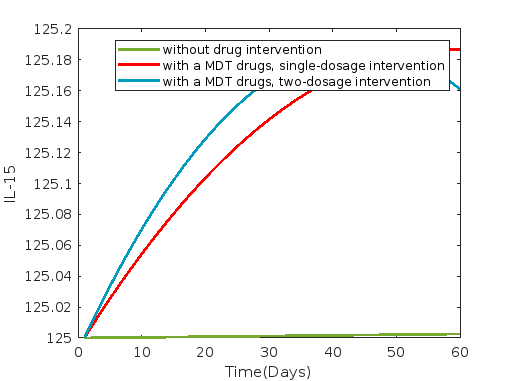}
        \caption{Graph 8}
        \label{fig:graph146}
    \end{subfigure}
    \hfill
    \begin{subfigure}{0.30\textwidth}
        \includegraphics[width=\textwidth]{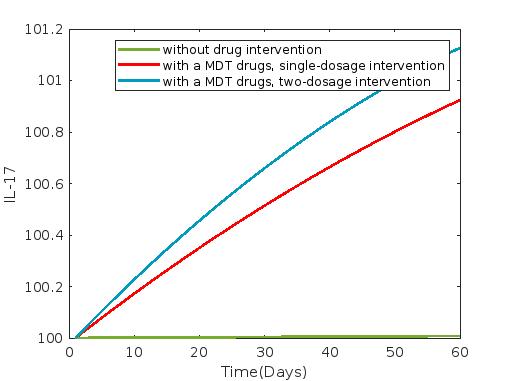}
        \caption{Graph 9}
        \label{fig:graph147}
    \end{subfigure}
    \hfill
    \begin{subfigure}{0.30\textwidth}
        \includegraphics[width=\textwidth]{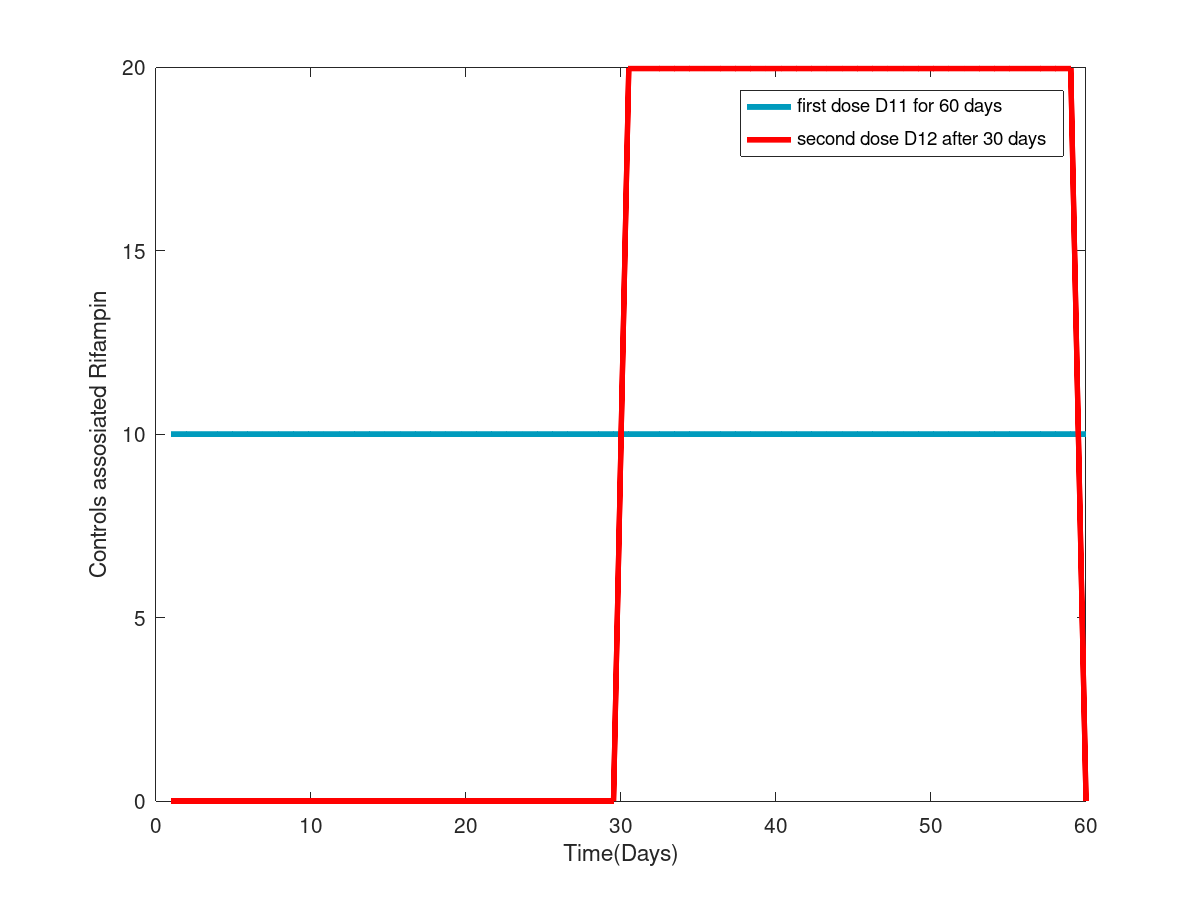}
        \caption{Graph 10}
        \label{fig:graph148}
    \end{subfigure}
    \hfill
    \begin{subfigure}{0.30\textwidth}
        \includegraphics[width=\textwidth]{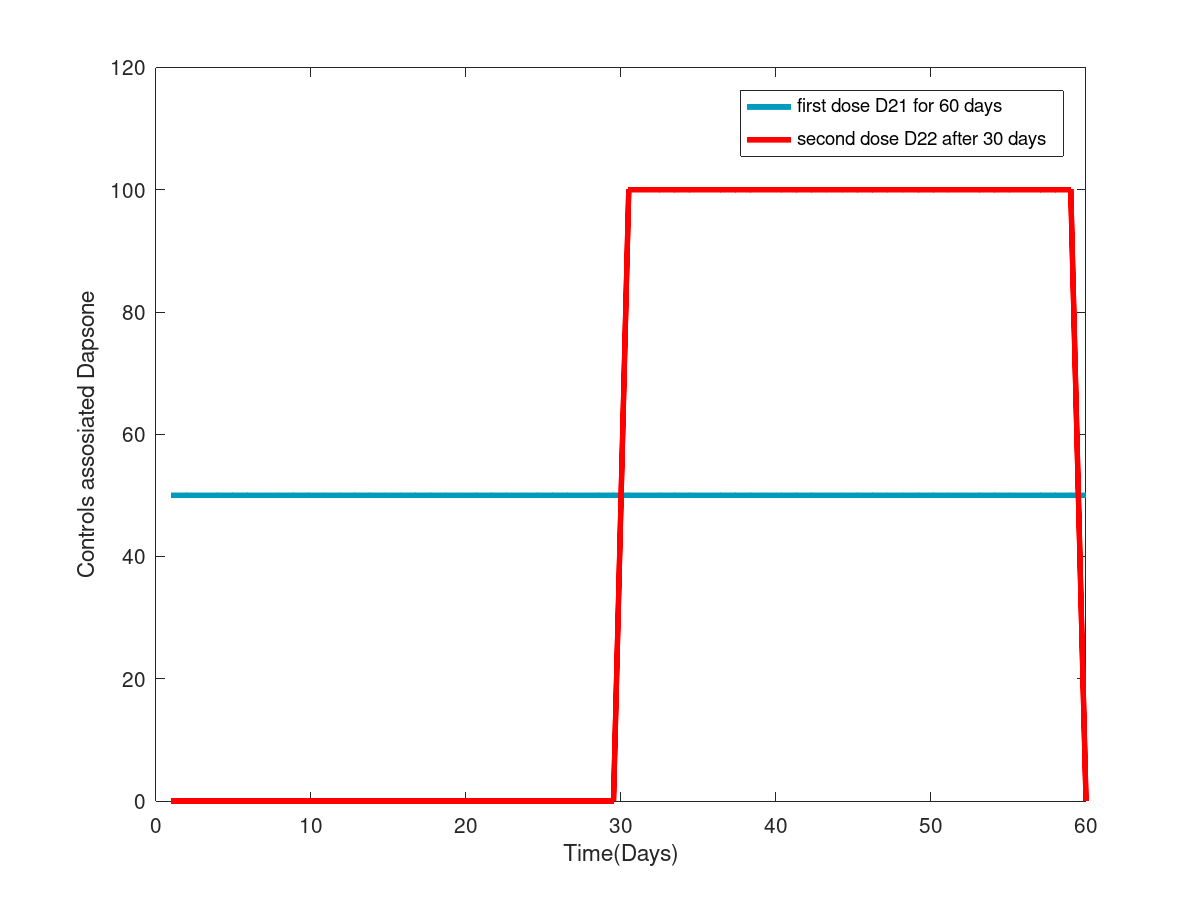}
        \caption{Graph 11}
        \label{fig:graph149}
    \end{subfigure}
    \hfill
    \begin{subfigure}{0.30\textwidth}
        \includegraphics[width=\textwidth]{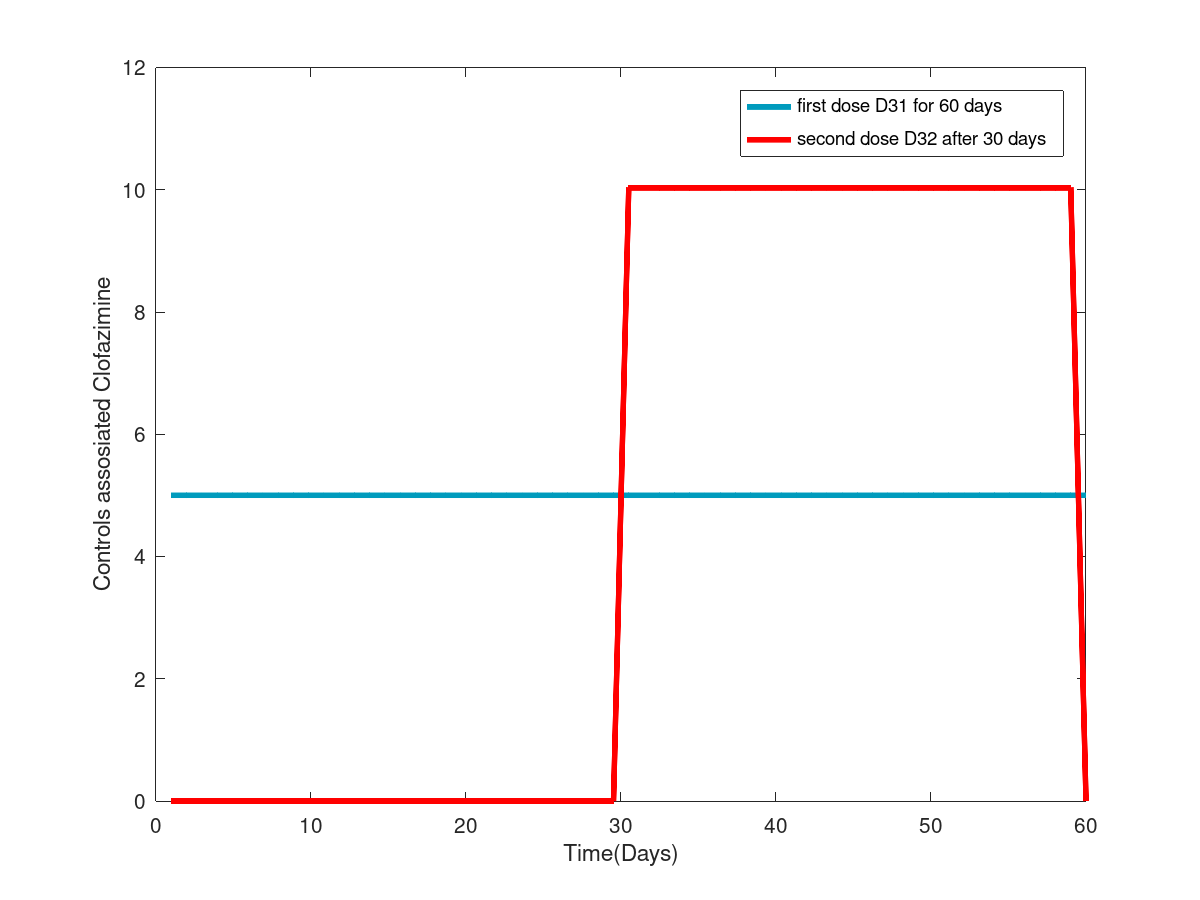}
        \caption{Graph 12}
        \label{fig:graph150}
    \end{subfigure}
    \hfill
     \caption{Plots depicting the combined influence of two dosages of MDT drugs rifampin, clofazimine and dapsone over a period of 60 days with the second dose being administered on 31st day. }
       \label{fig:mdt_dly}
\end{figure}

\newpage

\begin{table}[htbp]
   \centering
    \begin{tabular}{|c|c|c|c|c|}
    \hline
        \textbf{compartments} & \textbf{without drugs } & \textbf{with rifampin } & \textbf{with dapsone } & \textbf{with clofazimine }\\  
        \hline 
$S(t)$ & 519.999174   & 516.754182 & 465.067903 & 508.975006 \\
\hline
$I(t)$ & 249.999158  & 249.830213 & 249.813952 & 249.827797 \\
\hline
$B(t)$ & 2499.956501  & 2491.320879 & 2491.320869 & 2491.320878 \\
\hline
$I_{\gamma}(t)$ & 49.976625  &  45.368405 & 45.368540 & 45.368425 \\
\hline
$T_{\alpha}(t)$ & 49.999836 & 49.948694  & 49.948685 & 49.948693 \\
\hline
$I_{10}(t)$ & 74.968042 & 68.999333  &  68.999308 & 68.999329 \\
\hline
$I_{12}(t)$ & 124.997136 & 124.380295  & 124.380272 & 124.380292\\
\hline
$I_{15}(t)$ &  125.001285 & 125.138320 & 125.138267 & 125.138313 \\
\hline
$I_{17}(t)$ & 100.003874 & 100.632488 & 100.632426 & 100.632479\\
\hline
\end{tabular}
\caption{Average compartments values on administration of two dosages of  rifampin, dapsone, clofazimine over 60-days period with second dosage administered  on  31st day.}
\label{tab:avg1_dly}
\end{table}

\begin{table}[htbp]
   \centering
    \begin{tabular}{|c|c|c|c|c|}
    \hline
        \textbf{compartments} & \textbf{without drugs } & \textbf{with rifampin } & \textbf{with dapsone } & \textbf{with clofazimine }\\  
        \hline
$S(t)$ & 519.998376   & 513.676795 & 416.540017 & 498.630305 \\
\hline
$I(t)$ & 249.998344  & 249.664919 & 249.620094  & 249.658183 \\
\hline
$B(t)$ & 2499.914452  & 2482.950034 & 2482.949997 & 2482.950028 \\
\hline
$I_{\gamma}(t)$ & 49.954030  &  40.931982  & 40.932477 & 40.932056  \\
\hline
$T_{\alpha}(t)$ & 49.999677 & 49.882312  & 49.882282 & 49.882308 \\
\hline
$I_{10}(t)$ & 74.937158 & 63.540457  &  63.540364 & 63.540443 \\
\hline
$I_{12}(t)$ & 124.994367 & 123.738840  & 123.738757 & 123.738828 \\
\hline
$I_{15}(t)$ &  125.002524 & 125.164927 & 125.164738 & 125.164899 \\
\hline
$I_{17}(t)$ & 100.007615 & 101.115553 & 101.115334 & 101.115521 \\
\hline
\end{tabular}
\caption{60th-day compartments values on administration of two dosages of  rifampin, dapsone, clofazimine over 60-days period with second dosage administered  on  31st day.}
\label{tab:last1_dly}
\end{table}

\begin{table}[htbp]
   \centering
    \begin{tabular}{|c|c|c|c|c|}
    \hline
        \textbf{compartments} & \textbf{without drugs} & \textbf{rifampin, dapsone } & \textbf{dapsone, clofazamine } & \textbf{clofazimine, rifampin } \\  
        \hline
$S(t)$ & 519.999174   & 439.854184 & 370.958732  & 483.892621 \\
\hline
$I(t)$ & 249.999158  & 249.805823 & 249.782830 & 249.819934 \\
\hline
$B(t)$ & 2499.956501  & 2491.320864 & 2491.320850 & 2491.320873 \\ 
\hline
$I_{\gamma}(t)$ & 49.976625  &  45.368608 & 45.368805 & 45.368490 \\
\hline
$T_{\alpha}(t)$ & 49.999836 & 49.948681  & 49.948669 & 49.948689 \\
\hline
$I_{10}(t)$ & 74.968042 & 68.999295  &  68.999257 & 68.999317 \\
\hline
$I_{12}(t)$ & 124.997136 & 124.380260  & 124.380225 & 124.380280 \\
\hline
$I_{15}(t)$ &  125.001285 & 125.138240 & 125.138162 & 125.138287 \\
\hline
$I_{17}(t)$ & 100.003874 & 100.632395 & 100.632304 & 100.632449 \\
\hline
\end{tabular}
\caption{Average compartments values on administration of two dosages of  rifampin \& dapsone, dapsone \& clofazamine and clofazamine \& rifampin over 60-days period with second dosage administered  on  31st day.}
\label{tab:avg2_dly}
\end{table}

\begin{table}[htbp]
   \centering
    \begin{tabular}{|c|c|c|c|c|}
    \hline
        \textbf{compartments} & \textbf{without drugs } & \textbf{rifampin, dapsone } & \textbf{dapsone, clofazamine } & \textbf{clofazimine, rifampin } \\    
        \hline
$S(t)$ & 519.998376   & 371.691390 & 258.561537 & 451.132838 \\
\hline
$I(t)$ & 249.998344  & 249.598166 & 249.538020  & 249.636439 \\
\hline
$B(t)$ & 2499.914452  & 2482.949978 & 2482.949926 & 2482.950010 \\
\hline
$I_{\gamma}(t)$ & 49.954030  &  40.932724  & 40.933426 & 40.932295  \\
\hline
$T_{\alpha}(t)$ & 49.999677 & 49.882267  & 49.882224 & 49.882293 \\
\hline
$I_{10}(t)$ & 74.937158 & 63.540317  &  63.540186 & 63.540398 \\
\hline
$I_{12}(t)$ & 124.994367 & 123.738715  & 123.738597 & 123.738787 \\
\hline
$I_{15}(t)$ &  125.002524 & 125.164643 & 125.164375 & 125.164808 \\
\hline
$I_{17}(t)$ & 100.007615 & 101.115224 & 101.114913 & 101.115415 \\
\hline
\end{tabular}
\caption{60th-day compartments values on administration of two dosages of  rifampin \& dapsone, dapsone \& clofazamine and clofazamine \& rifampin over 60-days period with second dosage administered  on  31st day.}
\label{tab:last2_dly}
\end{table}

\begin{table}[htbp]
    \centering
    \begin{tabular}{|c|c|c|c|c|}
        \hline
        \multirow{2}{*}{\textbf{compartments}} & \multicolumn{2}{c|}{\textbf{without drugs}} & \multicolumn{2}{c|}{\textbf{rifampin, dapsone and clofazimine }} \\
        \cline{2-5}
         &  Average & 60th day & Average & 60th day \\
        \hline   
        $S(t)$ & 519.999174   & 519.998376  & 338.956016 & 211.188822 \\ 
        \hline
       $I(t)$ & 249.999158  &  249.998344 & 249.771687 & 249.509958 \\ 
       \hline
       $B(t)$ & 2499.956501  & 2499.914452 & 2491.320842 & 2482.949901 \\
       \hline
       $I_{\gamma}(t)$ & 49.976625  &  49.954030 & 45.368903 & 40.933768\\ 
       \hline
       $T_{\alpha}(t)$ & 49.999836 & 49.999677 & 49.948662 & 49.882203 \\
       \hline
       $I_{10}(t)$ & 74.968042 & 74.937158 & 68.999239  & 63.540122 \\
       \hline
       $I_{12}(t)$ & 124.997136 & 124.994367 & 124.380208  & 123.738539 \\ 
       \hline
       $I_{15}(t)$ & 125.001285 & 125.002524 & 125.138123 & 125.164245  \\
       \hline
       $I_{17}(t)$ & 100.003874 & 100.007615 & 100.632259 & 101.114762 \\
       \hline
    \end{tabular}
    \caption{ Average and 60th day compartments values on administration of two dosages of  all MDT drug over 60-days period with second dosage administered  on  31st day.}
    \label{tab:Avg&last3_dly}
\end{table}

\newpage

\section{Discussions and Conclusions} \label{sec6}

The present work is novel  and first of its kind dealing with  the dynamics dealing with the levels of crucial bio-markers that are involved in Type 1 lepra reaction and their quantitative correlations with the MDT drugs along with the optimal dosages. \\

 We have explored these correlations for administration of  drugs in two  dosages for the MDT drugs namely rifampin, clofazimine \& dapsone with respect to individual and combined implementation. These scenarios have been numerically simulated and the findings have been extensively discussed. These study also explored the optimal drug dosages for administration and it was found out  that the optimal
drug dosage of the MDT drugs found through these optimal control studies and the dosage prescribed as per WHO guidelines are almost the same. \\

 In conclusion we suggest that  the present research work can  be extrapolated to  real-administration scenario based on the WHO 2018 guidelines for Multi Drug therapy (MDT) consisting of drugs rifampin, dapsone and clofazimine with certain dosage   administered every 30 days over a period of 12 months for the treatment of lepra type I and type II reactions. The duration may vary based on whether it's paucibacillary leprosy (6 months) or multibacillary leprosy (12 months).   \\

 Also this study can be of important help to the clinician in early detection of the leprosy and avoid and control the disease from going to Lepra reactions and help in averting major damages.


\section*{Funding}
This research was supported by Council of Scientific and Industrial Research (CSIR) under project grant -{\bf{ Role and Interactions of Biological Markers in Causation of Type1/Type 2 Lepra Reactions: A In Vivo Mathematical Modelling with Clinical Validation (Sanction Letter No. 25(0317)/20/EMR-II).}} 

\section*{Data Availability Statement (DAS)}

We do not analyse or generate any datasets, because our work proceeds within a theoretical and mathematical approach.

\section*{Declarations}
The authors declare no Conflict of Interest for this research work.

\section*{Ethics Statement} This research did not require any ethical approval.

\section*{Acknowledgments}
The authors dedicate this paper to the founder chancellor of SSSIHL, Bhagawan Sri Sathya Sai Baba. The corresponding author also dedicates this paper to his loving elder brother D. A. C. Prakash who still lives in his heart. 

\printbibliography

\end{document}